\documentclass{amsart}
\usepackage[top=3cm, bottom=3cm, left=2cm, right=2cm]{geometry}
\usepackage{graphicx} 
\usepackage[utf8]{inputenc}
\usepackage{amsmath, amssymb, mathrsfs}
\usepackage{bbm}
\usepackage{mathtools}
\usepackage{csquotes}
\usepackage{tikz-cd}
\usepackage{cite}
\usepackage{amsthm}
\usepackage{thmtools}
\usepackage[all]{xy}
\usepackage{ stmaryrd }

\usepackage{tikz}
\usetikzlibrary{decorations.markings,positioning,calc}
\usetikzlibrary{decorations.pathmorphing} 

\usetikzlibrary{babel}
\usepackage{longtable}
\usepackage{stmaryrd}
\usepackage{comment}
\usepackage{float}
\newfloat{myquote}{t}{}

\usepackage{imakeidx}
\usepackage[initsep=0pt]{idxlayout}
\makeindex
\usepackage{hyperref}
\usepackage{cleveref}

\hypersetup{
	colorlinks=true,
	linkcolor=blue,
	filecolor=magenta,      
	urlcolor=cyan,
}

\newcommand{\bgm}{B\mathbb{G}_m}

\newcommand{\bA}{\mathbb{A}}

\newcommand{\bG}{\mathbb{G}}

\newcommand{\bP}{\mathbb{P}}

\newcommand{\oH}{\operatorname{H}}

\newcommand{\cA}{\mathcal{A}}
\newcommand{\cB}{\mathcal{B}}
\newcommand{\cC}{\mathcal{C}}
\newcommand{\cD}{\mathcal{D}}
\newcommand{\cE}{\mathcal{E}}

\newcommand{\cH}{\mathcal{H}}
\newcommand{\cI}{\mathcal{I}}

\newcommand{\cM}{\mathcal{M}}

\newcommand{\cO}{\mathcal{O}}
\newcommand{\cP}{\mathcal{P}}

\newcommand{\cU}{\mathcal{U}}

\newcommand{\cX}{\mathcal{X}}

\newcommand{\cZ}{\mathcal{Z}}

\newcommand{\car}{\operatorname{char }}

\def\Spec{\operatorname{Spec}}

\newcommand{\spec}{\operatorname{Spec}}

\newcommand{\Sym}{\operatorname{Sym}}
\newcommand{\Hom}{\operatorname{Hom}}

\newcommand{\Ga}{\bG_{{\rm a}}}
\newcommand{\GL}{\operatorname{GL}}

\newcommand{\Aut}{\operatorname{Aut}}
\newcommand{\CAut}{\operatorname{Aut^{\circ}}}
\newcommand{\PGL}{\operatorname{PGL}}
\newcommand{\im}{\operatorname{im}}

\newcommand{\Ext}{\rm Ext}

\newcommand{\Def}{{\rm Def}}
\newcommand{\cHom}{\mathcal{H}\kern -.5pt om}
\newcommand{\cExt}{\mathcal{E}\kern -.5pt xt}

\newcommand{\ST}{\overline{\operatorname{ST}}}
\newcommand{\Gm}{\mathbb{G}_m}

\newcommand{\Map}{{\rm Map}}

\newtheoremstyle{thmcite}
  {}
  {}
  {\itshape}
  {}
  {\bfseries}
  {.}
  { }
  {\thmname{#1}\thmnumber{ #2}\thmnote{ \normalfont\checkcite{#3}}}

\newtheoremstyle{defcite}
  {} 
  {} 
  {\normalfont} 
  {} 
  {\bfseries} 
  {.} 
  { } 
  {\thmname{#1}\thmnumber{ #2}\thmnote{ \normalfont\checkcite{#3}}}

\ExplSyntaxOn
\NewDocumentCommand{\checkcite}{m}
 {
  \regex_if_match:nnTF { \A \cB\{* \c{cite} [^\c{cite}]* \Z } {#1}
   {
    #1
   }
   {
    (#1)
   }
 }
\ExplSyntaxOff

\theoremstyle{thmcite}
\newtheorem{theorem}{Theorem}[section]
\newtheorem*{theorem*}{Theorem}

\newtheorem{lemma}[theorem]{Lemma}
\newtheorem{corollary}[theorem]{Corollary}
\newtheorem{proposition}[theorem]{Proposition}
\newtheorem{theoremalpha}{Theorem}

\theoremstyle{defcite}
\newtheorem{remark}[theorem]{Remark}
\newtheorem{example}[theorem]{Example}
\newtheorem{definition}[theorem]{Definition}

\theoremstyle{definition}

\let\oldtocsection=\tocsection

\let\oldtocsubsection=\tocsubsection

\let\oldtocsubsubsection=\tocsubsubsection

\renewcommand{\tocsection}[2]{\hspace{0em}\oldtocsection{#1}{#2}}
\renewcommand{\tocsubsection}[2]{\hspace{1em}\oldtocsubsection{#1}{#2}}
\renewcommand{\tocsubsubsection}[2]{\hspace{2em}\oldtocsubsubsection{#1}{#2}}

\DeclareRobustCommand{\SkipTocEntry}[5]{}

\title{The local geometry of the stack of $A_r$-stable curves}



\author{Davide Gori}
\address{Universität Duisburg-Essen, Essen, Germany}
\email{davide.gori@uni-due.de}
\author{Ludvig Modin}
\address{Leibniz Universität Hannover, Institut für Algebraische Geometrie, Welfengarten 1, 30167 Hannover}
\email{modin@math.uni-hannover.de}
\author{Michele Pernice}
\address{Department of Mathematics, University of Washington, 4110 E Stevens Way NE, Seattle, WA 98195}
\email{mpernice@uw.edu}
\begin{document}

\begin{abstract}
    In this paper we study the local geometry of the stack of pointed $A_r$-stable curves. In particular, we analyze the deformation theory of $A_r$-stable curves and their automorphism groups in order to study the combinatorics of families of curves over $[\mathbb{A}^1/\mathbb{G}_m]$, and use this to classify all closed points of the stack of \(A_r\)-stable curves. As a byproduct, we also classify all open substacks of the moduli stack of degree \(2\) cyclic covers of $\bP^1$ that admit a separated good moduli space. This is the first in a series of three papers aimed at studying obstructions for the existence of good moduli spaces for stacks of curves with $A$-type singularities, and using these to find an open substack of the stack of \(A_r\)-stable curves that admits a proper non-projective good moduli space when \(r=5\).
\end{abstract}

\maketitle
\tableofcontents

\section{Introduction}

The moduli space of smooth curves is one of the most studied moduli spaces in algebraic geometry, going back to Riemann. In \cite{DeligneMumford}, Deligne and Mumford constructed the first modular compactification using algebraic methods. In later years, new modular compactifications appeared in the literature. Among the most prominent examples are the Schubert compactification \cite{PseudostableSchubert} and those arising in the Hassett--Keel program, which investigates the modularity of the log-Minimal Model Program (log-MMP) of $\overline{\operatorname{M}}_{g,n}$ (see, for instance, \cite{hassett2009log, hassett2013log, AlpFedSmyWyck, AlpFedSmyExistence, AlpFedSmyProjectivity}). All of these involve curves with $A$-type singularities. Further interplay between the log-MMP of $\overline{\operatorname{M}}_{g,n}$ and alternative compactifications can be found in \cite{Viviani, codogni2021some, AltCompClustAlg}.

The first compactifications were constructed using Geometric Invariant Theory (GIT), via an analysis of stability in the Hilbert and Chow schemes. The series of papers \cite{AlpFedSmyWyck, AlpFedSmyExistence, AlpFedSmyProjectivity} introduces a more intrinsic approach, using what is called local Variation of GIT (VGIT). Later compactifications constructed in \cite{Viviani} use the same technique.

This paper is the first in a series dedicated to the problem of constructing modular compactifications using intrinsic techniques within the framework of the ``Beyond GIT'' program, introduced in \cite{StructureOfInstability}.
Namely, we aim to use the criteria introduced in \cite{ExistenceOfModuli} and verify $\Theta$- and $\textsf{S}$-completeness to construct new modular compactifications that admit a proper good moduli space that is not projective. This modular compactification will be contained inside the moduli stack $\cM_{g,n}^r$ parametrizing $A_r$-stable curves, i.e., curves with at worst $A_r$-singularities and whose log-dualizing line bundle is ample. We work over a field of characteristic $0$; the results generalize to \emph{characteristic big enough with respect to $(g,n)$}, similar to what is done in \cite[Definition~4.1.1]{AltCompClustAlg}.

This paper aims to provide a \emph{combinatorial classification of isotrivial degenerations} of $A_r$-stable curves, or, equivalently, morphisms $\Theta \coloneqq [\bA^1/\bG_m] \rightarrow \cM_{g,n}^r$. This is achieved by:
\begin{itemize}
    \item[1)] classifying all the curves with positive-dimensional stabilizers (see \Cref{cor:gm-decom-aut});
    \item[2)] studying their equivariant deformation theory, which allows us to control how the combinatorics of the curve changes along $\Gm$-equivariant deformations (see \Cref{prop:defor-rational-chain} and \Cref{prop:iso-deg-rosary});
    \item[3)] globalizing the classification, which allows us to determine whether an isotrivial degeneration between two curves exists or not, based only on the combinatorics of the curves (see the results in \Cref{sub:describing-attractors}).
\end{itemize}

An important byproduct of the study of isotrivial degenerations is \Cref{thm:closed_points_Mgnr}, where we describe the closed points of $\cM_{g,n}^r$. As a corollary, we obtain the following result.

\begin{theoremalpha}[{\Cref{cor:stable-gms}}]\label{theo:A}
    Let \(\cM_{g,n}\) be the stack of smooth curves and \(\cM_{g,n} \to \operatorname{M}_{g,n}\) its coarse moduli space.
    Assume \(r<2g\) or \(n>2\), and that \(\cM_{g,n}\subset \mathcal{U}\subset \cM_{g,n}^r\) is a sequence of open embeddings such that \(\cU\) admits a good moduli space \(\pi:\cU\to \operatorname{U}\). Then \(\cM_{g,n}\subset \cU\) is saturated with respect to \(\pi\); in particular, the induced map \(\operatorname{M}_{g,n}\to \operatorname{U}\) is an open embedding.
\end{theoremalpha}

In other words, if \(g\) is sufficiently large relative to \(r\), then the good moduli space of our modular compactification descends to a compactification of the coarse moduli space of smooth curves. Essentially, the restriction on \(r\) comes from smooth hyperelliptic curves admitting nontrivial isotrivial degenerations to curves with one \(A_{2g}\)-singularity (respectively, one \(A_{2g+1}\)-singularity).

The combinatorial classification of isotrivial degenerations will also be crucial for dealing with $\Theta$-completeness and $\textsf{S}$-completeness in future papers in this series, and it relies heavily on the strong connection between $A$-singularities and (reduced) $2\!:\!1$ covers of the projective line, which we call honestly hyperelliptic curves (following \cite{Cat}). More precisely, given a generically smooth family of curves over a discrete valuation ring (dvr) whose special fiber contains an $A$-singularity, we can locally modify the family by replacing this singularity with a hyperelliptic curve (up to a dvr extension). This process is called semistabilization and can be done explicitly for $ADE$ singularities; see, for instance, \cite{Fedor} and \cite{CasmarLaza}. In the case of an $A_r$-singularity, the semistabilization procedure can be described depending on the parity of $r$:

\begin{itemize}
    \item[(e)] if $r$ is even, then a blowup of a smoothing family of the singularity replaces the singularity with an honestly hyperelliptic curve of genus $r/2$ (with at worst $A_{r-1}$-singularities) attached to the rest of the curve at a Weierstrass point;
    \item[(o)] if $r$ is odd, then an appropriate blowup (described in loc.~cit.) of a smoothing family of the singularity replaces the singularity with an honestly hyperelliptic curve of genus $(r-1)/2$ (with at worst $A_{r-1}$-singularities) attached to the rest of the curve at two points exchanged by the hyperelliptic involution.
\end{itemize}
The semistable replacement can be obtained by repeating the process until only nodal singularities remain. This strong connection is also reflected in the deformation space of $A_r$-singularities with a nontrivial $\Gm$-action, which we call \emph{hyperelliptic} $A$-singularities (see \Cref{def:hyp-even-sing}). Informally, the $\Gm$-action induces a $\Gm$-equivariant decomposition of the deformation space of the singularity: the positively weighted isotypic component corresponds to deformations that preserve the singularity but change the crimping datum, whereas the negatively weighted isotypic component corresponds to deformations that deform the singularity but produce an honestly hyperelliptic component as a result (similar to how it occurs in the semistabilization procedure described above). For the precise statement, see \Cref{cor:gm-decom-tang}.

Thus, a key step in our study of general isotrivial degenerations is to analyze the moduli stack $\cH_{g}^{2g+1,\circ}$ of honestly hyperelliptic curves of genus $g$. This moduli stack is directly related to the moduli stack parametrizing a divisor of degree $2g+2$; namely, there exists a morphism
\[
\cH_{g}^{2g+1,\circ} \to \left[ \bP(\Sym^{2g+2}E)/\PGL_2 \right]
\]
which is a $\mu_2$-gerbe. For more details, see \Cref{sub:hyper-A-stable}.
As a byproduct, we classify all open substacks of $\cH_g^{2g+1,\circ}$ that admit a separated good moduli space. Let $\cH_{g}^{r,\circ}$ denote the open substack of \(\cH_{g}^{2g+1,\circ}\) parametrizing honestly hyperelliptic curves that are $A_r$-stable. We obtain the following characterization:

\begin{theoremalpha}[{\Cref{theo:classify-opens-hyp}}]
Let $\cU \subset \cH_{g}^{2g+1,\circ}$ be any open substack. Then $\cU$ admits a separated good moduli space if and only if one of the following holds:
\begin{itemize}
    \item[(i)] $\cU \subset \cH_g^{g-1,\circ}$;
    \item[(ii)] $\cU \subset \cH_{g}^{g,\circ}$ and $\cH_{g}^{g,\circ}\setminus \cH_g^{g-1, \circ} \subset \cU$.
\end{itemize}
\end{theoremalpha}

The techniques used in the proof also allow us to answer negatively a question posed in \cite[Remark~3.11]{codogni2021some} (see \Cref{prop:Viviani-answer}). This question can alternatively be resolved by a direct inspection of dual graphs, as done in \cite[Remark~3.8.8]{AltCompClustAlg}. The classification obtained in \Cref{theo:classify-opens-hyp} shows that all the open substacks of $\cH_{g}^{2g+1,\circ}$ with separated good moduli spaces can be obtained as $\mu_2$-gerbes over some open substack of the GIT-semistable locus in \(\left[ \bP(\Sym^{2g+2}E)/\PGL_2 \right]\). In contrast, the open substacks with separated good moduli spaces in \(\cM_{g,n}^r\) cannot all be obtained by GIT methods, as we will show in upcoming papers in this series.

\addtocontents{toc}{\SkipTocEntry}
\subsection*{Outline of the paper}

In \Cref{sec:linearization}, we recall basic notions concerning good moduli spaces, with the goal of introducing a criterion useful for establishing $\textsf{S}$- and $\Theta$-completeness for open embeddings (see \Cref{prop:affine-lin-GIT}).

The first part of \Cref{sec:notation} collects general results about $A_r$-stable curves and their moduli stack (many of which are recalled from \cite{Per1} and \cite{VanDerWyck}). The second part of the section is devoted to understanding the link between $A$-singularities and $2\!:\!1$ covers of the projective line. Finally, we prove \Cref{theo:classify-opens-hyp}.

In \Cref{sec:pos-dim-stab}, we study families of $A_r$-stable curves over $\Theta$ using a three-step procedure. First, \Cref{prop:stab-no-out} and \Cref{cor:gm-decom-aut} characterize curves with positive-dimensional stabilizers (the possible central fibers). Next, the study of equivariant deformation spaces (see, for instance, \Cref{prop:def-hyp-rat-sing}) yields combinatorial constraints on the generic fiber (see \Cref{rem:alternating-deformations}). In \Cref{sub:describing-attractors}, we prove that these conditions are not only necessary but also sufficient. Finally, in \Cref{thm:closed_points_Mgnr}, we classify the closed points of $\cM_{g,n}^r$ and deduce \Cref{theo:A}.

\addtocontents{toc}{\SkipTocEntry}
\subsection*{Acknowledgments}

The authors are tremendously indebted to Jarod Alper, Luca Battistella, Andrea Di Lorenzo, Andres Fernandez Herrero, Jochen Heinloth, Giovanni Inchiostro, David Rydh, Filippo Viviani, and Dario Wei{\ss}mann for helpful discussions along the way. We further want to thank Jochen Heinloth for the opportunity to work in person for a week on the project during a research visit of the first and third authors in Essen. The first-named author was partially supported by the INdAM group GNSAGA. The second-named author was supported by DFG-Research Training Group 2553 at Universität Duisburg-Essen during the first phase of this project and is currently a member of the Institut für Algebraische Geometrie in Leibniz Universität Hannover. The third-named author was supported by the Knut and Alice Wallenberg Foundation 2021.0291.
\section{Good moduli spaces and linearization}\label{sec:linearization}
In this section, we collect some results about how one can use local quotient descriptions of a stack \(\mathcal{X}\) to study maps \(\Theta_k,\Theta_R,\ST_R\to \cX\). In this paper we use the maps to reduce the study of isotrivial degenerations to the \(\mathbb{G}_m\)-equivariant geometry of the deformation spaces of \(A_r\)-stable curves. In the next papers of the series, \Cref{prop:local-Theta-complete} will play a key role in the proof of \(\Theta\)- and \textsf{S}-completeness of our choice of open substack of \(\cM_{g,n}^r\).

\subsection{Local structure and good moduli spaces}\label{sub:gms}

Let us start by recalling the notion of good moduli spaces and the \'etale local structure for algebraic stacks.

Good moduli spaces were introduced in \cite{Alp} as a characterization of the key properties of the morphism \([\Spec A/G]\to \Spec A/\!/G\), where \(G\) is a linearly reductive group acting on \(\Spec A\) and \(\Spec A/\!/G\) is the quotient constructed by geometric invariant theory (i.e. the spectrum of the \(G\)-invariants of \(A\)). The definition that captures these properties is surprisingly sparse.
\begin{definition}
    A quasi-compact quasi-separated morphism \(\pi:\mathcal{X}\to X\) to an algebraic space \(X\) is a good moduli space if the following holds.
    \begin{enumerate}
        \item the canonical morphism \(\mathcal{O}_X\to \pi_*\mathcal{O}_{\mathcal{X}}\) is an isomorphism,
        \item the functor \(\pi_*:\operatorname{QCoh}(\mathcal{X})\to \operatorname{QCoh}(X)\)
        is exact.
    \end{enumerate}
\end{definition}
Other than geometric invariant theory quotients, a key example to keep in mind is that a classifying stack \(BG\) admits a good moduli space if and only if \(G\) is a linearly reductive group scheme. Furthermore, if \(\mathcal{X}\) admits a good moduli space then so does every closed substack of \(\mathcal{X}\), which has the consequence that the automorphism group schemes of the closed points of \(\mathcal{X}\) are linearly reductive.

When a stack is not constructed as a quotient of a quasi-projective scheme by a reductive action, one cannot use the methods of geometric invariant theory directly to study whether it admits a good moduli space: there are however strong local structure results that allow one to do this \'etale locally. We state the first version of the theorem that appeared, which is for stacks locally of finite type over an algebraically closed field; there are also versions over general bases with weaker assumptions than here, namely \cite[Theorem~A]{AlpHalRyd} and \cite[Theorems~1.3,1.4,1.5 and 1.6]{AlpHalHalLeiRyd}.

\begin{theorem}[\cite{LunaetaleSliceStacks}]
Let \(\mathcal{X}\) be an algebraic stack, locally of finite type and quasi-separated over an algebraically closed field \(\kappa\), with affine stabilizers. Let \(x\in \mathcal{X}(\kappa)\) be a point and \(H\subset G_x\) a subgroup of the stabilizer of \(x\) such that \(H\) is linearly reductive and \(G_x/H\) is \'etale. Then there exists an affine scheme \(\Spec A\) with an action of \(H\), a \(\kappa\)-point \(w\in \Spec A\) fixed by \(H\) and an \'etale morphism of pointed stacks
\[f:([\Spec A /H],w)\to (\mathcal{X},x)\]
such that \(f^{-1}(\cB G_x)=\cB H\), in particular \(f\) induces the given inclusion \(H\subset G_x\) on stabilizer group schemes at \(w\). Moreover, if \(\cX\) has affine diagonal, then \(f\) can be arranged to be affine.
\end{theorem}

In particular, this implies that every point of a stack with linearly reductive stabilizer has an \'etale neighborhood that admits a good moduli space, namely \([\Spec A/G_x]\to \Spec A^{G_x}\). With this in hand, one would like to glue these local good moduli spaces to a global one: for instance, the linearization of an action on a quasi-projective variety plays this gluing role in geometric invariant theory. In the general case, it turns out that this boils down to two stacky valuative criteria, which also guarantee that the resulting good moduli space will be separated.

To introduce these, we need two stacks associated to any dvr \(R\) together with associated punctured versions (playing the role of the spectrum of the fraction field in classical valuative criteria). To introduce these stacks, let \((m,n)\in\{(1,0),(1,-1)\} \) and let \(\mathbb{G}_m \) act on \(\mathbb{A}^2\) linearly with weights \((m,n)\). Then the map \([\mathbb{A}^2/\mathbb{G}_m]\to \mathbb{A}^1\) induced respectively by \(t\mapsto y\) and \(t\mapsto xy\) on the coordinate rings \(k[t]\) and \(k[x,y]\) of \(\mathbb{A}^1\) and \(\mathbb{A}^2\) are good moduli spaces. For any dvr \(R\) with uniformizer \(\pi\), we define the stacks \(\Theta_R\) and \(\ST_R\) by the cartesian diagrams
\[\begin{tikzcd}
	{\Theta_R} & {[\mathbb{A}^2/\mathbb{G}_m]} \\
	{\Spec R} & {\mathbb{A}^1}
	\arrow[from=1-1, to=1-2]
	\arrow[from=1-1, to=2-1]
	\arrow["{y\mapsfrom t}", from=1-2, to=2-2]
	\arrow["{\pi\mapsfrom t}"', from=2-1, to=2-2]
\end{tikzcd}\]
with \((m,n)=(1,0)\) and
\[\begin{tikzcd}
	{\ST_R} & {[\mathbb{A}^2/\mathbb{G}_m]} \\
	{\Spec R} & {\mathbb{A}^1}
	\arrow[from=1-1, to=1-2]
	\arrow[from=1-1, to=2-1]
	\arrow["{xy\mapsfrom t}", from=1-2, to=2-2]
	\arrow["{\pi\mapsfrom t}"', from=2-1, to=2-2]
\end{tikzcd}\]
with \((m,n)=(1,-1)\). We also respectively denote by \(\Theta_R\setminus 0\) and \(\operatorname{ST}_R\) the stacks obtained for the analogous cartesian diagrams with \([\mathbb{A}^2/\mathbb{G}_m]\) replaced with \([(\mathbb{A}^2\setminus \{0\})/\mathbb{G}_m]\), inducing open immersions with codimension \(2\) complements \(\iota:\Theta_R\setminus 0\subset \Theta\) and \(\iota:\operatorname{ST}_R\subset \ST_R\).

\begin{definition}[\cite{ExistenceOfModuli}]\label{Theta- and S-completeness}
    Let \(f:\mathcal{X}\to \mathcal{Y}\) be a morphism of locally noetherian algebraic stacks. We say that \(f\) is \(\Theta\)-complete if for all dvr's \(R\) and all commutative solid diagrams
    \[\begin{tikzcd}
	{\Theta_R\setminus 0} & {\mathcal{X}} \\
	{\Theta_R} & {\mathcal{Y}}
	\arrow[from=1-1, to=1-2]
	\arrow["\iota"',from=1-1, to=2-1]
	\arrow["f", from=1-2, to=2-2]
	\arrow[dotted, from=2-1, to=1-2]
	\arrow[from=2-1, to=2-2]
\end{tikzcd}\]
there is a unique dotted arrow making the diagram commute. We say that \(f\) is \textsf{S}-complete if for all dvr's \(R\) and all commutative solid diagrams
    \[\begin{tikzcd}
	{\operatorname{ST}_R\setminus 0} & {\mathcal{X}} \\
	{\ST_R} & {\mathcal{Y}}
	\arrow[from=1-1, to=1-2]
	\arrow["\iota"',from=1-1, to=2-1]
	\arrow["f", from=1-2, to=2-2]
	\arrow[dotted, from=2-1, to=1-2]
	\arrow[from=2-1, to=2-2]
\end{tikzcd}\]
there is a unique dotted arrow making the diagram commute. If \(\mathcal{X}\to \Spec \kappa\), the structure morphism of \(\mathcal{X}\) is \(\Theta\)- or \textsf{S}-complete, then we say that \(\mathcal{X}\) is  respectively $\Theta$-complete or \textsf{S}-complete.
\end{definition}

When first encountering these valuative criteria, it is a good idea to keep the example of \(f:\mathcal{X}\to \mathcal{Y}\) being a morphism of schemes in mind: in this case \(S\)-completeness amounts to a reformulation of the valuative criterion of separatedness, which can be seen by noting that \(\operatorname{ST}_R\cong \Spec R\cup\Spec R\) glued along the generic point and the good moduli space of \(\ST_R\) is \(\Spec R\).

We can now state the existence theorem for separated good moduli spaces, \cite[Theorem~A]{ExistenceOfModuli}.

\begin{theorem}[\cite{ExistenceOfModuli}]
    Let \(\mathcal{X}\) be a stack of finite type  with affine stabilizers and separated diagonal over \(\Spec \kappa\). Then \(\mathcal{X}\) admits a separated good moduli space \(X\) if and only if it is \(\Theta\)- and \textsf{S}-complete.
    Moreover, \(X\) is proper if and only if \(\mathcal{X}\) satisfies the existence part of the valuative criterion for properness.
\end{theorem}

Our study of obstructions to open substacks of \(\cM_{g,n}^r\) builds on the following specialization diagrams describing \(\Theta_R\) and \(\ST_R\) respectively:
    $$
    \begin{tikzcd}
     & \Spec k \arrow[ld, dashed]     &                                                                    \\\cB \bG_{m,k} &                                 & \Spec Q\arrow[lu] \arrow[ld, dashed] \\ & \cB\bG_{m,Q} \arrow[lu] &
    \end{tikzcd}
    $$
    and
    $$
    \begin{tikzcd}
     & \Spec k \arrow[ld, dashed]     &                                                                    \\\cB \bG_{m,k} &                                 & \Spec Q\arrow[lu] \arrow[ld] \\ & \Spec k \arrow[lu, dashed] &
    \end{tikzcd}
    $$

where \(Q\) is the fraction field and \(k\) the residue field of \(R\), and the solid arrows represent an inclusion of \(\Spec R\) or \(\cB \bG_{m,R}\) and the dotted arrows and inclusion of \(\Theta_Q\) or \(\Theta_k\). We often call these dotted specialization arrows \emph{isotrivial degenerations}: we can think of them as families of objects over the ``$0$-dimensional dvr'' $\Theta$. We say that an isotrivial degeneration is degenerate if the closed and open topological points of $\Theta$ are sent to the same topological point of $\cX$. The punctured versions are described by removing the \(\cB \bG_{m,k}\) from these diagrams.

As can be seen from these diagrams, a key component for the application of the existence theorem is to understand families over \(\Theta_k\), see \Cref{sec:pos-dim-stab}, in other words we need to understand maps \(\Theta_k\to \cX\). These maps are parametrized by an algebraic stack \(\operatorname{Filt}(\mathcal{X}):=\Map(\Theta, \cX)\), as is shown in \cite[Proposition 1.1.2]{StructureOfInstability}.
Following \cite{hassett2013log} and \cite[Definition 3.5.2]{AltCompClustAlg}, we introduce the definition of \emph{basin of attraction}.
With the purpose to study isotrivial degenerations, we define:
\begin{definition}[Basin of attraction]
    Let $\cX$ be a quasi-compact and quasi-separated algebraic stack with affine stabilizers over a quasi-separated algebraic space $S$. A point $x \in |\Map_S(\bgm, \cX)|$ is given by a point $p \in |\cX|$ and a cocharacter $\rho \colon \Gm \to \Aut_{\cX}(p)$. We have a diagram
    \[
    \begin{tikzcd}
        \Map_S(\Theta, \cX) \dar{ev_1} \rar{ev_0}& \Map_S(\cB\Gm, \cX) \\
        \cX\\
    \end{tikzcd},
    \]
    where $ev_0$ (resp. $ev_1$) is the restriction of a map to the closed substack $B\Gm$ (resp. to the open point). We define the \emph{basin of attraction} of $x$ with respect to $\rho$ as $|ev_1(ev_0^{-1}(x))| \subset |\cX|$.
\end{definition}

The study mentioned above can be done via local structure theorems, and thus we need to study the \(\bG_m\)-equivariant deformation theory of \(A_r\)-stable curves, which in turn amounts to studying certain \(\mathbb{G}_m\)-representations.

\subsection{Affine linear GIT}

In this short subsection, we analyze the problem of finding $G$-equivariant closed subset $Z\subset V$ of a $G$-representation containing $0$ and with the property that the complement of \([Z/G]\) in \([V/G]\) is $\Theta$-complete (respectively is $\textsf{S}$-complete, it has a separated good moduli space).

\begin{definition}\label{def:multi-rep-notation}
    Let $V$ be a $\bG_m$-representation. We denote by $V^>$ (respectively $V^{\geq}$, $V^<$, $V^{\leq}$) the subrepresentation of $V$ defined as the sum of the irreducible subrepresentation of $V$ with positive (respectively non negative, negative and non positive) weights. We denote by $V^0$ the subspace of $\bG_m$-invariants vectors inside $V$.
\end{definition}

The following lemma is a special case of \cite[Theorem~1.4.8]{StructureOfInstability}.
\begin{lemma}\label{lem:Theta-and-BGmR}
    Let $k$ be an algebraically closed field, $G$ be an algebraic group over $k$ and $V$ be a $G$-representation. Suppose we are given a commutative diagram
    $$
    \begin{tikzcd}
    {\cB \bG_{m,k}} \arrow[d, hook] \arrow[r] & \cB G \arrow[d, hook] \\
    \Theta_k \arrow[r]                        & {[V/G]}.       \end{tikzcd}
    $$
    Then the morphism $\Theta_k \rightarrow [V/G]$ factors through the morphism $[V^{>}/\bG_{m}]\rightarrow [V/G]$. Moreover, let $R$ be a dvr with residue field $k$ and suppose we have a commutative diagram
    $$
    \begin{tikzcd}
    {\cB \bG_{m,k}} \arrow[d, hook] \arrow[r] & \cB G \arrow[d,       hook] \\
    {\cB \bG_{m,R}} \arrow[r]                 & {[V/G]}
    \end{tikzcd}
    $$
    Then the morphism $\cB \bG_{m,R} \rightarrow [V/G]$ factors through the morphism $[V^{0}/\bG_{m}]\rightarrow [V/G]$.
\end{lemma}

\begin{proof}
    First, we notice that we have a commutative diagram
    $$
    \begin{tikzcd}
    \Theta_k \arrow[r] \arrow[rd] & {[V/\bG_m]} \arrow[r]   \arrow[d] & {[V/G]} \arrow[d] \\& {\cB \bG_{m,k}} \arrow[r]       & \cB G
    \end{tikzcd}
    $$
    where the right-hand square is cartesian. Thus without loss of generality, we can assume $G=\bG_m$. Since $\Theta$ is coherently complete along its closed point, it is straightforward to prove that the morphism $\Theta_k \rightarrow [V/\bG_m]$ is induced by a $\bG_m$-equivariant morphism $\bA^1 \rightarrow V$. The result follows from a straightforward computation. The proof can be adapted easily to the other case.
\end{proof}

\begin{proposition}\label{prop:affine-lin-GIT}
    Let $V$ be a $G$-representation and let $Z$ be a $G$-equivariant closed subset of $V$ which contains $0$ and denote by $U$ the open complement. The following are equivalent:
    \begin{itemize}
        \item[(i)] there exists a cocharacter $\bG_m \rightarrow G$ such that $Z$ contains neither $V^{>}$ nor $V^0$;
        \item[(ii)] there exists a commutative diagram
    $$
    \begin{tikzcd}
    \Theta_R \setminus 0 \arrow[d, hook] \arrow[r, "f^0"] & {[U/G]} \arrow[d, hook] \\
    \Theta_R \arrow[r, "f"]                               & {[V/G]}.
    \end{tikzcd}
    $$
    where $f$ sends the closed point of $\Theta_R$ to the $0$ of the $G$-representation $V$.
    \end{itemize}
    Moreover, the following are equivalent:
    \begin{itemize}
        \item[(i)] there exists a cocharacter $\bG_m \rightarrow G$ such that $Z$ contains neither $V^{>}$ nor $V^{<}$;
        \item[(ii)] there exists a commutative diagram
    $$
    \begin{tikzcd}
    \operatorname{ST}_R \arrow[d, hook] \arrow[r, "f^0"] & {[U/G]} \arrow[d, hook] \\
    \overline{\operatorname{ST}}_R \arrow[r, "f"]                               & {[V/G]}.
    \end{tikzcd}
    $$
    where $f$ sends the closed point of $\overline{\operatorname{ST}}_R$ to the $0$ of the $G$-representation $V$.
    \end{itemize}
\end{proposition}

\begin{proof}
    We will treat the $\Theta_R$ case. First of all, that (ii) implies (i) follows straightforwardly from \Cref{lem:Theta-and-BGmR}. Let us assume (i). First of all, by pulling back through the cocharacter $\chi:\bG_m \rightarrow G$, it is enough to consider the case $G=\bG_m$, with the identity cocharacter. We have that $Z\subset V$ is a $\bG_m$-equivariant subset of $V$ which does not contain $V^{>}$ nor $V^0$. First, notice that because $Z$ does not contain $V^0$, this implies we can find a dvr $R$ and a morphism $f_0:\spec R \rightarrow V^0$ such that $f$ sends the closed point of $\spec R$ in the origin of the subrepresentation $V^0$ while the open point is mapped to the generic point of $V^0$, which is not in $Z$. Similarly, because $Z$ does not contain $V^{>}$ we can find a closed rational point $x_1\in V^{>}(k)$ that is not in $Z$. Because the weights of $\bG_m$ on $V^{>}$ are strictly positive, we know that $0 = \lim_{t\rightarrow 0}t.x_1$ in the quotient stack $[V^{>}/\bG_m]$. Therefore we get a morphism $f_+:\Theta_k \rightarrow [V^{>}/\bG_m]$ such that the generic point $\spec k \subset \Theta_k$ is mapped to $x_1 \in [V^{>}/\bG_m]$. We constructed a morphism
    $$ (f_+,f_0):\Theta_R=\Theta_k \times_k \spec R \longrightarrow [V^{>}\oplus V^0/\bG_m]\subset [V/\bG_m];$$
    we leave it to the reader to verify that it satisfies the desired properties.
    The $\ST_R$ case is identical. Indeed, (ii) implies (i) follows from \Cref{lem:Theta-and-BGmR}. We assume (i). Repeating the argument above we can construct the following diagram
    $$
    \begin{tikzcd}
    {[\bA^1/\bG_m]} \arrow[rr, "f_+"] \arrow[rd] &                                           & {[V^{>}/\bG_m]} \arrow[rd] &          \\
      & \cB \bG_m \arrow[rr, Rightarrow, no head] &                 & \cB\bG_m \\
    {[\bA^1/\bG_m]} \arrow[rr, "f_-"] \arrow[ru] &                                           & {[V^{<}/\bG_m]} \arrow[ru] &
    \end{tikzcd}
    $$
    where the source of $f_+$ (respectively $f_-$) is the quotient of $\bA^1$ by the $\bG_m$-action of weight $1$ (respectively of weight $-1$). By taking the fiber product, we get a morphism
    $$ (f_+,f_-): [\bA^2/\bG_m] \longrightarrow [V^{>}\oplus V^{<}/\bG_m] \subset [V/\bG_m]$$
    where the $\bG_m$-action of $\bA^2$ has weights $(1,-1)$. Finally, the result follows because $\overline{\rm ST}_R$ is the local version of $[\bA^2/\bG_m]$, namely
    $$ \overline{\rm ST}_R\simeq [\bA^2/\bG_m]\times_{\bA^1}\spec R\longrightarrow [\bA^2/\bG_m]$$
    where $[\bA^2/\bG_m]\rightarrow \bA^1$ is the good moduli space morphism and $\spec R \rightarrow \bA^1$ is a morphism sending the closed point of $\spec R$ to $0 \in \bA^1$. In our case, we can just choose $R=k[[t]]$.
\end{proof}

\subsection{Reduction to the deformation space}
In this subsection, we show how one can use local structure results to give a global version of what we did in the previous subsection.

Let $\cX$ be an algebraic stack with affine stabilizers, of finite type over a noetherian base scheme $S$, and $x:\spec k \rightarrow \cX$ be a point. We denote by $G_x$ the stabilizer group and by $T_x\cX$ the $G_x$-representation given by the tangent vector space of $x \in X$.

\begin{remark}\label{rem:formal-local}
    Let $\cX$ be an algebraic stack with affine stabilizers of finite type over a noetherian base scheme $S$ and $x$ be a point in $\cX$ with linearly reductive stabilizer group. We know the completion exists
    $$ \eta_x: \widehat{\cX}_x \longrightarrow \cX$$
    at the point $x \in \cX$ (see Corollary 1.12 of \cite{AlpHalRyd}). Suppose now that we are given a morphism $x:\spec k \rightarrow \cX$ such that the stabilizer group $G_x$ is linearly reductive and $x$ is a smooth point of $\cX$. Then we can identify $\widehat{\cX}_x$ with the completion $[T_x\cX/G_x]^{\wedge}$ of the linearly fundamental stack $[T_x\cX/G_x]$ at the closed point $\cB G_x$. The identification follows from the fact that the first thickening  $\cX_x^{[0]}:=\cB G_x \subset \cX_x^{[1]}$ of $\cB G_x$ in $\widehat{\cX}_x$ is actually trivial, i.e. it admits a section, because $G_x$ is linearly reductive. Recall that the completion $[T_x\cX/G_x]^{\wedge}$ coincides with the fiber product
    $$
    \begin{tikzcd}
    {[T_x\cX/G_x]^{\wedge}} \arrow[r] \arrow[d] & {[T_x\cX/G_x]} \arrow[d, "p"] \\
    (T_x\cX /\!/ G_x)^{\wedge} \arrow[r]            & T_x\cX /\!/ G_x
\end{tikzcd}
    $$
    where $p$ is the good moduli space morphism and the bottom horizontal morphism is the completion of the good moduli space in $0$.
\end{remark}

\begin{lemma}\label{lem:lifting-complete}
    Let $\cX$ be an algebraic stack with affine stabilizers of finite type over a field $k$ and $(\cZ,\cZ_0)$ be a coherently complete pair of algebraic stacks. Suppose given a diagram
    $$
    \begin{tikzcd}
    \cZ_0 \arrow[r, hook] \arrow[d] & \cZ \arrow[d, "f"] \\
    \cB G_x \arrow[r]                 & \cX
    \end{tikzcd}
    $$
    where $G_x$ is the stabilizer group of a rational smooth point $x \in \cX(k)$. Assume moreover that $G_x$ is linearly reductive, then there exists a unique lifting
    $$ \begin{tikzcd}
{(\cZ,\cZ_0)} \arrow[r, "\exists ! \hat{f}", dashed] \arrow[rd, "f"] & {([T_x \cX/G_x]^{\wedge},\cB G_x)} \arrow[d] \arrow[r] & {[T_x\cX/G_x]} \\ &
\cX &\end{tikzcd}
$$
where the right side of the diagram is described in \Cref{rem:formal-local}.

The same description holds if we replace \(G_x\) by \(G_x^{0}\), the connected component of the identity of the stabilizer of \(x\).

\end{lemma}

\begin{proof}
    It follows by the definition of coherently complete pairs and \cite[Theorem~1.1]{CoherentTannaka} as in \cite[Section~1.7]{AlpHalRyd}.
\end{proof}

\begin{remark}\label{rem:formal-closed-subset}
    Let $\cX$ be an algebraic stack of finite type over a noetherian base scheme $S$ and $x$ be a point in $\cX$. Given $\cZ$ be a closed subset of $\cX$, we denote by $\widehat{\cZ}_x$ the completion of $\cZ$ in $x$, namely the preimage $\eta_x^{-1}(\cZ)$. If \(G_x\) is linearly reductive and $x$ is a smooth point of $\cX$, we claim that $\cZ_x$ coincides with the preimage of a closed $G_x$-invariant subset $Z_x \subset T_x \cX$ through the (completion) morphism
    $$ \tau_0: \widehat{\cX}_x\simeq [T_x\cX/G_x]^{\wedge} \longrightarrow [T_x \cX/G_x]. $$
    Indeed, it follows easily from the explicit construction of $\tau_0$ and $\eta_x$ that $\tau_0 ^{-1}\tau_0(\cZ_x) =\cZ_x$ and, since $\tau_0$ is universally generalizing and $\cZ_x$ is closed, we have that $\tau_0^{-1} \left( \overline{\tau_0(\cZ_x)}\right)=\cZ_x$ and we can set $Z_x:= \overline{\tau_0(\cZ_x)}$. Notice that there is no guarantee that $Z_x$ is unique: indeed if we add to $Z_x$ a closed $G_x$-invariant subset of $T_x\cX$ which does not contain the origin of the representation $T_x\cX$, its preimage through $\tau_0$ does not change. Nevertheless, we only need to control its behavior around the origin, so we can choose $Z_x$ as above for the following proposition. 
\end{remark}

\begin{proposition}\label{prop:local-Theta-complete}
    Let $\cX$ be an algebraic stack smooth over a field with affine diagonal and linearly reductive stabilizers (at every point). Let $\cZ$ be a closed subset of $\cX$ and $Z_x$ be the induced closed $G_x$-equivariant subset of $T_x\cX$ as described in \Cref{rem:formal-closed-subset}. The following are equivalent:
    \begin{itemize}
        \item[(i)]  for every geometric point $x:\spec k \rightarrow \cZ$ and every cocharacter $\bG_{m,k} \rightarrow G_x$, the $G_x$-equivariant closed subset $Z_x \subset T_x \cX$ contains either $T_x \cX^{>}$ or $T_x \cX^0$ (respectively  $T_x \cX^{>}$ or $T_x \cX^{<}$);
        \item[(ii)] the open embedding $\cU:=\cX \setminus \cZ \subset \cX$ is $\Theta$-complete (respectively $\textsf{S}$-complete).
    \end{itemize}
\end{proposition}

\begin{proof}
    We will prove the case for $\Theta$-completeness. The same idea works for $\textsf{S}$-completeness.
    We start with the implication $(i)\implies (ii)$. Let $R$ be a complete dvr, and suppose given a commutative diagram
    $$
    \begin{tikzcd}
    \Theta_R \setminus 0 \arrow[r] \arrow[d, hook] & \cU        \arrow[d, hook] \\ \Theta_R \arrow[r]                             & \cX.                \end{tikzcd}
    $$
    By contradiction, we assume that the closed point of $\Theta_R$ maps to a point $x$ in $\cZ$. Because $\cX$ is defined over a field, Cohen's structure theorem implies that $R$ is a $k$-algebra where $k$ is the residue field of $R$. Thus, we can base change the whole problem to $\spec k$, and, without loss of generality, assume that $x$ is a rational point. Using \Cref{lem:lifting-complete} for the pair $(\Theta_R, 0)$, we can reduce to the study of the induced morphism
     $$ \Theta_R \longrightarrow [T_x \cX / G_x]$$
     which sends the closed point of $\Theta_R$ to the zero of the tangent space. The statement follows from \Cref{prop:affine-lin-GIT}.
     For the converse, we can reason by contradiction again and use \Cref{prop:affine-lin-GIT} to construct a diagram
     $$
    \begin{tikzcd}
    \Theta_R \setminus 0 \arrow[r] \arrow[d, hook] & {[T_x\cX \setminus Z_x/G_x]}        \arrow[d, hook] \\ \Theta_R \arrow[r]                             & {[T_x\cX/G_x]}                \end{tikzcd}
    $$ 
    which can be lifted to a diagram
    $$
    \begin{tikzcd}
    \Theta_R \setminus 0 \arrow[r] \arrow[d, hook] &     \widehat{\cX}_x\setminus \cZ_x  \arrow[d, hook] \\ \Theta_R \arrow[r]                             &       \widehat{\cX}_x        \end{tikzcd}
    $$ 
    thanks to  \Cref{rem:formal-closed-subset}. To conclude, it is enough to notice that $\Theta$-completeness (and $\textsf{S}$-completeness) are properties stable by base change.
\end{proof}

\begin{remark}
    The equivalent condition of $\Theta$-completeness in the previous proposition can be equivalently stated as follows (even if apparently stronger):
    \begin{itemize}
        \item[(i)]  for every $x:\spec k \rightarrow \cZ$ geometric point and every cocharacter $\bG_{m,k} \rightarrow G_x$, the $G_x$-equivariant closed subset $\cZ_x \subset T_x \cX$ contains either $T_x \cX^0$ or $T_x \cX^{>} \cup T_x \cX^{<}$;
    \end{itemize}
\end{remark}

\begin{remark}
    As it was pointed out in \Cref{rem:formal-closed-subset}, the closed subset $Z_x\subset T_x\cX$ is not unique (although it is not a problem for our application). Nevertheless, condition $(i)$ in \Cref{prop:local-Theta-complete} can be rewritten in a more intrinsic way. For every geometric point $x:\spec k \rightarrow \cZ$ and every cocharacter $\bG_{m,k} \rightarrow G_x$, we have the following diagram
    $$\begin{tikzcd}
    & X_x \arrow[d, "p"] \arrow[r, "t_0"] & T_x\cX \arrow[d] \\
    \cX & \widehat{\cX}_x \arrow[r, "\tau_0"'] \arrow[l, "\eta_x"] & {[T_x\cX/G_x]}          
    \end{tikzcd}$$
    where the square is Cartesian. Condition $(i)$ can then be restated as 
    \begin{itemize}
        \item[$(i')$] for every geometric point $x:\spec k \rightarrow \cZ$ and every cocharacter $\bG_{m,k} \rightarrow G_x$, we have that $p^{-1}(\cZ_x)$ contains either $t_0^{-1}(T_x \cX^{>})$ or $t_0^{-1}(T_x \cX^0)$;
    \end{itemize}
    where $X_x$ is the fiber product of affine schemes $T_x\cX \times_{T_x\cX//G_x} (T_x\cX//G_x)^{\wedge}$. The equivalence between $(i)$ and $(i')$ is a direct consequence of \Cref{rem:formal-closed-subset}. The same is true for $S$-completeness.
\end{remark}
\section{\texorpdfstring{$A_r$}{A\_r}-stable curves}
\label{sec:notation}
This section is dedicated to the introductory material regarding curves with singularities of type $A$. Specifically, we start by recalling some basic theory regarding curves with $A$-singularities and their moduli stack in \Cref{sub:A_r-crimping}, where we also study the crimping spaces associated to $A$-singularities. In \Cref{sub:equivariant-normalization}, we study the equivariant properties of normalizations of $A$-singularities, with the goal of understanding which automorphisms (in the connected component of the identity) descend through the normalization. Finally, in \Cref{sub:hyper-A-stable} we concentrate our efforts on understanding hyperelliptic curves with $A$-singularities, as they play a crucial role in the study of isotrivial degenerations of $A$-singularities. We end the section by proving that, in the case of the stack of cyclic covers of the projective line, the full classification of opens which admit a separated good moduli space coincides with the answer provided by Geometric Invariant Theory (see \Cref{theo:classify-opens-hyp}).

We fix a base field $\kappa$ of characteristic zero. Every time we talk about genus, we mean arithmetic genus, unless specified otherwise. If $R$ is a discrete valuation ring (or dvr for short), then we will denote by $Q$ its fraction field and by $k$ its residue field.

\subsection{\texorpdfstring{$A_r$}{A\_r}-stable curves and crimping}\label{sub:A_r-crimping}
In this subsection, we collect some results about curves with \(A_k\)-singularities, in particular about the normalization of an \(A_r\) singularity, and conversely how one or two smooth points on a curve together with some infinitesimal data can be \emph{crimped} into an \(A_k\) singularity. Many of the results about crimping spaces originally come from the PhD thesis of van der Wyck \cite{VanDerWyck}. We will follow the notation introduced in the third-named author's PhD thesis (specifically \cite{Per1}), where some of the results of \cite{VanDerWyck} were generalized to families over any base scheme.

We start by fixing some classical notations.
\begin{definition}
    A curve $C$ over an algebraically closed field $k$ is a connected, reduced, one-dimensional scheme proper over $k$. An $n$-pointed curve $(C,p_1,\dots,p_n)$ over a field $k$ is a curve $C$ together with $n$ rational smooth points $p_1,\dots,p_n$.
\end{definition}

\begin{definition}[Partial Normalization]
    Let $C$ be a curve over an algebraically closed field $k$, and let
    $S \subset \operatorname{Sing}(C)$ be a subset of its singular points.
    A finite morphism
    \[
    \nu \colon \widetilde{C} \to C
    \]
    is called the \emph{partial normalization of $C$ along $S$} if $\nu$ is an isomorphism over $C \setminus S$ and the induced map $\widetilde{C} \setminus \nu(S^c) \to C\setminus S^c$ is the normalization, where $S^c = \operatorname{Sing}(C) \setminus S$.
\end{definition}

\begin{definition}
    A closed point $p \in C$ over an algebraically closed field $k$ is called an \emph{$A_h$-singularity} if
    $$ \widehat{\cO_{C,p}} \simeq k[[x,y]]/(y^2-x^{h+1}).$$
    If we do not want to stress the integer $h$, we will drop the subscript and write \emph{$A$-singularity}. An $A$-singularity $p$ of $C$ is \emph{separating} if the partial normalization of $C$ at $p$ is disconnected. Moreover, we say that $C$ is an $A$-prestable curve if all the singular points of $C$ are $A$-singularities.
\end{definition}

\begin{definition}[Pointed partial normalization]
    \index{Pointed partial normalization}
    Let $(C,\{p_i\}_{i\in I})$ be a pointed curve, and let $S \subset \operatorname{Sing}(C)$ be a finite subset.
    The \emph{pointed partial normalization of $(C,\{p_i\})$ along $S$} is the triple $(\widetilde{C},\{q_j\}_{j\in J}, \nu)$ where
    \[
    \nu \colon \widetilde{C} \to C
    \]
    is the partial normalization of $C$ along $S$ and $\{q_j\}_{j\in J} \coloneq \nu^{-1}\!\left(S \cup \{p_i \mid i \in I\}\right)$. If $S = \operatorname{Sing}(C)$, then $\nu$ is called the \emph{pointed normalization}.
\end{definition}

Notice that the normalization morphism is not a morphism between pointed curves; however, it sends markings to special points, which are either markings or singularities.

Let $g$ and $n$ be two nonnegative integers. Fix a nonnegative integer $r$.

\begin{definition}[$A_r$-stable curve]\index{$A_r$-stable curve}
    Let $k$ be an algebraically closed field and $(C,p_1,\dots,p_n)/k$ be an $n$-pointed $A$-prestable curve. We say that $C$ is $A_r$-prestable if every singular point is an $A_h$-singularity for $h\leq r$. We say that $(C,p_1,\dots,p_n)$ is \emph{$A_r$-stable} if it is $A_r$-prestable and the line bundle $\omega_C(p_1+\dots+p_n)$ is ample.
\end{definition}
\begin{remark}
    Notice that an $A$-prestable curve $C$ is l.c.i. by definition, therefore the dualizing complex $\omega_C$ is in fact a line bundle.
\end{remark}

\begin{remark}\label{rem:genus-count}
    Let $C$ be a curve over an algebraically closed field and let $p$ be an $A_r$-singularity.  We denote by $b:\widetilde{C}\rightarrow C$ the partial normalization at the point $p$ and by $J_b$ the conductor ideal of $b$. Then a classical computation shows that
    \begin{enumerate}
        \item if $r=2h$, then $g(C)=g(\widetilde{C})+h$;
        \item if $r=2h+1$ and $\widetilde{C}$ is connected, then $g(C)=g(\widetilde{C})+h+1$,
        \item if $r=2h+1$ and $\widetilde{C}$ is not connected, that is \(p\) is a separating singularity, then $g(C)=g(\widetilde{C})+h$.
    \end{enumerate}
    Furthermore, Noether's conductor formula gives us that $b^*\omega_C \simeq \omega_{\widetilde{C}}(J_b^{\vee})$. See, for instance, \cite[Proposition~1.2]{Cat}.
\end{remark}

Let $\cM_{g,n}^r$ be the fibered category over $\kappa$-schemes whose objects over \(S\) are tuples \((C,p_1,\ldots,p_n)\) where \(C\to S\) is a proper flat morphism with sections \(p_i:S\to C\) such that for every geometric point \(x:\Spec k\to S\), \((C_x,p_1(x),\ldots,p_n(x))\) is an \(n\)-pointed \(A_r\)-stable curve of genus \(g\). These families are called \emph{$n$-pointed $A_r$-stable curves} over $S$. Morphisms are just isomorphisms of $n$-pointed curves.

We recall the following properties of $\cM_{g,n}^r$. See \cite[Theorem~2.2]{Per1} for the proof.

\begin{theorem}\label{theo:descr-quot}
    $\cM_{g,n}^r$ is a smooth connected algebraic stack of finite type over $\kappa$. Furthermore, it is a quotient stack: that is, there exists a smooth quasi-projective scheme X with an action of $\GL_N$ for some positive $N$, such that
    $ \cM_{g,n}^r \simeq [X/\GL_N]$.
\end{theorem}

\begin{remark}\label{rem: max-sing}
Recall that we have an open embedding $\cM_{g,n}^r  \subset \cM_{g,n}^s$ for every $r\leq s$. Notice that $\cM_{g,n}^r=\cM_{g,n}^{2g+1}$ for every $r\geq 2g+1$, because of \Cref{rem:genus-count}. Thus, we can consider $\cM_{g,n}^{2g+1}$ as the moduli stack of all $A$-stable curves (with no restriction on the singularities).
\end{remark}

\subsection*{Crimping morphisms of \texorpdfstring{$A_r$}{A\_r}-singularities}

In this subsection, we describe the crimping datum needed to construct an $A_r$-stable curve starting from its normalization. Contrary to the nodal and cuspidal cases, it is not enough to remember only the closed points mapping to a given singularity if it is of type \(A_l\) with \(l\geq 3\). We want to stress again that the crimping morphism described below is not new: it can be found, for instance, in \cite{VanDerWyck}. We slightly generalize the analysis of \emph{loc.~cit.} for $A$-type singularities by extending the crimping morphism also to the case where the combinatorics of the curve is not fixed. In particular, when we restrict the normalization map in \Cref{prop:descr-an-even} to the locally closed substack given by specified combinatorics of the curve, we recover the description given in \cite[Theorem~1.105]{VanDerWyck} for $A$-type singularities.

The main players in the description are the stacks \(\cA_l\) for \(0\leq l \leq r\), parametrizing families of \(A_r\)-stable curves (with \(n\) marked points) together with a section that is an \(A_l\)-singularity in each fiber. This stack is naturally a locally closed substack of the universal curve over \(\cM_{g,n}^r\) as is shown in \cite[Section~4]{Per1} (in the case \(n=0\), but the argument in the pointed case is identical).
Depending on whether \(l\) is even or odd, the geometry of \(\cA_l\) and its relation to crimping vary, so we treat the two cases separately in what follows.
The following was done in \cite[Sections~3 and 4]{Per1} in the case $n=0$, but since the markings are always smooth, nothing changes in the proofs of the following results.

\subsubsection*{Even $A_r$-singularities}

Suppose $r=2h$ for some integer $h\geq 1$. We start by discussing the data one needs to construct an $A_{2h}$-singularity \emph{crimping} a smooth point.

Consider the affine scheme $E_{h,2}$ which classifies finite flat extensions of degree $2$ of the form
$$ \kappa[t]/(t^h) \hookrightarrow \kappa[t]/(t^{2h});$$
these are determined by the choice of an element $\phi(t) \in \kappa[t]/(t^{2h-3})$ with $\phi(0)\neq 0$, using the association $t\mapsto t^2\phi(t)$. Therefore $E_{h,2}$ is isomorphic to $(\bA^1 \setminus \{0\}) \times \bA^{2h-3}$. If we denote by $G_h$ the automorphism group of $\kappa[t]/(t^h)$, we have an algebraic stack $\cE_{h,2}:=[E_{h,2}/G_h\times G_{2h}]$, which parametrizes finite flat extensions of degree $2$ as above, up to isomorphisms of both the source and target. As we are working over a field of characteristic \(0\), $G_h$ is isomorphic to $\bG_m\ltimes U_h$, where $U_h$ is the unipotent radical, and an element of $G_h$ can be identified with the application $t\mapsto t\varphi(t)$ where $\varphi(t)\in \kappa[t]/(t^{h-1})$ such that $\varphi(0)\neq 0$. See \cite[Section~3]{Per1} for a more detailed discussion. The stack $\cB G_{h}$ classifies finite flat algebras of length $h$ which \'etale locally are of the form $\kappa[t]/(t^h)$ (cf. \cite[Corollary~3.5]{Per1}). The natural morphism
$$ \cE_{h,2} \longrightarrow \cB G_{2h}$$
classifies $A_{2h}$ singularities locally as we now explain. The idea is that to construct an $A_{2h}$-singularity starting from a smooth point in a curve, it is enough to give a finite flat (sub)extension of degree $2$ as the one described above, and then consider the pushout of the curve with the finite flat (sub)algebra of the infinitesimal neighbourhood of order $2h-1$ of the smooth point. The rest of this section will elaborate on this idea. We start with some remarks on the spaces introduced above.

\begin{remark}[{\cite[Lemma~3.8]{Per1}}]
\label{rem:crim-mor-even}
The quotient
\[
V_{h,2}:=[E_{h,2}/G_h]
\]
is a scheme and can be described as an affine space $\bA^{h-1}$: a point in such a space corresponds to a finite flat extension determined by $\phi(t) \in \kappa[t]/(t^{2h-2})$ such that $\phi(t)=1+a_1 t+a_2 t^3 +\dots +a_{h-1} t^{2h-3}$.
\end{remark}

Before describing $\cA_{2h}$, we introduce the stack which classifies the partial normalization along the equisingular section without the crimping datum needed to recover the curve.

\begin{definition}\label{def:stabilization}
    Given a positive integer $l$, we denote by $\cM_{g,n,[l]}^r$ the stack parametrizing objects $(\widetilde{C},p_1,\dots,p_n,q)$ where $\widetilde{C}$ is an $A_r$-prestable curve, $p_1,\dots,p_n,q$ are smooth distinct sections, and $\omega_C(p_1+\dots+p_n+lq)$ is ample. By definition, we have that
    $\cM_{g,n,[1]}^r \simeq \cM_{g,n+1}^r.$
\end{definition}

One can compare the stack of partial normalizations with the moduli stacks of $A_r$-stable curves using the stabilization morphism $\cM_{g,n,[l]}\rightarrow \cM_{g,n+1}$, defined by contracting unstable irreducible components (see \cite[Remark~1.6, Corollary~1.9 and Proposition~1.17]{Per1}).

\begin{proposition}[{\cite[Proposition~4.9]{Per1}}]\label{prop:descr-base-crimping-even}
    Given $g,n$ nonnegative integers and $l\geq 2$, we have that
    $$ \cM_{g,n,[l]}^r \simeq \cM_{g,n+1}^r \times [\bA^1/\bG_m]$$ if $2g-2+n\geq 0$. If $g=0$ and $n=1$, we have
    $$ \cM_{0,1,[l]}^r \simeq \cB \bG_m$$ for any $l\geq 2$. If $g=n=0$, then $\cM_{g,n,[l]}^r$ is isomorphic to $\cB(\bG_m \ltimes \bG_a)$ if $l\geq 3$, whereas for $l=2$ it is empty.
\end{proposition}

Let \(\widetilde{C}\to T\) be a family of curves over an algebraic stack $T$ and \(q:T\to \widetilde{C}\) a smooth section. This data induces a map \(j_h:T\to \cB G_{h}\) for any \(h\in \mathbb{Z}_{\geq 1}\) defined by the infinitesimal neighbourhood of order $h-1$ of the section \(q\). Whenever $\widetilde{C}$ is a family of curves in $\cM_{g-h,n,[2h]}^r$ and $q$ is the section corresponding to the last marking, the map is given by the composition
\[
j_{2h} \colon T \to \cM_{g-h,n,[2h]}^r \to \cB G_{2h} .
\]
Over a field $\kappa$, we can consider a curve $(C, s)$ and an extension of algebras as above $\operatorname{cr} \colon \spec \kappa[t]/(t^{2h}) \to \spec \kappa[t]/(t^h)$ and perform a Ferrand pushout as in the following diagram
\begin{equation*}\label{eq:singularity-from-crimping-datum-even}\begin{tikzcd}
    {\Spec \kappa[t]/(t^{2h})} & \widetilde{C} \\
    {\Spec \kappa[t]/(t^{h})} & {C}
    \arrow["{q_{2h}}", from=1-1, to=1-2]
    \arrow["{\operatorname{cr}}"', from=1-1, to=2-1]
    \arrow[from=1-2, to=2-2]
    \arrow[from=2-1, to=2-2]
\end{tikzcd}\end{equation*}
where \(q_{2h}:\Spec \kappa[t]/{t^{2h}} \to \widetilde{C}\) denotes the closed inclusion of the \((2h-1)\)-th order infinitesimal neighborhood of \(q\in C\). The pushout defines a curve \(C\) isomorphic to \(\widetilde{C}\) outside the image of the horizontal arrows, with the image of \(q\) being an \(A_{2h}\)-singularity with \(\widetilde{C}\) the partial normalization of \(C\) at \(q\), see \Cref{prop:descr-an-even} below.

Conversely, if \((\widetilde{C},q)\) is the partial normalization of a pointed curve \((C,p)\) over an algebraically closed field \(\spec \kappa\) where \(p\) is an \(A_{2h}\)-singularity, then there exists an extension of algebras (not necessarily unique) $\operatorname{cr} \colon \spec \kappa[t]/(t^{2h}) \to \spec \kappa[t]/(t^h)$ and a Ferrand pushout diagram as above.

Denote by $\cA_{2h}$ the locally closed substack of the universal curve $\cC_{g,n}^r$ of $\cM_{g,n}^r$ classifying objects $(C,p_1,\dots,p_n,p)$ where $p$ is an $A_{2h}$-singularity. Then the discussion above can be carried out in families; this is the content of the following \cite[Corollary~4.12]{Per1}.

\begin{proposition}\label{prop:descr-an-even}
    The pointed normalization of a point $(C,p_1,\dots,p_n,p) \in \cA_{2h}(T)$ along the equisingular section $p$ defines a morphism
    $$\cA_{2h} \longrightarrow \cM_{g-h,n,[2h]}^r$$
    that we call the \emph{crimping morphism}. This is obtained via the pullback
    \[
    \begin{tikzcd}
    \cA_{2h} \arrow[d] \arrow[r]   & {[V_{h,2}/G_{2h}]} \arrow[d] \\
    {\cM_{g-h,n,[2h]}^r} \arrow[r] & \cB G_{2h}
    \end{tikzcd}
    \]
    where the bottom morphism is the one described by mapping $(\widetilde{C},p_1,\dots,p_n,q)$ to the infinitesimal neighbourhood of order $2h-1$ of $q$, which is \'etale locally of the form $\kappa[t]/(t^{2h})$.
    In particular, $\cA_{2h}$ is an affine bundle of dimension $h-1$ over the stack $\cM_{g-h,n,[2h]}^r$.
\end{proposition}

This motivates the following definition.

\begin{definition}\label{def:even-crimping-datum}
    Given a family $\widetilde{C}$ of curves over an algebraic stack \(T\) and a section $q$ over \(T\), a \emph{crimping datum} of type \(A_{2h}\) for the pair is a lift \(\operatorname{cr} \colon T\to \cE_{2h}\) of \(j_{2h}:T\to \cB G_{2h}\), i.e., a dotted arrow making the following diagram
    $$
    \label{eq:crimping-data}
    \begin{tikzcd}
    & {\cE_{2h}} \\
    T & {\cB G_{2h}}
    \arrow[from=1-2, to=2-2]
    \arrow["{\operatorname{cr}}", dashed, from=2-1, to=1-2]
    \arrow["{j_{2h}}"', from=2-1, to=2-2]
    \end{tikzcd}
     $$
    commute.
\end{definition}

\subsubsection*{Odd $A_r$-singularities}
We will now deal with the odd case; similar results obtained in the previous section hold, and we will adopt the same notation.

\begin{remark}\label{rem:crim-mor-odd}
In \cite[Section~3]{Per1}, the author proved that the diagonal morphism
$$ \cB G_{h} \longrightarrow \cB G_{h} \times \cB G_{h} $$
is the crimping morphism of the $A_{2h-1}$-singularity. In other words, given a pair of smooth points (in either one or two distinct curves), describing a crimping of them into an \(A_r\)-singular point is equivalent to giving an identification of their respective infinitesimal neighbourhoods of order $h-1$.
\end{remark}

As in the even case, denote by $\cA_{2h-1}$ the locally closed substack of the universal curve $\cC_{g,n}^r$ of $\cM_{g,n}^r$ classifying objects $(C,p_1,\dots,p_n,q)$ where $q$ is an $A_{2h-1}$-singularity. Before describing the relation of $\cA_{2h-1}$ to crimping morphisms, we introduce the stack which classifies the partial normalization along the equisingular section, in the case where the partial normalization is connected.

\begin{definition}
    Let $l$ be a positive integer. We denote by $\cM_{g,n,2[l]}^r$ the stack parametrizing $(\widetilde{C},p_1,\dots,p_n,q_1,q_2)$ where $\widetilde{C}$ is an $A_r$-prestable curve and $p_1,\dots,p_n,q_1,q_2$ are smooth distinct sections such that the line bundle $\omega_C(p_1+\dots+p_n+l(q_1+q_2))$ is ample. By definition we have $\cM_{g,n,2[1]}^r \simeq \cM_{g,n+2}^r.$
\end{definition}

Similarly to the even case, we can compare the stack of partial normalizations with the moduli stacks of $A_r$-stable curves using the stabilization procedure (see \cite[Proposition~4.15]{Per1}).
\begin{proposition}\label{prop:descr-base-crimping-odd}
    Given $g,n$ nonnegative integers and $l\geq 2$, we have that
    $$ \cM_{g,n,2[l]}^r \simeq \cM_{g,n+2}^r \times [\bA^1/\bG_m] \times [\bA^1/\bG_m]$$ if $2g-1+n\geq 0$. If $g=n=0$, then $\cM_{0,0,2[l]}^r$ is isomorphic to $[\bA^1/\bG_m^2]$ for $l\geq 2$.
\end{proposition}

\begin{remark}\label{rem:sprout-proj-line}
   \Cref{prop:descr-base-crimping-even} and \Cref{prop:descr-base-crimping-odd} warn us that the pointed partial normalization of an $A_r$-stable curve at an $A$-singularity is not $A_r$-stable, in contrast with the nodal case. For instance, we can have an irreducible rational curve $(\Gamma,p)$, attached to a smooth curve in a node $p$, with an even $A$-singularity $q$ on $\Gamma$. When we normalize, the component $(\widetilde{\Gamma},p,q)$ is not stable anymore, only prestable. This is taken into account by adding the $[\bA^1/\bG_m]$ in \Cref{prop:descr-base-crimping-even}. The generic point of $[\bA^1/\bG_m]$ corresponds to the case where there is no unstable component, thus the point $p$ will be pinched to create the singularity directly in the smooth curve. The closed point corresponds to the case where we pinch an unstable component, which has in fact a $\bG_m$ as automorphism group, making it stable. The odd case is similar, with the only difference that we can sprout an unstable component from both the points we are gluing in the normalization.
\end{remark}

Contrary to the even case, $\cA_{2h-1}$ is the disjoint union of several connected components, namely $\cA_{2h-1}^{\rm ns}$ and $\cA_{2h-1}^{i,S}$ for $0\leq i\leq (g-m+1)/2$ and $S\subset \{1,2,\dots,n\}$. The stack $\cA_{2h-1}^{\rm ns}$ parametrizes $(C,p_1,\dots,p_n,q)$ such that the partial normalization of $C$ at $q$ is connected; in other words, the singularity \(q\) is non-separating. Furthermore, $\cA_{2h-1}^{i,S}$ parametrizes pairs $(C,p_1,\dots,p_n,q)$ such that the partial normalization of $C$ at $q$ is the disjoint union of two components, respectively of genus $i$ and $g-h-i+1$, such that $p_j$ belongs to the component of genus $i$ if and only if $j \in S$.

The relation of these components to the crimping morphism is slightly more involved than in the even case, but in the end rather similar. In \cite{Per1}, the author introduces the moduli stack $\cI_{2h-1}^{\rm ns}$ parametrizing the partial normalization of a curve in $\cA_{2h-1}^{\rm ns}$ at the equisingular section plus a crimping datum, and similarly $\cI_{2h-1}^{i,S}$. In the following remark, we recall the formal constructions.

\begin{remark}\label{rem:descr-an-disp-i}
     Exactly as in the even case, one can construct $\cI_{2h-1}^{i,S}$ by considering the following diagram:
    $$ \begin{tikzcd}
    \cA_{2h-1}^{i,S}  &{\cI_{2h-1}^{i,S}} \arrow["F^{i, S}"', l] \arrow[r] \arrow[d]                 & \cB G_h \arrow[d]      \\
    &{\cM_{g-i-h+1,n-s,[h]} \times \cM_{i,s,[h]}} \arrow[r] & \cB G_h \times \cB G_h
    \end{tikzcd}
    $$
    where the square is cartesian. The same is true for the non-separating case. We now explain how to construct $F^{i,S}$. Given a crimping datum over \(\Spec \kappa\) and a (non-connected) curve $\widetilde{C} \in  {\cM_{g-i-h+1,n-s,[h]} \times \cM_{i,s,[h]}}$, the Ferrand pushout diagram
\begin{equation*}\label{eq:singularity-from-crimping-datum-odd}\begin{tikzcd}
    {\Spec \kappa[t]/(t^h) \sqcup \Spec \kappa[t]/(t^h)} & &  \widetilde{C} \\
    {\Spec \kappa[t]/(t^{h})} && {C}
    \arrow["{q_{1,h} \sqcup q_{2,h}}"', from=1-1, to=1-3]
    \arrow["{\operatorname{cr}}"', from=1-1, to=2-1]
    \arrow[from=1-3, to=2-3]
    \arrow[from=2-1, to=2-3]
\end{tikzcd}\end{equation*}
 gives a curve $C\in \cA_{2h-1}^{i,S}$. Notice that the morphism $\operatorname{cr}: \spec \kappa[t]/(t^h) \sqcup  \spec \kappa[t]/(t^h) \rightarrow  \spec \kappa[t]/(t^h) $ is determined by the crimping datum \(\operatorname{cr}:\Spec \kappa \to \cB G_{h}\), whereas \(q_{h}:\Spec \kappa[t]/{t^{h}} \to  \widetilde{C}\) denotes the closed inclusion of the infinitesimal neighborhood of order $h-1$, as in the even case. The curve \(C\) is isomorphic to \(\widetilde{C}\) outside the image of the horizontal arrows, with the image $q$ of $q_1$ (hence also of $q_2$) being an \(A_{2h}\)-singularity with \(\widetilde{C}\) the partial normalization of \(C\) at \(q\). This gives the morphism
 $$ F^{i,S}:\cI_{2h-1}^{i,S} \longrightarrow \cA_{2h-1}^{i,S};$$
 a similar construction defines a morphism
 \[F^{\rm ns}: \cI_{2h-1}^{\rm ns} \longrightarrow \cA_{2h-1}^{\rm ns}.\]
\end{remark}

\begin{proposition}\label{prop:descr-an-disp-i}
 Let $S\subset \{1,2,\dots,n\}$ and denote by $s$ its cardinality. The morphism
 $$ F^{i,S}:\cI_{2h-1}^{i,S} \longrightarrow \cA_{2h-1}^{i,S}$$
 where $\cI_{2h-1}^{i,S}$ is affine over the product
 $$ \cM_{i,s,[h]}^r \times \cM_{g-h-i+1,n-s,[h]}^r.$$
  If $2i=g-h+1$ and $n=0$ (so \(S=\emptyset\)), then $F^{i,S}$ is finite \'etale of degree $2$; otherwise, it is an isomorphism.
\end{proposition}

A similar statement is true for the other connected component.

\begin{proposition}\label{prop:descr-an-disp-ns}
 The morphism
 $$ F^{\rm ns}: \cI_{2h-1}^{\rm ns} \longrightarrow \cA_{2h-1}^{\rm ns}$$
 is an \'etale cover of degree $2$, and $\cI_{2h-1}^{\rm ns}$ is affine over $\cM_{g-h,n,2[h]}^r$.
\end{proposition}

\begin{remark}\label{rem:involution}
    If $2i=g-h+1$ and $n=0$, we have a $\mathbb{Z}/2$ action on the target of the crimping morphism. This exchanges the marked points arising via the normalization of the equisingular section. In that case, we have a map from $\cA_{2m-1}$ to the quotient via this action and we can recover the crimping via pullback. In the non-separating case, we have
    $$ \begin{tikzcd}
    {\cI_{2m-1}^{\rm ns}} \arrow[r] \arrow[d] & \cA_{2m-1}^{\rm ns}   \arrow[d]   \\
    {\cM_{g-m,n,2 [m]}} \arrow[r] & {[{\cM_{g-m,n,2 [m]}}/{\mathbb{Z}_2}]}.
    \end{tikzcd}
    $$
\end{remark}

The following is the analogue of \Cref{def:even-crimping-datum} in the odd case.
\begin{definition}\label{def:odd-crimping-datum}
    Given a family  \((\widetilde{C}, q_1,q_2)\) of curves over a stack \(T\) and two smooth distinct sections, a \emph{crimping datum} of type \(A_{2h-1}\) for the triplet is a lift \(\operatorname{cr}:T\to \cB G_h\) of \((j_{h,1}, j_{h,2}) :T \to \cB G_h \times \cB G_h\), i.e., a dotted arrow making the following diagram
    \[
    \begin{tikzcd}
    & {\cB G_h} \\
    T & {\cB G_h \times \cB G_h}
    \arrow[from=1-2, to=2-2]
    \arrow["{\operatorname{cr}}", dashed, from=2-1, to=1-2]
    \arrow["{(j_{h,1}, j_{h,2})}"', from=2-1, to=2-2]
\end{tikzcd}
\]
commute, where $j_{h,i}$ is the morphism induced by the family $(\widetilde{C},q_i)\rightarrow T$ for $i=1,2$.

\end{definition}

\subsection{Equivariant properties of normalization}\label{sub:equivariant-normalization}
In this subsection, we collect some results on descending automorphisms along the partial normalization map of an \(A_r\)-(pre)stable curve.

Because $A$-singularities are either unibranched (in the even case) or double branched (in the odd case), there are strong restrictions on which curves have positive-dimensional stabilizers. Indeed, given that the normalization is smooth local (cf. \cite[\href{https://stacks.math.columbia.edu/tag/07TD}{Lemma 07TD}]{stacks-project}), we have the following remark.
\begin{remark}
    \label{rem:motivation-hyp-sing}
    Let $(C, p_1, \dots, p_n)$ be an $n$-pointed $A_r$-stable curve over an algebraically closed field $k$ with $\car k=0$, and denote by $\nu \colon \widetilde{C} \to C$ the partial normalization along some singularities $S \subset {\rm Sing\,} C$. The connected component of the identity of the automorphism group lifts to the partial normalization, and we have a natural inclusion
        \[
         \iota \colon \Aut^{\circ}(C, p_1, \dots,p_n) \hookrightarrow \Aut^{\circ}(\widetilde{C}, \nu^{-1}(p_1), \dots, \nu^{-1}(p_n)) .
        \]
    The image of $\iota$ acts trivially on any connected component $\Gamma \subset \widetilde{C}$ not isomorphic to $\bP^1$.\\
    Consider a cocharacter $\Gm \rightarrow \Aut^{\circ}(C, p_1, \dots,p_n)$ and a singularity $p \in C$. We have two possibilities:
    \begin{itemize}
        \item $p$ is an $A_{2h}$-singularity. Then $\nu^{-1}(p)$ contains a unique topological point $q$, and we have an induced action of $\Gm$ on $T_{q} \widetilde{C}$.
        \item $p$ is an $A_{2h+1}$-singularity. Then $\nu^{-1}(p)$ contains two topological points $q_1, q_2$, and we have an induced action of $\Gm$ on $T_{q_i} \widetilde{C}$ for $i=1,2$.
    \end{itemize}
    In these cases, the action of $\Gm$ on the tangent space of $p$ in $C$ is nontrivial only if the connected component (respectively components) of $\widetilde{C}$ containing $q$ (resp. $q_1$ and $q_2$) is isomorphic to $\bP^1$.
\end{remark}

\begin{remark}
    \label{rem:aut-and-normalization-along-nodes}
    With the notation of \Cref{rem:motivation-hyp-sing}, whenever $S \subset C(k)$ only contains nodal singularities, the map $\iota$ is an isomorphism.
\end{remark}

In what follows, we are going to relate the automorphism group of an $A_r$-prestable curve with the one of its partial pointed normalization at a singularity. As always, we treat the even and odd cases separately. This will be fundamental in the classification of curves with positive-dimensional stabilizers.

\subsubsection*{Even case}

First, we describe the moduli stack of curves with a prescribed partial normalization at an even $A$-singularity. This is a direct consequence of \Cref{prop:descr-an-even}.

\begin{lemma}\label{cor:diff-crimping-same-curve-even}
    Let $(\widetilde{C},q)$ be a $1$-pointed curve over an algebraically closed field. The fiber product
    $$
    \begin{tikzcd}
    {[V_{h,2}/\Aut(\widetilde{C},q)]} \arrow[r] \arrow[d] & {[V_{h,2}/G_{2h}]}=\cE_{h,d} \arrow[d] \\
    {\cB \Aut(\widetilde{C},q)} \arrow[r]                 & \cB G_{2h}
    \end{tikzcd}
    $$
    parametrizes curves $(C,p)$ where $p$ is an $A_{2h}$-singularity such that the (pointed) partial normalization at $p$ is isomorphic to $(\widetilde{C},q)$.
\end{lemma}

\begin{definition}
    Let $(\widetilde{C},q)$ be a $1$-pointed curve over an algebraically closed field and $G\subset \Aut(\widetilde{C},q)$ be a subgroup. A $G$-equivariant crimping datum is a crimping datum for the family $[(\widetilde{C},q)/G]\rightarrow \cB G$.
\end{definition}

\begin{remark}\label{rem:diff-crimping-same-curve-even}
A $G$-equivariant crimping datum
$$\operatorname{cr} \colon \cB G \to \cE_{h,2}$$
exists if and only if one can crimp an $A_{2h}$-singularity on $(\widetilde{C},q)$ in such a way that the subgroup $G \subset \Aut(\widetilde{C},q)$ descends to a subgroup of the automorphisms of $C$ along the normalization $\widetilde{C} \to C$. Notice that different $G$-equivariant crimping data may give rise to the same curve. Indeed, two $G$-equivariant crimping data
$$ \operatorname{cr}_1, \operatorname{cr}_2: \cB G \to \cE_{h,2} $$
will give rise to the same curve if and only if the two induced morphisms $\cB G \rightarrow \cE_{h,2} \times_{\cB G_{2h}} \cB\Aut(C,s)$ are isomorphic. This follows directly from the fact that the quotient stack $\cE_{h,2} \times_{\cB G_{2h}} \cB\Aut(C,s)$ parametrizes curves $(C,p)$ where $p$ is an $A_{2h}$-singularity and the partial normalization at $p$ is isomorphic to $(\widetilde{C},q)$ (see \Cref{cor:diff-crimping-same-curve-even}).
\end{remark}

The case \(G=\mathbb{G}_m\) will play a central role in our study of isotrivial degenerations of \(A_r\)-stable curves.

\begin{lemma}\label{lem:hyp-crimping-datum-even}
    Let $h$ be a positive integer. Let $\widetilde{C}/k$ be a (possibly pointed) curve over an algebraically closed field $k$ with a rational smooth point $q \in \widetilde{C}(k)$. Let $\Gm \rightarrow \Aut^{\circ}(\widetilde{C},q)$ be a cocharacter such that the action on $T_q \widetilde{C}$ is nontrivial. Then there exists a unique $\Gm$-equivariant crimping datum, i.e., there exists a unique dotted arrow making the following diagram
    $$
    \begin{tikzcd} & {\cE_{h,2}} \arrow[d] \\
    \cB \Gm \arrow[r, "j_{2h}"'] \arrow[ru, "\exists !", dashed] & \cB G_{2h}
    \end{tikzcd}
    $$
    commute.
\end{lemma}

\begin{proof}
    This follows from a straightforward computation. Indeed, we have seen in \Cref{sub:A_r-crimping} that $\cE_{h,2}$ can be described as the quotient stack $[V_{h,2}/G_{2h}]$ where $V_{h,2}$ is an affine space. Thanks to \Cref{rem:motivation-hyp-sing}, it is clear that $q$ lies in an irreducible component of $\widetilde{C}$ whose normalization is isomorphic to $\bP^1$, and the cocharacter given induces a cocharacter $\Gm \rightarrow \Aut^{\circ}(\bP^1,\infty)\simeq \Gm \ltimes \Ga$. Since the morphism $j_{2h}:\cB \bG_m\to \cB G_{2h}$ is local on $q$, we can assume $(C,p)\simeq (\bP^1,\infty)$ and compute $j_{2h}$ explicitly.  Using the explicit description of the $G_{2h}$-action on $V_{h,2}$, we get that the fiber product $\cB \Gm \times_{\cB G_{2h}} \cE_{h,2} $ is isomorphic to the quotient stack $[V_{h,2}/\Gm]$ where $\Gm$ acts on the affine space $V_{h,2}$ with strictly positive weights. Therefore, we have that there exists a unique lifting of $j_{2h}$, corresponding to the $0$-section of $V_{h,2}$.
\end{proof}

\begin{remark}\label{rem:unique-curve-crimp-even}
   \Cref{lem:hyp-crimping-datum-even} implies that there exists a unique curve $(C,p)$ (constructed using the unique crimping datum given by \Cref{lem:hyp-crimping-datum-even}) such that the $\Gm$-action on $(\widetilde{C},q)$ descends to $(C,p)$.
\end{remark}

\subsubsection*{Odd case}
In analogy with the even case, we start by describing the moduli stack of curves with a prescribed normalization at an odd $A$-singularity. The following is a direct consequence of \Cref{prop:descr-an-disp-i} and \Cref{prop:descr-an-disp-ns}.
\begin{lemma}\label{cor:diff-crimping-same-curve-odd}
    Let $(\widetilde{C},q_1,q_2)$ be a $2$-pointed, not necessarily connected curve over an algebraically closed field. The fiber product
    $$
    \begin{tikzcd}
    {[G_h/\Aut(\widetilde{C},q_1,q_2)]} \arrow[r] \arrow[d] & {\cB G_h} \arrow[d] \\
    {\cB \Aut(\widetilde{C},q_1, q_2)} \arrow[r]                  & \cB G_h \times \cB G_h
    \end{tikzcd}
    $$
    parametrizes curves $(C,p)$, where $p$ is an $A_{2h-1}$-singularity such that the (pointed) partial normalization at $p$ is isomorphic to $(\widetilde{C},q_1,q_2)$.
\end{lemma}

\begin{definition}
    \label{def:G-equiv-crimping}
    Let $(\widetilde{C},q_1,q_2)$ be a $2$-pointed, not necessarily connected curve over an algebraically closed field and $G\subset \Aut(\widetilde{C},q_1,q_2)$ be a subgroup. A $G$-equivariant crimping datum is a crimping datum for the family $[(\widetilde{C},q_1,q_2)/G]\rightarrow \cB G$.
\end{definition}

\begin{remark}
\label{rem:diff-crimping-same-curve-odd}
A $G$-equivariant crimping datum
$$\operatorname{cr} \colon \cB G \to \cB G_h$$
exists if and only if one can construct an $A_{2h-1}$-singularity gluing $q_1$ and $q_2$ such that the subgroup $G \subset \Aut(\widetilde{C},q_1,q_2)$ descends to a subgroup of the automorphisms of $C$ along the normalization $ \widetilde{C}\to C$. As in the even case, different $G$-equivariant crimping data may give rise to the same curve. Indeed, two $G$-equivariant crimping data
$$ \operatorname{cr}_1, \operatorname{cr}_2: \cB G \to \cB G_h $$
will give rise to the same curve if and only if the two induced morphisms $\cB G \rightarrow \cB G_h \times_{\cB G_{h}\times \cB G_h} \cB\Aut(C,s)$ are isomorphic, as follows directly from \Cref{cor:diff-crimping-same-curve-odd}.
\end{remark}

A simple consequence of \Cref{cor:diff-crimping-same-curve-odd} is the following result. It will help us understand how the automorphism group changes when attaching rational curves using odd $A$-singularities.

\begin{lemma}\label{cor:quotient-stack-description}
     Let $h$ be a positive integer with $h\geq 2$. Let $\widetilde{C}/k$ be a (possibly pointed) curve over an algebraically closed field $k$ with one rational smooth point $q\in \widetilde{C}(k)$. Let $\cA_{2h-1}(\widetilde{C}\cup\bP^1)$ be the moduli stack parametrizing (possibly pointed) $A_r$-stable curves $C$ with a separating $A_{2h-1}$-singularity such that the pointed normalization of $C$ at $p$ is isomorphic to $(\widetilde{C}\sqcup \bP^1,p,\infty)$. Then $\cA_{2h-1}(\widetilde{C}\cup \bP^1)$ is isomorphic to the quotient stack $[(U_h/\bG_a)/\Aut(\widetilde{C},q)]$ where $U_h$ is an $\Aut(C,p)$-representation. If $\bP^1$ has another marking (say the point $0$), then the quotient stack is of the form $[U_h/\Aut(\widetilde{C},q)]$.
\end{lemma}

\begin{proof}
    We leave the $2$-pointed case to the reader, since the same proof works. By \Cref{cor:diff-crimping-same-curve-odd}, $\cA_{2h-1}(\widetilde{C}\cup\bP^1)$ is isomorphic to $[G_h/\Aut(\widetilde{C},q) \times \Aut(\bP^1,\infty)]$, where $h\geq 3$ because the curve is $A_r$-stable. To prove the statement, we need to study the explicit description of the morphism
    $$ \eta: \Aut(\bP^1,\infty)\simeq \cB (\bG_m \ltimes \bG_a) \longrightarrow \cB G_h.$$
     We identify an element $(\lambda,t)\in\bG_m\ltimes \bG_a$ with the automorphism $z\mapsto \lambda^{-1}(z-t)$ of $\bP^1$. This automorphism will act on the infinitesimal neighborhood of order $h-1$ of $\infty \in \bP^1$ by the formula
    $$ \xi \mapsto \frac{\lambda \xi}{1-\xi t}=\lambda \xi \sum_{i=0}^{m-2}(\xi t)^i=\lambda\xi+\lambda t\xi^2 +\lambda t^2\xi^3+\dots$$
    where $\xi$ is the local parameter of $\infty$ in $\bP^1$. Therefore, the morphism $\eta$ is induced by an explicit embedding $\bG_m \ltimes \bG_a \subset G_h\simeq \bG_m \ltimes U_h$ given by the formula
    $$ (\lambda,t) \mapsto (\lambda, \lambda t,\lambda t^2,\dots,\lambda t^{h-2})$$
    and the quotient is an affine space, namely $U_h/\bG_a$. Thus, we have the statement as claimed.
\end{proof}
\begin{remark}
    In the previous corollary, the $A_r$-stable condition on $C$ ensures that if $h=2$, then $\bP^1$ has to be $2$-pointed. Indeed, if this were not the case, there would be an additional copy of $\Ga$ in the group by which we want to quotient.
\end{remark}

We include the following remark, which will be useful for subsequent papers in the series.

\begin{remark}\label{rem:case-non-reductive}
    Suppose we are in the situation of \Cref{prop:descr-an-disp-i}, where $S=\emptyset$, $i=0$ and $h\geq 2$.  We have the affine crimping morphism
    $$ \cI:=\cI_{2h-1}^{0,\emptyset} \rightarrow \cM_{g-h+1,n,[h]} \times \cM_{0,0,[h]} $$
    where $\cM_{0,0,[h]}\simeq \cB(\bG_m\ltimes \bG_a)$ (see \Cref{prop:descr-base-crimping-even}). We can consider the diagram
    $$
    \begin{tikzcd}
    \cI \arrow[r] \arrow[d]                                        & {[(U_h/\Ga)/G_h]} \arrow[r] \arrow[d]          & \cB G_h \arrow[d]        \\
    {\cM_{g-h+1,n,[h]}^r \times \cM_{0,0,[h]}} \arrow[r] \arrow[d] & \cB G_h \times \cB (\Gm \ltimes \Ga) \arrow[d] \arrow[r] & \cB G_{h} \times \cB G_h \\
    {\cM_{g-h+1,n,[h]}^r} \arrow[r]                                & \cB G_h                                        &
    \end{tikzcd}
    $$
    where all the squares are cartesian thanks to \Cref{cor:quotient-stack-description} and \Cref{rem:descr-an-disp-i}.
    Thus, the composition
    $$ \cI \longrightarrow \cM_{g-h+1,n,[h]} \times \cB(\bG_m\ltimes \bG_a) \rightarrow \cM_{g-h+1,n,[h]} $$
    is still affine, since it is the pullback of an affine morphism.
\end{remark}

In analogy with the even case, we now characterize $\Gm$-equivariant crimping data. Since we are gluing the two infinitesimal neighborhoods together, the idea is that any cocharacter that descends through the normalization morphism should have balanced action on the tangent spaces of the points we are about to glue. This is summarized in the following result.

\begin{lemma}\label{lem:hyp-crimping-datum-odd}
    Let $h$ be a positive integer with $h\geq 2$. Let $\widetilde{C}/k$ be a not necessarily connected curve over an algebraically closed field $k$ with two rational smooth points $q_1,q_2\in \widetilde{C}(k)$. Let $\phi:\Gm \rightarrow \Aut^{\circ}(\widetilde{C},q_1,q_2)$ be a cocharacter and denote by $n_i$ the integer associated to the $1$-dimensional $\Gm$-representation $T_{q_i} \widetilde{C}$ for $i=1,2$. The following are equivalent:
    \begin{itemize}
        \item[(a)] $n_1=n_2$;
        \item[(b)] there exists a $\Gm$-equivariant crimping datum, i.e., there exists a dotted arrow making the following diagram
    $$
    \begin{tikzcd} & \cB G_h \arrow[d] \\
    \cB \Gm \arrow[r, "j_{h,1}\times j_{h,2}"'] \arrow[ru, "\exists ", dashed] & \cB G_{h} \times \cB G_h
    \end{tikzcd}
    $$
    commute, where $j_{h,i}$ is the morphism associated to the infinitesimal neighbourhood of order $h-1$ of $q_i$ for $i=1,2$.
    \end{itemize}
    Moreover, suppose that $n_1=n_2\neq 0$ and the cocharacter $\phi$ factors through a $2$-dimensional subtorus $\Gm^2 \subset \Aut^{\circ}(\widetilde{C},q_1,q_2)$. Then the curves obtained by gluing $p_1$ and $p_2$ with two different $\Gm$-equivariant crimping data are isomorphic.
\end{lemma}

\begin{proof}
    As in the proof of \Cref{lem:hyp-crimping-datum-even}, $p_1$ and $p_2$ lie in irreducible components (possibly the same) whose normalizations are copies of $\bP^1$. Therefore, because the morphisms \(j_{h,i}:\cB \bG_m\to \cB G_h\) are local, we can factor the morphism $j_{h,1}\times j_{h,2}:\cB \Gm \rightarrow \cB G_{h} \times \cB G_h$ through two copies of $\bG_m \subset \Aut(\bP^1,\infty)$. More precisely, we have a commutative diagram
    $$ \begin{tikzcd}
    \cB \Gm \arrow[r, "{(n_1,n_2)}"'] \arrow[rr, "j_{h_1} \times j_{h_2}", bend left] & \cB \Gm \times \cB \Gm \arrow[r] & \cB G_{h} \times \cB G_{h}
    \end{tikzcd}$$
    where the rightmost morphism can be explicitly computed as (two copies of) the morphism \(j_h:\cB \Gm\to \cB G_h\) induced by $\infty\in \mathbb{P}^1$ and associated to the cocharacter $ \Gm \subset \Aut^{\circ}(\bP^1,\infty) \simeq \Gm \ltimes\Ga$. Since an odd crimping datum is a lift of the morphism $j_{h_1}\times j_{h_2}$, any $\bG_m$-equivariant crimping datum should factor through $[G_{h}/\Gm \times \Gm]$. Recall that $G_{h} \simeq \Gm \ltimes U_h$ where $U_h$ is the unipotent radical, and the two $\Gm$ actions are respectively left and right multiplication for elements of the subgroup $\Gm \subset G_h$. Notice that the right multiplication is done after taking the inverse of the element in $\Gm$. Therefore, after pulling back the diagonal of $\cB G_h$ through the morphism $j_{h_1} \times j_{h_2}$, we get the quotient stack $[G_h/\Gm]$ where $\Gm$ acts on $G_h$ as follows:
    $$ \lambda.(u_1, u_2, \dots, u_{h-1}) := (\lambda^{n_1-n_2}u_1, \lambda^{2n_1-n_2}u_2, \dots, \lambda^{(h-1)n_1-n_2}u_{h-1}) $$
    where $u_1 \in \Gm$ and $(u_2,\dots,u_{h-1})\in U_h$. Thus $[G_h/\Gm] \rightarrow \cB \Gm$ has a section only if $n_1=n_2$; otherwise, $[G_h/\Gm]$ only has finite stabilizers, thus $(b) \implies (a)$. Furthermore, if $n_1=n_2$, then we have $[G_h/\Gm]\simeq \Gm \times [U_h/\Gm]$ where $U_h$ is a $\Gm$-representation with nonnegative weights,
    thus proving $(a)\implies (b)$.

    Finally, we prove the concluding assertion. Since the cocharacter $\phi$ factors through $2$-dimensional subtorus $\Gm^2 \subset \Aut^{\circ}(\widetilde{C},q)$, \Cref{rem:diff-crimping-same-curve-odd} tells us that two crimping data give rise to isomorphic curves if the induced morphisms
    $$ c_1,c_2:\cB \Gm \longrightarrow (\cB \Gm \times \cB \Gm) \times_{(\cB G_h\times \cB G_h)} \cB G_h$$
    are isomorphic. Nevertheless, the explicit nature of the morphism gives that
    $$  (\cB \Gm \times \cB \Gm) \times_{(\cB G_h\times \cB G_h)} \cB G_h \simeq [G_h/\Gm \times \Gm] \simeq [U_h/\Gm]$$
    where $\Gm$ acts on $U_h$ as described above. Since $n_1=n_2 \neq 0$, we have that $U_h$ is a $\Gm$-representation with strictly positive weights; therefore, there exists a unique point with $\Gm$-stabilizers. The result follows.
\end{proof}

\begin{remark}\label{rem:at-most-one-Gm}
    Notice that \Cref{lem:hyp-crimping-datum-odd} tells us that if we start with a curve $(\widetilde{C},q_1,q_2)$ with two copies of $\Gm\subset \Aut(\widetilde{C},q_1,q_2)$ acting independently on $T_{q_1}\widetilde{C}$ and $T_{q_2}\widetilde{C}$, then at most one copy of $\Gm$ will survive after we glue $q_1$ and $q_2$. Moreover, if such a copy survives, the induced restriction to the normalization can be identified with the diagonal embedding $\Gm \subset \Gm^2$.
    Note further that \Cref{lem:hyp-crimping-datum-odd} is stated for $h\geq 2$; thus, it does not deal with nodes. Indeed, when gluing two points using nodes, there is no condition $n_1=n_2$, and thus both copies of $\Gm$ (if they exist) will survive: see \Cref{rem:aut-and-normalization-along-nodes}.
\end{remark}
\begin{remark}
    It is important to stress that in the final statement of \Cref{lem:hyp-crimping-datum-odd}, the conditions on the cocharacter $\phi$ are necessary. Indeed, consider two disjoint copies of $(\bP^1,0,\infty)$: the automorphism group of the disjoint union is $\Gm^2$. For simplicity, we can assume $h=2$, thus we are dealing with tacnodal singularities. If we apply \Cref{lem:hyp-crimping-datum-odd} to the diagonal cocharacter $\Gm \rightarrow \Gm^2$ gluing the two copies at $0$, there exists a unique curve $\widetilde{C}$ which has $\Gm$ as (the connected component of the identity of) the group of automorphisms: it is constructed by gluing the two projective lines at $0$, and $\Gm$ acts as scalar multiplication on both components. Suppose that we want to glue the two irreducible components at $\infty$ as well, i.e., we want to apply \Cref{lem:hyp-crimping-datum-odd} to $(\widetilde{C},q_1,q_2)$ where $q_1,q_2$ are the points $\infty$ in the two irreducible components. Following the proof of \Cref{lem:hyp-crimping-datum-odd}, we have that the moduli stack of curves constructed in this way is isomorphic to $\Gm \times [U_h/\Gm]$ (see also \Cref{rem:diff-crimping-same-curve-odd}). Therefore, there is a family of non-isomorphic curves over $\Gm$ constructed with different $\Gm$-equivariant crimping data. The final statement of \Cref{lem:hyp-crimping-datum-odd} fails because the curve $(\widetilde{C},q_1,q_2)$ does not have a copy of $\Gm^2$ as a subtorus.
\end{remark}

\Cref{lem:hyp-crimping-datum-odd} has the following interesting consequence.
\begin{corollary}
    Let $(C,p_1,\dots,p_n)$ be an $n$-pointed curve of genus $g$. Let $\Gamma$ be an irreducible component and $\bG_m\rightarrow \Aut^{\circ}(C,p_1,\dots,p_n)$ be a cocharacter such that the induced morphism $\Gm \rightarrow \Aut^{\circ}(\Gamma,P_{\Gamma}\cup I_{\Gamma})$ is not trivial. Then the same is true for every component $\Gamma'$ meeting $\Gamma$ in (at least one) worse-than-nodal odd $A$-singularity.
\end{corollary}

\subsection{Hyperelliptic \texorpdfstring{$A_r$}{A\_r}-stable curves}\label{sub:hyper-A-stable}

As anticipated, we conclude the section by focusing our attention on hyperelliptic $A_r$-stable curves. We start with the definition.

\begin{definition}
    A hyperelliptic $n$-pointed curve is a tuple $(C,p_1,\dots,p_n,\sigma)$ where $(C,p_1,\dots,p_n)$ is an $n$-pointed curve and $\sigma$ is a hyperelliptic involution, i.e., a two-torsion element in $\Aut(C)$ such that
    \begin{itemize}
        \item the scheme of fixed points of $\sigma$ is finite;
        \item the geometric quotient $C/\sigma$ is a nodal curve of genus $0$.
    \end{itemize}
    We say that a (possibly singular) point $q \in C$ is a \emph{Weierstrass} point if it is fixed by the involution.
\end{definition}

We define $\cH_{g,n}^r$ as the following fibered category: its objects consist of the data of a pair $(C/S,\sigma,p_1,\dots,p_n)$ where $(C,p_1\dots,p_n)$ is an $A_r$-stable curve over $S$ and $\sigma$ is an involution of $C$ over $S$ such that $(C_s,\sigma_s)$ is an $A_r$-stable hyperelliptic curve of genus $g$ for every geometric point $s \in S$. These are called \emph{hyperelliptic $n$-pointed $A_r$-stable curves over $S$}. A morphism is a morphism of $n$-pointed $A_r$-stable curves that commutes with the involutions. We clearly have a morphism of fibered categories
$$\eta:\cH_{g,n}^r  \longrightarrow \cM_{g,n}^r$$
over $\kappa$ defined by forgetting the involution. This morphism is known to be a closed embedding for $r\leq 1$. Recall that by convention $\cM_{g,n}^0$ coincides with the moduli stack of $n$-pointed smooth curves of genus $g$.

\begin{theorem}
     The morphism $\eta$ is a closed embedding for every $r$. In fact, we can identify $\cH_{g,n}^r$ with the closure of the locally closed substack $\cH_{g,n}^0 \subset \cM_{g,n}^0 \subset \cM_{g,n}^r$ parametrizing smooth $n$-pointed hyperelliptic curves of genus $g$.
\end{theorem}

\begin{proof}
     Indeed, we have that $\cH_{g,n}^r$ coincides with the fiber product $\cH_g^r \times_{\cM_{g}^r}\cM_{g,n}^r$ where the morphism
     $$ \cM_{g,n}^r \longrightarrow \cM_g^r$$
     can be constructed using \cite[Theorem~2.5]{Per1}, which states that \(\cM_{g,n}^r\) is an open substack of the universal \(n-1\)-pointed \(A_r\)-stable curve. Thus, the statement follows from \cite[Section~$3$]{Per2}, which proves that \(\cH_g^r\to \cM_{g}^r\) is a closed immersion.
\end{proof}

To every $n$-pointed hyperelliptic curve, we can associate the geometric quotient $C\rightarrow Z:=C/\sigma$, where $Z$ is a nodal curve of genus $0$.

\begin{definition}
    If $Z$ is irreducible, then we say that $C$ is an $n$-pointed \emph{honestly hyperelliptic} $A_r$-stable curve, namely a cyclic cover of $\bP^1$ of degree $2$.
\end{definition}

The terminology is borrowed from \cite{Cat}, specifically Definition 3.18. The substack of $\cH_{g,n}^r$ classifying honestly hyperelliptic curves, $\cH_{g,n}^{r,\circ}$, is an open substack inside $\cH_{g,n}^r$. The rest of this section is dedicated to giving a description of $\cH_{g,n}^{r,\circ}$ as a quotient stack in some specific cases we will need. The reason is that these cases are the building blocks for the points in $\cM_{g,n}^r$ that have nontrivial positive-dimensional stabilizers, and understanding how they behave with respect to isotrivial degenerations is fundamental for the understanding of the local picture of the moduli stack we are interested in.

We start with the case $n=0$. We have the following identification:
$$ \cH_{g}^{r,\circ} \simeq [U_{r+1}/(\GL_2/\mu_{g+1})]$$
where $U_{r+1}$ is an open subset of the representation $\Sym^{2g+2}E$, where $E$ is the standard representation of $\GL_2$. This directly follows from \cite[Theorem~4.1]{ArVis}. The open $U_{r+1}$ parametrizes homogeneous binary forms $f \in \Sym^{2g+2}E$ of degree $2g+2$ such that they do not have roots of multiplicity strictly greater than $r+1$. We are identifying every honestly hyperelliptic curve with the branching divisor of length $2g+2$ in $\bP^1$. A point of length $h$ in this divisor corresponds to an $A_{h-1}$-singularity in the (honestly) hyperelliptic curve. Notice that $U_{r+1}$ never contains the $0$-section of the representation, as it corresponds to a non-reduced cover of $\bP^1$, and it is equal to $\Sym^{2g+2}E \setminus 0$ if $r=2g+1$. The question we are interested in is the following: for which $r$ does $\cH_{g}^{r,\circ}$ admit a good moduli space?

One attempt to solve this problem uses GIT. The moduli stack $\cH_g^{2g+1,\circ}$ is strongly related to the quotient $[\bP(\Sym^{2g+2}E)/\PGL_2]$ where the action of $\PGL_2$ on $E$ is the standard one: there is a morphism
$$ \cH_g^{2g+1,\circ} \rightarrow [\bP(\Sym^{2g+2}E)/\PGL_2]$$
which is a $\mu_2$-gerbe, whose triviality depends on the parity of $g$. Therefore, it is enough to study the open substacks of the right-hand side of the morphism which admit a separated good moduli space.

If we apply the GIT machinery, more specifically the Hilbert-Mumford criterion, we find that the open locus of semistable points corresponds to $U_{g+1}$, namely the binary forms without roots of multiplicity greater than $g+1$. Although this is classical, the authors learned about it in \cite[Proposition~7.12]{Vicnote}. Moreover, $U_r$ consists of only stable points if $r\leq g$. We claim that the GIT answer is somehow ``maximal''. By abuse of notation, we denote by $U_r$ the projectivization of the open subset of $\Sym^{2g+2}E\setminus 0$ and by $\cU_r$ the induced open substack of $\cU_{2g+2}:=[\bP(\Sym^{2g+2}E)/\PGL_2]$.
\begin{theorem}\label{theo:classify-opens-hyp}
    In the situation above, let $\cU \subset \cU_{2g+2}$ be any open substack. Then $\cU$ admits a separated good moduli space if and only if one of the following holds:
    \begin{itemize}
        \item[(i)] $\cU\subset \cU_g$;
        \item[(ii)] $\cU \subset \cU_{g+1}$ and $\cU_{g+1}\setminus \cU_g \subset \cU$.
    \end{itemize}
\end{theorem}
\begin{proof}

First, notice that if either $(i)$ or $(ii)$ is satisfied, then the result follows from GIT machinery. In fact, $\cU_{g+1}$ (respectively $\cU_{g}$) coincides with the locus of semistable (respectively stable) points. Because $\cU_{g}$ is a separated DM stack, then $\cU\subset \cU_{g}$ implies that $\cU$ is also a separated DM stack.
Suppose then $\cU_{g+1}\setminus \cU \subset \cU_g$. Then \Cref{prop:local-Theta-complete} gives us that the open embedding $\cU\subset \cU_{g+1}$ is $\textsf{S}$-complete and $\Theta$-complete, because the complement $\cZ:=\cU_{g+1}\setminus \cU$ does not have positive-dimensional stabilizers. Because $\cU_{g+1}$ admits a separated good moduli space, so does $\cU$.

It remains to show that whenever $\cU$ is $\Theta$-complete, $\textsf{S}$-complete, and $\cU \not\subset \cU_g$, then $(ii)$ holds. First, we want to prove that for such an $\cU$, we have $\cU \subset \cU_{g+1}$. First, we notice that $\cU_{2g+2}\setminus \cU_{2g+1}$ consists of one (closed) point with non-reductive stabilizer. This can be seen by considering the morphism
$$\bP(E) \longrightarrow \bP(\Sym^{2g+2}E)$$
given by the association $f(x_0,x_1)\mapsto f^{2g+2}$, where $f$ is a binary linear form in $E$. The morphism is a closed immersion and identifies $\cU_{2g+2}\setminus \cU_{2g+1}$ with the quotient $[\bP^1/\PGL_2] \simeq \cB(\bG_m\ltimes \bG_a)$. Therefore, $\cU\subset \cU_{2g+1}$. Now, the only points that have positive-dimensional stabilizers in $\cU_{2g+1}$ are the binary forms $f_h$ having exactly two (set-theoretic) roots, one of multiplicity $h$ and the other of multiplicity $2g+2-h$. These points are closed in $\cU_{2g+1}$, as there are no further degenerations in $\cU_{2g+1}$. Their stabilizer groups are isomorphic to $\bG_m$, as it is the automorphism group of $\bP^1$ with two distinct points. We claim that $f_h \in \cU  \implies h=g+1$. Without loss of generality, suppose $h<g+1$. We can consider the following morphism
$$ [\bP(E)/\bG_m]\longrightarrow \cU \subset [\bP(\Sym^{2g+2}E)/\PGL_2]$$
defined by the association $f(x_0,x_1) \mapsto f^{2(g+1-h)}(x_0x_1)^h$, where $\bG_m \subset \PGL_2$ is the diagonal torus.  A standard computation shows that the morphism is finite: indeed, it is a $\mu_2$-torsor where $\mu_2$ acts on $E$ by exchanging $x_0$ and $x_1$, followed by a closed embedding. Since the map factors through the open \(\cU \subset [\bP(\Sym^{2g+2}E)/\PGL_2]\), we get a contradiction because $[\bP(E)/\bG_m]$ does not have a good moduli space (see \cite[Lemma~4.14]{Alp}).

Therefore, $f_h \notin \cU$ if $h\neq g+1$. We claim that this implies that the complement $\cZ$ contains $\cZ_h$, namely the locally closed subspace classifying binary forms with one root of multiplicity exactly $h$, when $h>g+1$, i.e., $\cU\subset \cU_{g+1}$.

Because $\cU$ is $\textsf{S}$-complete and $\Theta$-complete, so is the open embedding $\cU\subset \cU_{2g+1}$, since $\cU_{2g+1}$ has affine diagonal. This implies that the complement $\cZ$ of $\cU$ in $\cU_{2g+1}$ satisfies the condition of \Cref{prop:local-Theta-complete}. Let us apply it in the case where $x$ is the geometric point associated to $f_h$. First of all, notice that we know that the space of deformations of (the curve associated to)  $f_h$ can be described as the direct sum
$$ \Def_{f_h}=V_h\oplus V_{2g+2-h}$$
where $V_h$ (respectively $V_{2g+2-h}$) is the space that preserves the root of multiplicity $h$ (respectively $2g+2-h$). Moreover, such decomposition is not only $\bG_m$-equivariant, but sign-preserving, meaning that the irreducible representations inside $V_h$ have positive weights while the ones inside $V_{2g+2-h}$ have negative weights, or vice versa. This follows from the fact that $\bG_m$ acts with weights of opposite signs on the tangent spaces of $0$ and $\infty$. Therefore, \Cref{prop:local-Theta-complete} allows us to conclude that $\cZ$ either contains $V_h$ or $V_{2g+2-h}$, locally around $f_h$. This actually holds globally because of the following observation: every binary form which has a root of multiplicity $h$ admits an isotrivial degeneration to $f_h$; thus, the local structure (considered in the proof of \Cref{prop:local-Theta-complete}) of $\cU_{2g+1}$ at $f_h$ is surjective onto the locally closed $\cZ_h\subset \cZ$. Thus, we have that $\cZ$ contains either $\cZ_h$ or $\cZ_{2g+2-h}$ globally. Finally, we notice that $\cZ_h$ is contained in the closure of $\cZ_{2g+2-h}$ when $h>g+1$. Because $\cU$ is an open, in both cases $\cU$ will not intersect $\cZ_h$ for every $h>g+1$. Therefore, we have that $\cU\subset \cU_{g+1}$.

Suppose now that $f_{g+1} \notin \cU$. Then applying \Cref{prop:local-Theta-complete} locally on $f_{g+1}$ as before, we can prove that $\cZ_{g+1}$ is contained in the complement of $\cU$, thus giving $\cU\subset \cU_{g}$. Finally, it remains to prove that
$$ f_{g+1} \in \cU \iff \cZ_{g+1}\subset \cU.$$
This follows again from the fact that $\cU$ is open and every binary form in $\cZ_{g+1}$ admits an isotrivial degeneration to $f_{g+1}$.
\end{proof}

\begin{remark}
    The previous proof relies heavily on the fact that the positive-dimensional stabilizers, after removing the non-reductive ones, are one-dimensional tori. This allows us to extend the results from the local picture to the global one, giving us a way to classify the opens which admit a separated good moduli space.
\end{remark}

In the same spirit, we can answer the open question in \cite[Remark~3.11]{codogni2021some}. Indeed, using the notation of \cite[Figure~1]{codogni2021some}, we construct two isotrivial degenerations from $C_{nnn}$ to $C_{nt}$ which are not isomorphic, giving a finite map from \([\mathbb{P}^1/\mathbb{G}_m]\).

\begin{proposition}\label{prop:Viviani-answer}
    Using the notation of \emph{loc.cit.}, the stack $\overline{\cM}^{\rm irr}_2 \subset \cM^3_2$ does not admit a good moduli space.
\end{proposition}
\begin{proof}
     We can argue in the same way as in \Cref{theo:classify-opens-hyp}. First, notice that we have a finite $\mu_2$-gerbe $\varphi$:
     \[
     \overline{\cM}^{\rm irr}_2 = \cH_{2}^{3,\circ} \simeq [U_4/(\GL_2/\mu_3)] \xrightarrow{\varphi} \cU_4 \subset [\bP(\Sym^{6}E)/\PGL_2].
     \]
    We can therefore consider the finite map constructed in the proof of \Cref{theo:classify-opens-hyp}:
    \begin{align*}
        \phi \colon [\bP(E)/\Gm] &\to \cU_4 \subset [\bP(\Sym^6 E)/\PGL_2] \\
        f &\mapsto (x_0x_1)^2 f^2
    \end{align*}
    where $E$ is the dual of the standard representation of $\GL_2$.
    Therefore, $\overline{\cM}^{\rm irr}_2$ does not have a good moduli space because it has a finite map from a stack that does not admit a good moduli space.
\end{proof}

 We move to the case $n=1$. We introduce the following stratification of $\cH_{g,1}^{r,\circ}$: denote by $\cH_{g,w}^{r,\circ} \subset \cH_{g,1}^{r,\circ}$ the closed substack parametrizing pairs $(C,p,\sigma)$ such that $p$ is a (smooth) Weierstrass point, i.e., a point fixed by the involution. We denote by $\cH_{g,g_1^2}^{r,\circ}$ its open complement. Notice that $\cH_{g,g_1^2}^{r,\circ}$ can be identified with the moduli stack parametrizing $A_r$-stable honestly hyperelliptic curves of genus $g$ with two (smooth and distinct) markings which are exchanged by the involution. We end the section with the description of these substacks. This is Proposition 3.9 of \cite{AlpSmyVdW}.

\begin{proposition}\label{prop:hyp-global-descr}
    If $r\geq 2g$, the moduli stack $\cH_{g,w}^{r,\circ}$ is isomorphic to the quotient $[\bA^{2g-1}/\bG_m]$, where $\bG_m$ acts on $\bA^{2g-1}$ with positive weights. Moreover, if $r\geq 2g+1$, the moduli stack $\cH_{g,g_1^2}^{r,\circ}$ is isomorphic to the quotient $[\bA^{2g+1}/\bG_m]$, where $\bG_m$ acts on $\bA^{2g+1}$ with positive weights.
\end{proposition}

\begin{corollary}
    If $r\geq 2g$, then the morphism $\cH_{g,w}^{r,\circ} \rightarrow \spec \kappa $ is a good moduli space morphism. Moreover, if $r\geq 2g+1$, the morphism $\cH_{g,g_1^2}^{r,\circ} \to \spec \kappa$ is a good moduli space morphism.
\end{corollary}

\begin{remark}\label{rem:1-pointed-hyp-no-gms}
    The stack $\cH_{g,1}^{2g+1,\circ}$ does not admit a separated good moduli space. Namely, it fails to be $\Theta$-complete. Indeed, we have the stratification $\cH_{g,w}^{2g+1,\circ}\subset \cH_{g,1}^{2g+1,\circ}$ where the open stratum $\cH_{g,g_1^2}^{2g+1,\circ}$ has a unique closed point, and every other point of the open stratum admits an isotrivial degeneration to the closed point. The same is true for the closed stratum. Because the closed stratum is contained inside the closure of the open stratum, $\Theta$-completeness would imply that the closed point of the closed stratum should be contained inside the closure of the closed point of the open stratum. In the second paper of the series, we will show that this is not possible, which yields a contradiction.
\end{remark}

We conclude with the following corollary, which explains why we are interested mainly in the cases above.
\begin{corollary}\label{cor:cyc-cover-with-pos-dim}
    Let $(C,p_1,\dots,p_n,\sigma) \in \cH_{g,n}^{r,\circ}$ and suppose $n\geq 1$ and $g\geq 1$. If $\Aut(C,p_1,\dots,p_n)$ has positive dimension, then either
    \begin{itemize}
        \item[(1)] $n=2$, $r\geq 2g+1$, the two points are exchanged by the involution, and $C$ is the curve corresponding to the $0$-section of $[\bA^{2g}/\bG_m] \simeq \cH_{g,g_1^2}^{r,\circ}$;

        \item[(2)] $n=1$, $r \geq 2g+1$, $p_1$ is not a Weierstrass point, and $C$ is the curve corresponding to the $0$-section of $[\bA^{2g}/\bG_m] \simeq \cH_{g,g_1^2}^{r,\circ}$, where we are identifying $\cH_{g,g_1^2}^{r,\circ}$ with the open subset of $\cH_{g,1}^{r,\circ}$ where the marking is not Weierstrass;
        \item[(3)]  $n=1$, $r \geq 2g$, $p_1$ is a Weierstrass point, and $C$ is the curve corresponding to the $0$-section of $[\bA^{2g-1}/\bG_m] \simeq \cH_{g,w}^{r,\circ}$;
    \end{itemize}
\end{corollary}

\begin{proof}
    Notice that if $n\geq 3$, then a $2:1$-cover of $\bP^1$ with at least $3$ fixed points has finite automorphism group. Indeed, they will correspond to at least two points on $\bP^1$, which means the automorphism group is a subgroup of $\bG_m$. Moreover, since the automorphism would also have to fix the ramification of the $2:1$ cover, it will clearly be finite. Thus, we can assume $n\leq 2$.

    If $n=2$, and the two points are not exchanged by the hyperelliptic involution, we are in the same situation as above. Thus, we end up in case $(1)$ thanks to \Cref{prop:hyp-global-descr}.

    If $n=1$, we can use the stratification of $\cH_{g,1}^{r,\circ}$ by the closed substack $\cH_{g,w}^{r,\circ}$ and its open complement $\cH_{g,g_1^2}^{r,\circ}$. Then again, we can use \Cref{prop:hyp-global-descr} to conclude.
\end{proof}

\begin{remark}
    Cases (1) and (3) of \Cref{cor:cyc-cover-with-pos-dim} are the building blocks for everything that will appear from now on. We will refer to these curves as \emph{atoms}, c.f. \Cref{def:even-atoms,def:odd-atoms}.
\end{remark}
\section{Families of curves over \texorpdfstring{$[\bA^1/\Gm]$}{[A1/Gm]}}\label{sec:pos-dim-stab}

In this section, we describe all possible $n$-pointed $A_r$-stable curves with positive-dimensional stabilizers, extending the classification carried out in \cite{hassett2013log} and \cite{Viviani} for $r \geq 3$. We start by understanding the \emph{atomic} cases, treating the even and odd cases separately. These cases are explained in \Cref{sub:atoms}. Subsequently, in \Cref{sub:geometric-points}, we characterize the subcurves that contribute non-trivially to the dimension of the automorphism group of the whole curve, and we describe the connected component of the identity of the automorphism group of $A_r$-stable curves. Furthermore, we study how the automorphism group acts on the deformation space of a curve, specifically on the deformation space of $A$-singularities and crimping spaces, which is essential for understanding isotrivial degenerations. This is done in \Cref{sub:def-theory}. Finally, we prove that the local description of the isotrivial degenerations given by deformation theory can be globalized. This is explained in detail in \Cref{sub:describing-attractors}.

\begin{definition}\phantomsection\label{def:hyp-even-sing}
    An $A_m$-singularity $p$ of an $n$-pointed $A_r$-stable curve $(C,p_1,\dots,p_n)$ with $m\geq 2$ is called \emph{hyperelliptic} if there exists a cocharacter $\Gm \rightarrow \Aut^{\circ}(C,p_1,\dots,p_n)$ such that $\Gm$ acts non-trivially on $T_pC$.
\end{definition}

The reason we call these singularities \emph{hyperelliptic} is explained in \Cref{lem:even-atom} and \Cref{lem:odd-atom}.

Before starting our analysis of isotrivial degenerations of $A_r$-stable curves, we encourage the reader to read \Cref{sub:equivariant-normalization}, where we discuss when a cocharacter of the automorphism group of a curve is preserved when creating an $A$-singularity. Moreover, take the time to recall the ``propagation of the action'' phenomenon: if $\Gamma$ is an irreducible component with a nontrivial $\Gm$-action induced by a cocharacter of the automorphism group of the whole curve $C$, then any other component $\Gamma'$ intersecting $\Gamma$ in a worse-than-nodal singularity will have a nontrivial $\Gm$-action (induced by the same cocharacter). This will turn out to be fundamental in understanding which curves have positive-dimensional stabilizers.

\subsection{Atoms}\label{sub:atoms}

We start by finding examples of curves with positive-dimensional stabilizers. These curves are crucial to our analysis of the valuative criterion for the existence of a separated good moduli space, since every non-degenerate isotrivial degeneration of curves has a special fiber with positive-dimensional stabilizers.

In this subsection, we assume that every curve is defined over an algebraically closed base field. We study the even case first.

\begin{remark}\label{rem:even-atoms}
    Consider the $1$-pointed curve $(\bP^1,\infty)$ and suppose we want to crimp $\infty$ into an $A_{2h}$-singularity, with $h\geq 1$. By \Cref{cor:diff-crimping-same-curve-even}, we know that the moduli stack parametrizing curves $(C,p)$ whose (pointed) partial normalization at $p$ is $(\bP^1,\infty)$ can be described as $[V_{h,2}/\Gm\ltimes \Ga]$, where the $\Gm \ltimes \Ga$-action can be described explicitly using the morphism $\Aut(\bP^1,\infty) \rightarrow G_{2h}$. A straightforward computation shows that $\Ga$ acts freely on $V_{h,2}$, thus the moduli stack is isomorphic to the quotient $[(V_{h,2}/\Ga) / \Gm]$, where $V_{h,2}/\Ga$ is a $\Gm$-representation with strictly positive weights (see for instance the proof of \Cref{lem:hyp-crimping-datum-even}). Therefore, there exists a unique curve obtained by crimping $\bP^1$ at $\infty$ whose $A_{2h}$-singularity is hyperelliptic. It is easy to see that the curve $C$ obtained is $A_r$-stable if and only if $h\geq 2$ (see \Cref{rem:genus-count}).

    The same construction can be done starting with $(\bP^1,0,\infty)$, crimping at $\infty$ and remembering the marking $0$: it will give rise to a $1$-pointed $A_r$-stable curve (for $h\geq 1$). The automorphism group of both these curves is a one-dimensional torus.
\end{remark}

\begin{definition}[Even atom]
    \label{def:even-atoms}\index{even atom}
    The curve constructed in \Cref{rem:even-atoms} will be called the \emph{even atom} of genus $h$. If it contains a marking, we will call it the $1$-pointed even atom.
\end{definition}

\begin{figure}[H]
        \centering
        \includegraphics[width=0.4\textwidth]{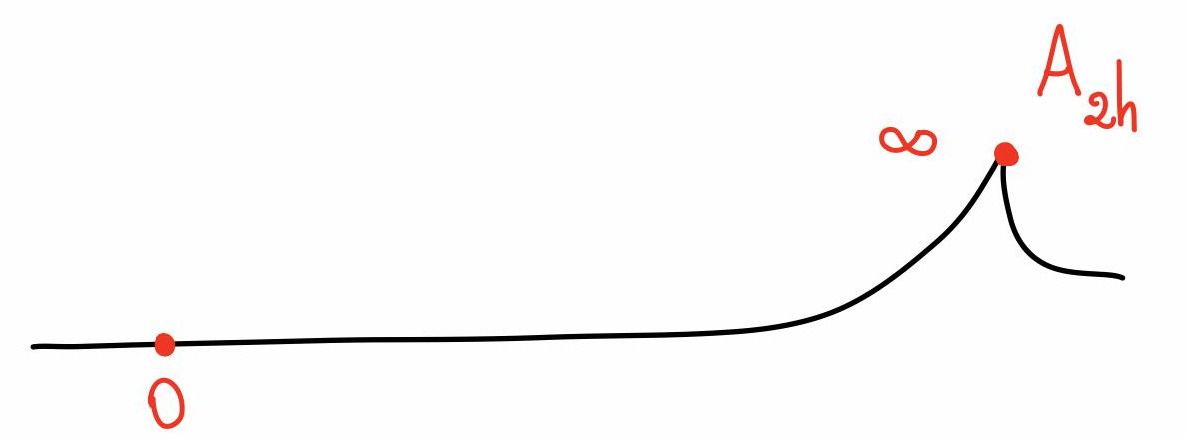}
        \caption{Even atom}
        \label{fig:even-atom}
\end{figure}

\begin{remark}
    The term \emph{atom} was first used in \cite{AlpFedSmyWyck}. These are the two building blocks (even and odd) of the $A_r$-stable curves with positive-dimensional stabilizers. As pointed out in \Cref{rem:motivation-hyp-sing}, it is clear that if an \(A_r\)-stable curve $C$ has positive-dimensional automorphisms, it must contain some components of geometric genus $0$ with fewer than $3$ special points, in contrast to what is required for nodal stable curves. This is why we will study how to construct $A$-prestable curves using projective lines, as they are the only possible components that contribute to the positive-dimensional automorphisms.
\end{remark}

We now explain why we used the term \emph{hyperelliptic} for the singularities. Namely, the even atom of genus $g$ is actually the only $A_r$-stable curve with an $A_{2g}$-singularity that is hyperelliptic, and it is in fact honestly hyperelliptic. Before showing this, let us introduce the following notation.

\begin{definition}\label{def:stack with at least one A_h singularity}
    For any \(0\leq h\leq r\), let \(\cD^h\subset \cM_{g,n}^r\) be the substack parametrizing curves with at least one singularity of type \(A_h\) in each geometric fiber.
\end{definition}

It follows from \cite[Section 4]{Per1} that \(\cD^h\) is the set-theoretic image of the natural morphism \(\cA_h\to  \cM_{g,n}^r\) coming from the description of \(\cA_h\) as a substack of the universal curve. Moreover, by \textit{loc.\ cit.}, the inclusion \(\cD^h\subset \cM_{g,n}^r\) is a locally closed immersion, and it is closed exactly when $h=r$.

\begin{lemma}\label{lem:even-atom}
    Suppose $n\leq 1$ and $r\geq 2g$. The intersection $\cD^{2g} \cap \cH_{g,n}^r\subset \cM_{g,n}^r$, parametrizing hyperelliptic curves with at least one singularity of type \(A_{2g}\), coincides with the $1$-pointed even atom of genus $g$. Moreover, if $n=1$, the marking must be a Weierstrass point.
\end{lemma}

\begin{proof}
     We start by studying the case $n=0$. Suppose that $C \in \cH_g^{r}\cap \cD^{2g}$. First of all, because of the absence of markings, the curve $C$ must be irreducible (see \Cref{rem:genus-count}); thus, it lies in $\cH_{g}^{r,\circ}$, namely it is a cyclic cover of $\bP^1$ of degree $2$. The ramification divisor of the morphism $C\rightarrow \bP^1$ has length $2g+2$, and \cite[Corollary~1.4]{Per2} shows that it has length $2g+1$ on the $A_{2g}$-singularity. This implies that the branching divisor is (set-theoretically) supported at two distinct points of $\bP^1$. Using the description of $\cH_g^{r,\circ}$ as cyclic covers of degree $2$ of $\bP^1$ (see \cite{ArVis} or \Cref{sub:hyper-A-stable}), one can prove that $\Aut(C)\simeq \bG_m$ when the branching divisor is supported at two distinct points, namely because $\Aut(\bP^1,0,\infty)\simeq \bG_m$. Notice that \Cref{rem:even-atoms} implies that there exists a unique curve such that
     \begin{itemize}
         \item it has an $A_{2g}$-singularity,
         \item the partial normalization at the singularity is $\bP^1$,
         \item it has positive-dimensional automorphism group;
     \end{itemize}
     and thus the statement follows. We leave it to the interested reader to prove the case $n= 1$.
\end{proof}

We now deal with the odd case. From now on, for convenience, we will consider $A_{2h+1}$-singularities (instead of $A_{2h-1}$); thus, the careful reader will encounter a shift in the numerical conditions on $h$.

\begin{remark}\label{rem:odd-atoms}
      Let $h\geq 1$. Consider the curve $(\widetilde{C},q_1,q_2)$ defined as the disjoint union of two copies of $(\bP^1,\infty)$. By \Cref{cor:diff-crimping-same-curve-odd} or \Cref{cor:quotient-stack-description}, we know that the moduli stack parametrizing curves $(C,p)$ such that
      \begin{itemize}
          \item $p$ is an $A_{2h+1}$-singularity,
          \item the (pointed) partial normalization of $C$ at $p$ is $(\widetilde{C},q_1,q_2)$
      \end{itemize}
    is isomorphic to the quotient $[(U_{h+1}/\Ga) /\Gm\ltimes \Ga]$, where $U_{h+1}/\Ga$ is a $\Gm$-representation with strictly positive weights. Therefore, there exists a unique curve whose $A_{2h+1}$-singularity is hyperelliptic (see also \Cref{lem:hyp-crimping-datum-odd}). It is easy to see that the curve $C$ is $A_r$-stable if and only if $h\geq 2$ (see \Cref{rem:genus-count}). Moreover, one can show easily that $\Aut^{\circ}(C)\simeq \Gm\ltimes \Ga$.

     The same construction can be carried out considering $0$ as a marking for one (or both) projective lines: it will give rise to a $1$-pointed (or $2$-pointed) $A_r$-stable curve for $h\geq 2$ (for $h\geq 1$). There is a (non-unique choice of) $\Gm$-equivariant crimping datum, thanks to \Cref{lem:hyp-crimping-datum-odd}. Moreover, any choice of such a $\Gm$-equivariant crimping datum gives rise to isomorphic curves. The (connected component of the identity of the) automorphism group of the pointed curves is a one-dimensional torus.
\end{remark}

\begin{definition}[Odd atom]\index{Odd atom}
    \label{def:odd-atoms}
    The curve constructed in \Cref{rem:odd-atoms} will be called the \emph{odd atom} of genus $h$. If it contains $n$ markings with $n\leq 2$, we will denote it as the $n$-pointed odd atom.
\end{definition}

\begin{figure}[H]
        \caption{}
        \centering
        \includegraphics[width=0.4\textwidth]{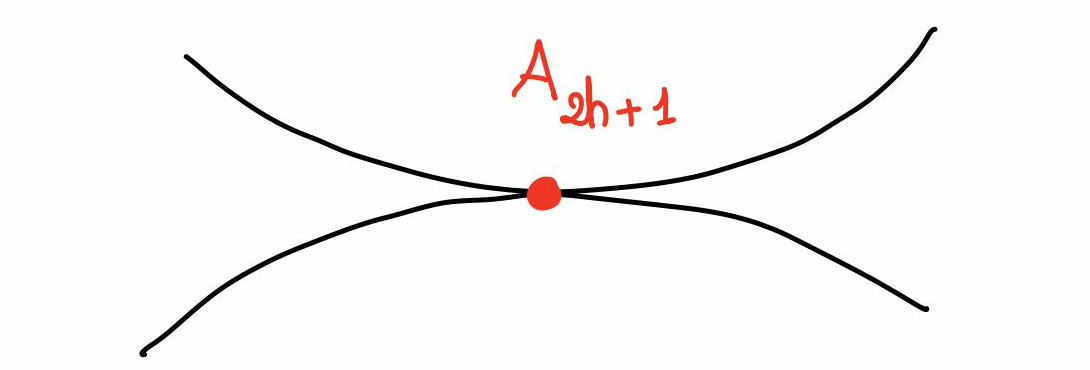}
        \label{fig:odd-atom}
    \end{figure}

As in the even case, the odd atom of genus $g$ is actually the only $A_r$-stable curve with an $A_{2g+1}$-singularity that is hyperelliptic, and it is in fact honestly hyperelliptic.

\begin{lemma}\label{lem:odd-atom}
    Suppose $n\leq 2$ and $r\geq 2g+1$. The intersection $\cD^{2g+1} \cap \cH_{g,n}^r\subset \cM_{g,n}^r$ is the $n$-pointed odd atom. Moreover, if $n=1$, the marking is not a Weierstrass point, and if $n=2$, the two markings are exchanged by the hyperelliptic involution.
\end{lemma}

\begin{proof}
    First, notice that $\cD^{2g+1}\subset \cM_{g,n}^r$ parametrizes curves with \emph{separating} $A_{2g+1}$-singularities, because of \Cref{rem:genus-count}.

    Let us deal with the case $n=0$, the same argument works in the cases \(n=1,2\). Suppose that $C\in \cH_g^{r} \cap \cD^{2g+1}$. The quotient $C/\sigma$ must be irreducible. Indeed, $C$ is the union of two irreducible rational smooth curves attached at an $A_{2g+1}$-singularity (see \Cref{rem:genus-count}). By \cite[Proposition~1.3]{Per1}, the involution must exchange the two irreducible components, thus giving $C\in \cH_g^{r,\circ}$.

    The ramification divisor of the morphism $C\rightarrow \bP^1$ has length $2g+2$, and \cite[Corollary~1.4]{Per2} shows that it must be concentrated at the $A_{2g+1}$-singularity. This implies that the branching divisor is (set-theoretically) supported at a single point of $\bP^1$. Using the description of $\cH_g^{r,\circ}$ as cyclic covers of degree $2$ of $\bP^1$ (see \cite{ArVis} or \Cref{sub:hyper-A-stable}), one can prove that $\Aut^{\circ}(C)\simeq \bG_m\ltimes \bG_a$, namely because $\Aut(\bP^1,0)\simeq \bG_m\ltimes \bG_a$. Therefore, by \Cref{rem:odd-atoms}, we get that it coincides with the unique point having $\bG_m\ltimes \bG_a$ stabilizers in $\cD^{2g+1}$.
    \end{proof}

\begin{remark}
    Notice that having markings on the atoms allows us to attach them to other curves using nodes or higher odd \(A\)-singularities, with the goal of constructing non-atomic curves with positive-dimensional stabilizers.
\end{remark}

\subsection{Curves with positive-dimensional stabilizers}\label{sub:geometric-points}

In this subsection, we classify all the subcurves that contribute to the dimension of the automorphism group over an algebraically closed field $k$. Before doing that, we start by introducing the following definitions and notation, which are (almost all) very classical.

\begin{definition}[Subcurve]\index{subcurve}
    Let $C$ be a curve over $k$. A \emph{subcurve} $\Gamma$ of $C$ is a proper, reduced, one-dimensional subscheme of $C$, i.e., the union of some irreducible components of $C$ with the reduced scheme structure. Let $(C,p_1,\dots,p_n)$ be an $n$-pointed $A$-prestable curve over an algebraically closed field $k$, and let $\Gamma\subset C$ be a subcurve. We denote by $C-\Gamma$ the subcurve of $C$ induced by the closure of $C\setminus \Gamma$ in $C$. Moreover, we denote by $I_{\Gamma}$ the set of points $\Gamma \cap (C-\Gamma)$ of $\Gamma$, and by $P_{\Gamma}$ the subset of the markings $\{p_1\dots,p_n\}$ that lie in the subcurve $\Gamma$. Notice that every point in $I_{\Gamma}$ is smooth in $\Gamma$ because there are only $A$-singularities.
\end{definition}

\begin{definition}[Tails and bridges]\index{tails}\index{bridges}
     Let $(C,p_1,\dots,p_n)$ be an $n$-pointed $A$-prestable curve over an algebraically closed field $k$, and let $\Gamma\subset C$ be a connected subcurve. We say that $(\Gamma, I_{\Gamma}\cup P_{\Gamma})$ is a \emph{tail} if $\vert I_{\Gamma}\cup P_{\Gamma} \vert =1$, whereas we say that it is a \emph{bridge} if $\vert I_{\Gamma}\cup P_{\Gamma} \vert =2$. We say that a bridge $\Gamma$ is separating if the curve $C-\Gamma$ is disconnected; otherwise, we say it is non-separating.
\end{definition}

Notice that in the previous definition of subcurve, we do not require $\Gamma$ to be properly contained in $C$. Therefore, given a $1$-pointed curve $(C,p)$, we have that $(C,p)$ itself is a tail. The same is true for bridges and $2$-pointed curves.

\begin{definition}\index{Outer singularity} \index{Inner singularity} \index{Separating singularity} \index{Lonely singularity}
    We say that an $A$-singularity $p$ of $C$ is:
    \begin{itemize}
        \item[(a)] \emph{outer} if it lies in the intersection of (exactly) two irreducible components of $C$;
        \item[(b)] \emph{lonely} if it is outer and it is the unique (set-theoretic) intersection of the two irreducible components;
        \item[(c)] \emph{separating} if the partial normalization at $p$ is disconnected;
        \item[(d)] \emph{inner} if it is not outer.
    \end{itemize}
    Clearly, $(c)\implies (b)\implies (a)$, and $p$ must be an odd $A$-singularity, i.e., an $A_{2i+1}$-singularity for some $i\geq 0$.
\end{definition}

It will be useful to study degenerations of $A$-prestable curves, specifically when they are generically outer. The following proposition will play a fundamental role in this paper; thus, we recall it here.

\begin{proposition}[{\cite[Proposition~2.10]{AlpFedSmyWyck}}]\label{prop:alp-fund}
Let $C\rightarrow \spec R$ be a family of $A$-prestable curves over a dvr $R$, and let $p$ be a section of the family. Suppose that the (geometric) generic fiber $p_{Q}$ is an outer $A_{2k_1+1}$-singularity. Then
\begin{itemize}
    \item the (geometric) special fiber $p_k$ is an outer $A_{2m+1}$-singularity;
    \item each singularity of the (geometric) generic curve $C_Q$ that approaches $p_k$ must be outer as well and must lie on the same two irreducible components of $C_Q$ as $p_Q$;
    \item the collection of such (generic) singularities approaching $p_k$ is necessarily of the form\newline $\{A_{2k_1+1},\dots, A_{2k_s+1}\}$, where $m=s-1+\sum_{i=1}^s k_i$.
\end{itemize}
\end{proposition}

\begin{remark}
    With the notation of \Cref{prop:alp-fund}, if $p_Q$ is lonely (respectively separating), then $p_k$ will also be lonely (respectively separating).
\end{remark}

\begin{definition}[Gluing morphism]
    \index{Gluing morphism}
    Let \((C,\{p_i\}_{i\in I})\) and \((D,\{q_j\}_{j\in J})\) be (not necessarily connected) pointed curves and \(\gamma:C\to D\) a finite morphism. It is called a \emph{gluing morphism} if $\gamma$ restricts to an open immersion $C\setminus \{p_i \mid i \in I\} \to D \setminus \{q_j \mid j \in J\}$ and the image of $p_i$ for $i \in I$ is either a marked point or a singularity.
\end{definition}

We are ready to introduce the ``molecules'' we can construct using our atoms and $A$-singularities. These will be called \emph{rosaries}, and they are a straightforward generalization of \cite[Definition~6.1]{hassett2013log}.

\begin{definition}[(Attached) Rosaries]
    \label{def:attached-rosaries}\index{(attached) Rosaries}
    Let $k$ be an algebraically closed field over $\kappa$. We say that $(R,r_1,r_2)$ is a \emph{2-pointed rosary of length $\ell$} if there exists a surjective gluing morphism
    \[
    \gamma \colon \coprod_{i=1}^{\ell} (R_i,p_{2i-1},p_{2i}) \longrightarrow (R,r_1,r_2)
    \]
    and a sequence of positive integers $k_1, \dots, k_{\ell-1}$ satisfying:
\begin{enumerate}
  \item $(R_i,p_{2i-1},p_{2i})$ is a $2$-pointed smooth rational curve for $i=1,\dots,\ell$;
  \item $\gamma(p_{2i})=\gamma(p_{2i+1})$ is an $A_{2k_i+1}$-type singularity of $R$ for $i=1,\dots,\ell-1$;
  \item $\gamma(p_1)=r_1$ and $\gamma(p_{2\ell})=r_2$.
\end{enumerate}
Moreover, we say that $(R, r_1)$ (resp.\ $R$) is a \emph{$1$-pointed rosary} (resp.\ \emph{0-pointed rosary}, or simply rosary).

We say that $(C,\{p_i\}_{i=1}^n)$ has an \emph{$A_{d_1}/A_{d_2}$-attached rosary of length $\ell$} if there exists a gluing morphism
\[
\gamma \colon (R,r_1,r_2) \longrightarrow (C,\{p_i\}_{i=1}^n)
\]
such that
\begin{enumerate}
  \item $(R,r_1,r_2)$ is a rosary of length $\ell$;
  \item for $j=1,2$, $\gamma(r_j)$ is an $A_{d_j}$-singularity of $C$, or if $d_j=0$, we allow $\gamma(r_j)$ to be a marked point of $(C,\{p_i\}_{i=1}^n)$.
\end{enumerate}

Similarly, we say that $(C,\{p_i\}_{i=1}^n)$ has an \emph{$A_d$-attached rosary of length $\ell$} if there exists a gluing morphism
\[
\gamma \colon (R,r_1) \longrightarrow (C,\{p_i\}_{i=1}^n)
\]
such that
\begin{enumerate}
  \item $(R,r_1)$ is a rosary of length $\ell$;
  \item $\gamma(r_1)$ is an $A_{d}$-singularity of $C$, or if $d=0$, we allow $\gamma(r_1)$ to be a marked point of $(C,\{p_i\}_{i=1}^n)$.
\end{enumerate}
We say that $C$ is a \emph{closed rosary of length $\ell$} if $C$ has an $A_{2d+1}/A_{2d+1}$-attached rosary
\[
\gamma \colon (R,r_1,r_2) \longrightarrow C
\]
of length $\ell$ such that $\gamma(r_1)=\gamma(r_2)$ is an $A_{2d+1}$-singularity of $C$.
\end{definition}

\begin{remark}
    A $2$-pointed rosary is always $A_r$-stable. With the notation of \Cref{def:attached-rosaries}, the $1$-pointed rosary $(R, r_1)$ is stable if and only if $k_{\ell-1} \geq 2$, and the $0$-pointed rosary is stable if and only if $k_1 \geq 2$ and $k_{\ell-1} \geq 2$. Finally, notice that crimping an even singularity to a marked point of a rosary creates a curve that has an attached rosary whose gluing morphism is surjective.
\end{remark}

\begin{figure}[H]
    \centering
    \includegraphics[width=0.7\textwidth]{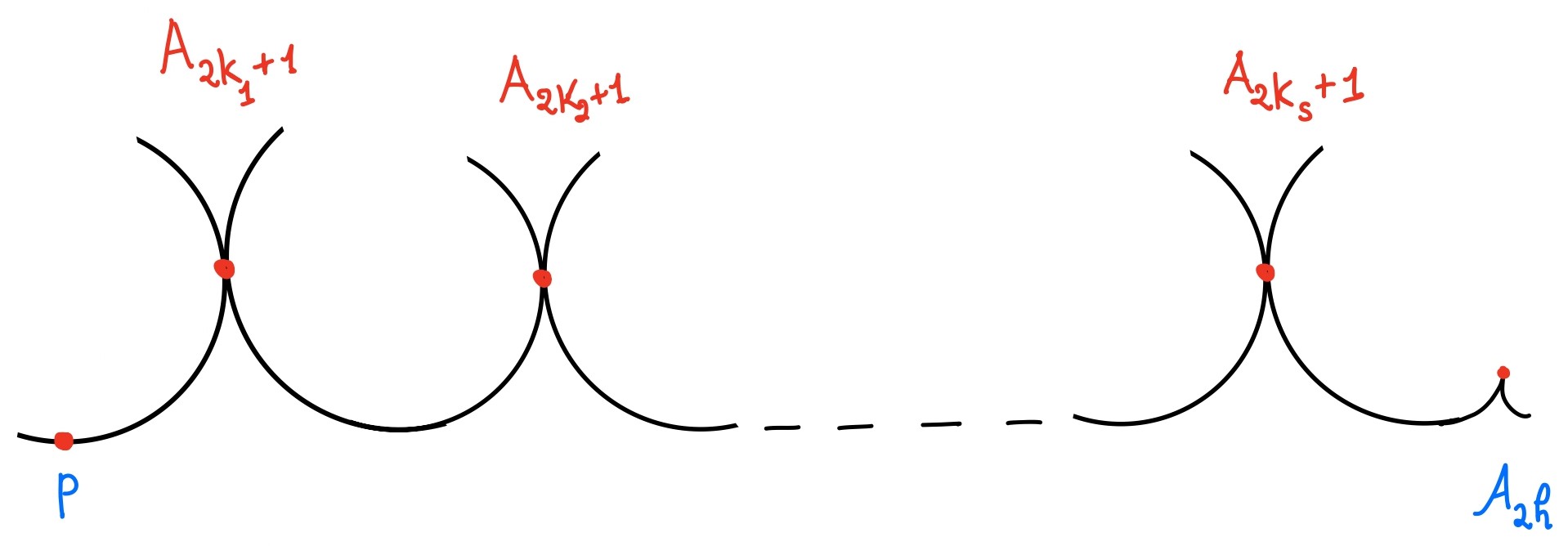}
    \label{fig:attached-rosary}
    \caption{Example of a curve that is not a rosary but has an attached $2$-pointed rosary.}
\end{figure}

\begin{remark}
    An $n$-pointed odd atom is an $n$-pointed rosary of length $2$ (with $n\leq 2$), whereas a $0$- or $1$-pointed even atom has respectively a $0$- or $1$-pointed attached rosary of length $1$. When \( (C,\{p_i\}_{i=1}^n)\) has an attached rosary \((R,r_1,r_2)\) or \((R,r_1)\) that is a \(2\)- or \(1\)-pointed atom as in \Cref{def:odd-atoms} or \Cref{def:even-atoms}, we will say that it has an \emph{attached} atom.
\end{remark}

We aim now to compute the automorphism group of rosaries; we will then prove that attached rosaries are the parts of curves that cause $\Aut^{\circ}$ not to be trivial. An explicit description of rosaries with $A_3$-singularities is carried out in \cite[Remark~3.4]{Viviani}; in this case, $2$-pointed rosaries always have a nontrivial action of $\Gm$ on the curve. This is not always the case for higher $A$-type singularities, as we will see.

\begin{proposition}
    \label{prop:aut0-rosaries}
    Let $n\leq 2$ and $\ell \geq 1$. Suppose that $(R, r_i)$ is an $n$-pointed rosary of length $\ell$ such that $\Aut^{\circ}(R, r_i)$ is not trivial. If $(R,r_i)$ is $A_r$-stable, then we have the following:
    \begin{itemize}
        \item if $n+\ell \geq 3$, then $\Aut^{\circ}(R,r_i)\simeq \bG_m$;
        \item if $n=0$ and $\ell=2$, then $\Aut^{\circ}(R)$ is either $\Ga$ or $\Gm\ltimes \Ga$.
    \end{itemize}
    Furthermore, a closed rosary $R$ such that $\Aut^{\circ}(R)$ is not trivial necessarily has even length.

    If $(R,r_i)$ is only $A_r$-prestable, we get that the other possible groups are $\Gm \ltimes \Ga^2$ and $\PGL_2$. Finally, all the possibilities for $\Aut^{\circ}$ listed above can occur.
\end{proposition}

\begin{proof}
    Assume first that the rosary is $2$-pointed. By \Cref{rem:motivation-hyp-sing}, the action of $\Aut^{\circ}(R, r_i)$ lifts to the pointed normalization
    \[
    \coprod_{i=1}^{\ell} (R_i,p_{2i-1},p_{2i}) \longrightarrow (R,r_1,r_2).
    \]
    Let $N_i \lhd \Aut^{\circ}(R, r_1, r_2)$ be the kernel of the map $\Aut^{\circ}(R,r_1,r_2) \to \Aut^{\circ}(R_i,p_{2i-1},p_{2i})$. Thanks to \Cref{lem:hyp-crimping-datum-odd} (see also \Cref{rem:at-most-one-Gm}), we have $N_i = N_{i+1}$ since by definition $(R_i,p_{2i-1},p_{2i})\simeq (\bP^1,0,\infty)$ and $\Aut(\bP^1,0,\infty)\simeq \bG_m$. Therefore, the injectivity of the map
    \[
    \Aut^{\circ}(R, r_1, r_2) \to \Aut^{\circ}(\coprod_{i=1}^{\ell} (R_i,p_{2i-1},p_{2i}))
    \]
    forces $N_i$ to be trivial for any $i = 1, \dots, \ell$, and therefore $\Aut^{\circ}(R, r_1, r_2) \simeq \Gm$. Since \Cref{lem:hyp-crimping-datum-odd} also allows us to construct a $\Gm$-equivariant crimping datum, we get the statement for the case $n=2$.

    Notice that the $2$-pointed case automatically lands us in the first case since $\ell \geq 1$. If $(R,r)$ is $1$-pointed and $\ell \geq 2$, then we know that there is an irreducible smooth rational component $\Gamma$ such that $P_{\Gamma}=\emptyset$ and $I_{\Gamma}=\{p\}$, where $p$ is an odd $A_h$-singularity with $h\geq 5$ (since $R$ is $A_r$-stable). Therefore, $(R-\Gamma,r,p)$ is a $2$-pointed $A_h/A_0$-attached rosary of length $\ell-1\geq 1$. Since $\Aut^{\circ}(R-\Gamma, r, p) \simeq \Gm$ because it is $2$-pointed, \Cref{cor:quotient-stack-description} shows that $(R,r)$ can be identified with an object of $[(U_h/\Ga)/\bG_m]$, where $(U_h/\Ga)$ is a $\bG_m$-representation with strictly positive weights. Thus, $\Aut^{\circ}(R,r)$ being nontrivial implies $\Aut^{\circ}(R,r)\simeq \bG_m$.

    Notice that the case $n=1$ and $\ell=1$ is not stable (it is a $1$-pointed smooth rational curve); thus, it remains to study the case $\ell = 2$ and $n = 0$. This follows again by \Cref{cor:quotient-stack-description}. We avoid writing down the classification of the non-stable cases, as they are not relevant for the current paper, even though it can be easily deduced using the same ideas.

    In the case of a closed rosary $R$, we have $\Aut^{\circ}(R) \subset \Gm$ using the same strategy. We can then observe that, on any connected component $(\bP^1, 0, \infty)$ of the pointed normalization, the action of $\Gm \simeq \Aut^{\circ}(\bP^1, 0, \infty)$ on $T_0 \bP^1$ and $T_{\infty} \bP^1$ has opposite weights: the claim then follows from \Cref{lem:hyp-crimping-datum-odd}.
\end{proof}

The following is a consequence of \Cref{lem:hyp-crimping-datum-odd} and \Cref{cor:quotient-stack-description}.
\begin{remark}
    Given a string of positive integers $k_1, \dots, k_{\ell-1}$ with at least one $k_j > 1$, it is possible to construct a $0$-pointed rosary $R$ as in \Cref{def:attached-rosaries} such that $\Aut^{\circ}(R)$ is trivial. In particular, for any choice of $d$, we can construct a closed rosary $R'$ such that $\Aut^{\circ}(R')$ is trivial.
\end{remark}

Moreover, the following is a consequence of \Cref{lem:hyp-crimping-datum-odd}.
\begin{remark}
    Let $(R, r_1, r_2)$ be a $2$-pointed rosary of even length such that $\Aut^{\circ}(R, r_1, r_2) \simeq \Gm$. Then, for any $d \geq 1$, we can construct a closed rosary $R'$ by crimping an $A_{2d+1}$-singularity at $r_1, r_2$ such that $\Aut^{\circ}(R') \simeq \Gm$. This procedure is not unique; it is therefore possible to construct a family of such closed rosaries over $\Gm$ whose geometric fibers are pairwise not isomorphic (see also \Cref{rem:at-most-one-Gm}).
\end{remark}

Finally, we prove the fundamental proposition, which explains why these are the building blocks for $A_r$-stable curves with positive-dimensional stabilizers.

\begin{proposition}\label{prop:stab-no-out}
    Let $(C,p_1,\dots,p_n)$ be an $n$-pointed $A_r$-stable curve
    \index{$A_r$-stable curve}
    of genus $g$ with \textbf{no outer nodes} and $r\leq 2g$. Then $\CAut(C,p_1,\dots,p_n)$ is either trivial or $\bG_m$. Moreover, if $\CAut(C,p_1,\dots,p_n)\simeq \bG_m$, then $n\leq 2$ and either
    \begin{itemize}
        \item[i)] $(C,p_1,\dots,p_n)$ has an attached rosary and the gluing morphism is surjective;
        \item[ii)] $n=0$ and $C$ is a closed rosary of even length;
        \item[iii)] $n=0$ and $C$ is irreducible of geometric genus $0$ with either one or two hyperelliptic even $A$-singularities.
    \end{itemize}
\end{proposition}

Notice that case iii) is actually a special case of i) for length $1$ rosaries.

\begin{proof}
    We proceed by induction on the genus $g$. The case $g=0$ is trivial, because of \Cref{rem:genus-count}. Suppose that $(C,p_1,\dots,p_n)$ is an integral $n$-pointed $A_r$-stable curve of genus $g$ such that the group $\CAut(C,p_1,\dots,p_n)$ is not trivial. It follows easily that $n\leq 1$ and $C$ has geometric genus $0$. If $n=1$, looking at the normalization we can deduce that $C$ is a $1$-pointed even atom (see for instance \Cref{rem:even-atoms}). If $n=0$, $C$ cannot have inner odd $A$-singularities; otherwise, $\CAut(C)$ would be trivial by \Cref{lem:hyp-crimping-datum-odd}. Therefore, $C$ is irreducible of geometric genus $0$ with either one or two even $A$-singularities.

    Suppose now that $C$ is not irreducible. Since there are no outer nodes, there exists at least one outer $A_{2k+1}$-singularity with $k\geq 1$. We denote by $q \in C$ such an $A_{2k+1}$-singularity, by $\widetilde{C}\rightarrow C$ the partial normalization at $q$, and by $q_1,q_2$ the two closed points of $\widetilde{C}$ that lie over $q$. Moreover, by abuse of notation, we denote by $p_i$ the preimage of $p_i$ in $\widetilde{C}$. Clearly, $\CAut(\widetilde{C},p_1,\dots,p_n,q_1,q_2)$ is not trivial, otherwise the same would be true for $\CAut(C,p_1,\dots,p_n)$.

    Suppose that $q$ is a non-separating singularity, i.e., $\widetilde{C}$ is connected. Notice that the arithmetic genus of $\widetilde{C}$ cannot be zero; otherwise, $C$ would either be irreducible or there would be outer nodes in $C$. This implies that $k\leq g-1$ and $(\widetilde{C},p_1,p_2,\dots,p_n,q_1,q_2)$ is an object of
    $\cM_{g-h,n,2[k]}^r$.
    Because there are no outer nodes, \Cref{rem:sprout-proj-line} shows that $(\widetilde{C},p_1,p_2,\dots,p_n,q_1,q_2)$ lives inside the open $\cM_{g-h,n+2}^r\subset  \cM_{g-h,n+2}^r \times [\bA^1/\bG_m] \times [\bA^1/\bG_m]$. By the induction hypothesis, this implies that $n=0$, $(\widetilde{C},q_1,q_2)$ is a $2$-pointed rosary of hyperelliptic $A$-singularities, and $\Aut^{\circ}(\widetilde{C},q_1,q_2)\simeq \bG_m$. By construction, $C$ is a closed rosary, and \Cref{prop:aut0-rosaries} gives that $C$ has even length and $\Aut^{\circ}(C)\simeq \bG_m$.

     Finally, suppose that $q$ is a separating (odd) $A$-singularity, and let $C_1$ and $C_2$ be the two connected components of $\widetilde{C}$. We denote by $P_i$ the subset of $\{p_1,\dots,p_n\}$ that lies in $C_i$ and by $n_i$ the cardinality of $P_i$ for $i=1,2$. Suppose that $C_1$ and $C_2$ have positive arithmetic genus, then
     $$ (C_i,P_i,q_i) \in \cM_{g_i, n_i, [k]} \simeq \cM_{g_i, n_i+1} \times [\bA^1/\bG_m] $$
     where $g_i$ is the arithmetic genus of $C_i$ for $i=1,2$. This follows from \Cref{prop:descr-an-disp-i}, while the isomorphism is described in \Cref{prop:descr-base-crimping-even}. Because $C$ does not have outer nodes, again we have
     $$ (C_i,P_i,q_i) \in  \cM_{g_i, n_i+1} \subset \cM_{g_i, n_i+1} \times [\bA^1/\bG_m] $$
     for $i=1,2$. By induction, we get that $n_i\leq 1$, and both $(C_1,P_1,q_1)$ and $(C_2,P_2,q_2)$ are rosaries of hyperelliptic $A$-singularities. Therefore, $n=n_1+n_2\leq 2$, and $C$ is the gluing of the two rosaries of hyperelliptic $A$-singularities along an $A_{2k+1}$-singularity that is also hyperelliptic; otherwise, $\Aut^{\circ}(C,p_1,\dots,p_n)$ would be trivial. Thus, we get case (i).

     Notice that if both $C_1$ and $C_2$ have arithmetic genus zero, then $k=g$ and therefore $r=2g+1$. Thus, the only remaining case is where only one curve, say $C_1$, has genus $0$. This is straightforward, and it consists of adding a rational component to one of the extremities of the rosary of hyperelliptic $A$-singularities $C_2$. By \Cref{cor:quotient-stack-description}, we have the statement.
\end{proof}

\Cref{prop:stab-no-out} motivates the following definition:

\begin{definition}[(Attached) $\Gm$-Rosaries]\label{def:Gm-rosaries}\index{(Attached) $\Gm$-Rosaries}
    Let $k$ be an algebraically closed field over $\kappa$. A (potentially marked) curve $R$ over $k$ is a \emph{$\Gm$-rosary} if $R$ is a (potentially marked) rosary and there exists an injective map $\Gm \to \Aut(R)$.

    We say that a (potentially marked) curve $C$ over $k$ has an \emph{attached $\Gm$-rosary} if there exists a (potentially marked) $\Gm$-rosary $R$ and a gluing morphism $\gamma \colon R \to C$.

    We say that $C$ has an \emph{attached rosary of hyperelliptic A-singularities} if it has an attached rosary $\gamma \colon R \to C$ that contains in its image a hyperelliptic $A$-singularity of odd type.

    Similarly, we say that $C$ is a \emph{closed $\Gm$-rosary} \index{closed $\Gm$-rosaries} if it is a closed rosary and there exists an injective map $\Gm \to \Aut(R)$.
\end{definition}

\begin{remark}
    Any attached rosary of hyperelliptic $A$-singularities is an attached $\Gm$-rosary. The converse is not always true: automorphisms of rosaries do not necessarily lift to the whole curve.
\end{remark}

In \Cref{prop:stab-no-out}, we require that the inequality $r\leq 2g$ holds. When $r > 2g$, the only case is $r=2g+1$, which is covered in the following corollary.

\begin{corollary}\label{cor:gm-decom-aut}
    Let $r$, $n$, and $g$ be three nonnegative integers. Let $(C,p_1,\dots,p_n)$ be an $A_r$-stable $n$-pointed curve of genus $g$ over an algebraically closed field, then $\CAut(C,p_1,\dots,p_n)$ is exactly one of the following:
    \begin{itemize}
        \item[i)] the trivial group,
        \item[ii)] a split torus, i.e., a direct product of copies of $\bG_m$,
        \item[iii)] $\Ga$ or $\bG_m \ltimes \Ga$;
    \end{itemize}
    moreover, a non-reductive automorphism group appears only if $r \geq 2g+1$, $n=0$, and $C$ is an $A$-stable curve with an $A_{2g+1}$-singularity (not necessarily an odd atom).
\end{corollary}
\begin{proof}
If $r = 2g+1$, $n=0$, and $C$ is an $A$-stable curve with an $A_{2g+1}$-singularity, then the claim follows from \Cref{prop:aut0-rosaries}. If this is not the case, let $(\widetilde{C}, q_j) \to (C, p_i)$ be the pointed normalization at every outer node. Thanks to \Cref{rem:motivation-hyp-sing}, we have an inclusion $\Aut^{\circ}(C, p_i) \hookrightarrow \Aut^{\circ}(\widetilde{C}, q_i)$ (it is indeed an isomorphism). The target is then a torus thanks to \Cref{prop:stab-no-out}; hence the claim.
\end{proof}

We end this section by giving an estimate of the dimension of the stabilizer group of an $A_r$-stable curve.

\begin{lemma}\label{lem:dim-stab-genus}
    Let $(g,n)$ be two nonnegative integers such that $3g-3+n>0$, and let $(C,p_1,\dots,p_n)$ be an $n$-pointed $A_r$-stable curve of genus $g$. Then $\dim \Aut(C,p_1,\dots,p_n) \leq g $.
\end{lemma}
\begin{proof}
    The result follows by a simple inductive argument on the genus $g$. If the genus is $0$, then the moduli stack is a scheme, so there is nothing to prove. Suppose $g>0$; thus, we can assume $\dim \Aut(C,p_1,\dots,p_n)>0$. Suppose we are not in case $(iii)$ of \Cref{cor:gm-decom-aut}. This implies there exists a hyperelliptic $A_h$-singularity $p\in C$ with $h\geq 2$. By \Cref{lem:hyp-crimping-datum-even} and \Cref{rem:at-most-one-Gm}, it is clear that the dimension of the automorphism group of the stabilization of the pointed partial normalization is $\dim \Aut(C,p_1,\dots,p_n) - 1$. By induction, we get that $\dim \Aut(C,p_1,\dots,p_n) - 1 \leq g'$, where $g'$ is the genus of the partial normalization at $p$. Since $h\geq 2$, we have $g'\leq g-1$; thus, we obtain the statement. If we are in case $(iii)$ of \Cref{cor:gm-decom-aut}, then by stability $g\geq 2\geq \dim \Aut(C)$, and we are done.
\end{proof}
\subsection{Deformation theory}\label{sub:def-theory}

In this subsection, we study the deformation theory of the hyperelliptic $A$-singularities we introduced in the previous subsection. This will play a key role in understanding the possible isotrivial degenerations in $\cM_{g,n}^r$.

\begin{remark}
Recall that given a curve $C$ over an algebraically closed field $k$ and a singularity $q$, we can always consider the surjective morphism of vector spaces
$$\eta: T^1\Def_C \rightarrow T^1_q$$
where $T^1_q$ is the deformation space of the singularity $q$.
The kernel of $\eta$ classifies first-order deformations of $C$ which preserve the singularity $q$. If $q$ is the only singularity in $C$, the exact sequence
$$
\begin{tikzcd}
&                                           & T^1\Def_C \arrow[d, Rightarrow, no head] \arrow[r] & T^1_q \arrow[d, Rightarrow,  no head]        &   \\
0 \arrow[r] & {\oH^1(\cHom(\Omega_C,\cO_C))} \arrow[r] & {\Ext^1(\Omega_C,\cO_C)} \arrow[r]                 & {\oH^0(\cExt^1(\Omega_C,\cO_C))} \arrow[r] & 0
\end{tikzcd}
$$
gives us the description of $\ker \eta$.
\end{remark}

Recall that we have a closed embedding
$$\cH_{g}^r\hookrightarrow \cM_g^r$$
which gives us an injective morphism of tangent spaces
$$ T_{[C]}\cH_g^r \hookrightarrow T_{[C]}\cM_g^r$$
where the right-hand side can be identified with $T^1\Def_C$. For the rest of the subsection, we will consider $r\geq 2g+1$, so that we can work in the stack $\cM^r_{g,n} = \cM_{g,n}^{2g+1}$ (see \Cref{rem: max-sing}).

\begin{proposition}\label{prop:def-atom}
    Let $C$ be an atom of genus $g$. Then there exists a section of
    $$\eta: T^1\Def_C \rightarrow T^1_q$$
    which identifies $T^1_q$ with $T_{[C]}\cH_g^r$, the tangent space of the hyperelliptic locus.
\end{proposition}

\begin{proof}
    Notice that $C$ is a cyclic cover of degree $2$ of $\bP^1$; therefore, the tangent space $T_{[C]}\cH_g$ coincides with the first-order deformations $T^1\Def_{f}$ of the finite flat morphism $f:C\rightarrow \bP^1$ of degree $2$.
    By \cite[tag~0E3V]{stacks-project}, we have an exact sequence
    $$
    \begin{tikzcd}
    0 \arrow[r] & T^1\Def_f \arrow[r] & {\Ext^1(\Omega_C,\cO_C)} \arrow[r, "\gamma"] & {\Ext^1(f^*\Omega_{\bP^1},\cO_C)}
    \end{tikzcd}
    $$
    where the injectivity on the left follows from \cite[Section~3]{Per2}. Moreover, the exact sequence
    $$
    \begin{tikzcd}
    0 \arrow[r] & {\oH^1(\cHom(\Omega_C,\cO_C))} \arrow[r]  &  {\Ext^1(\Omega_C,\cO_C)} \arrow[r]                & {\oH^0(\cExt^1(\Omega_C,\cO_C))} \arrow[r]& 0
    \end{tikzcd}
    $$
    comes from the local-to-global spectral sequence; thus, it is functorial in its entries. Therefore, we can consider the morphisms of short exact sequences
    $$
    \begin{tikzcd}
    0 \arrow[r] & {\oH^1(\cHom(\Omega_C,\cO_C))} \arrow[r]  \arrow[d, "\theta"] &  {\Ext^1(\Omega_C,\cO_C)} \arrow[r]      \arrow[d, "\gamma"]            & {\oH^0(\cExt^1(\Omega_C,\cO_C))} \arrow[r] \arrow[d, ]& 0 \\
    0 \arrow[r] & {\oH^1(\cHom(f^*\Omega_{\bP^1},\cO_C))} \arrow[r] &  {\Ext^1(f^*\Omega_{\bP^1},\cO_C)} \arrow[r]                 & {\oH^0(\cExt^1(f^*\Omega_{\bP^1},\cO_C))} \arrow[r] &  0
    \end{tikzcd}
    $$
    induced by the morphism $f^*\Omega_{\bP^1} \rightarrow \Omega_{C}$. Since the support of $\Omega_{C/\bP^1}$ is $0$-dimensional, we have that the cohomology group $\oH^1(\cHom(\Omega_{C/\bP^1},\cO_C))$ vanishes, which implies the injectivity of $\theta$. Thus the induced morphism
    $$ T^1\Def_f \simeq \ker \gamma \longrightarrow \oH^0(\cExt^1(\Omega_C,\cO_C))\simeq T^1_q$$
    is injective as well by the snake lemma. Finally, a straightforward computation in deformation theory shows that the two vector spaces have the same dimension, which concludes the proof.
\end{proof}

The same ideas work in the case of pointed atoms.

\begin{proposition}\label{prop:def-hyp-rat-sing}
    Let $(C,p_1,p_2)$ be a $2$-pointed odd atom of genus $g$. Then there exists a section of
    $$\alpha: T^1\Def_{(C,p_1,p_2)} \rightarrow  T^1\Def_C \rightarrow T^1_q$$
    which identifies $T^1_q$ with $T_{[C,p_1,p_2]}\cH_{g,g_1^2}^r$. Similarly, let $(C,p)$ be a $1$-pointed even atom of genus $g$. Then there exists a section of
    $$\beta: T^1\Def_{(C,p)} \rightarrow T^1\Def_C \rightarrow T^1_q$$
    which identifies $T^1_q$ with $T_{[C,p]}\cH_{g,w}^r$.
\end{proposition}

\begin{proof}
    We provide the proof for the first case. The same proof can be easily adapted to the second one.

    Consider the exact sequence
    $$
    \begin{tikzcd}
    & {T^1\Def_{(C,p_1,p_2)}} \arrow[r] \arrow[d,  Rightarrow, no head] & T^1\Def_C \arrow[d, Rightarrow, no head] \\
{\Hom(\Omega_C,\cO_{p_1}\oplus \cO_{p_2})} \arrow[r, "\phi"] & {\Ext^1(\Omega_C,\cO_C(-p_1-p_2))}      \arrow[r,"\psi"]                         & {\Ext^1(\Omega_C,\cO_C)}
\end{tikzcd}
$$
induced by the short exact sequence
$$
    \begin{tikzcd}
    0 \arrow[r] & \cO_C(-p_1-p_2) \arrow[r] & \cO_C \arrow[r] & \cO_{p_1} \oplus \cO_{p_2} \arrow[r] & 0.
\end{tikzcd}
$$
We have that $\psi^{-1}(T_q^1)\subset T^1\Def_{(C,p_1,p_2)}$ classifies deformations of $(C,p_1,p_2)$ such that $C$ is hyperelliptic, because of \Cref{prop:def-atom}. Therefore, we have that $T_{[C,p_1,p_2]}\cH_{g,g_1^2}$ is contained inside $\psi^{-1}(T_q^1)$; if we prove that the induced morphism
$$ T_{[C,p_1,p_2]}\cH_{g,g_1^2} \longrightarrow T_q^1$$
is bijective, the statement follows. A straightforward computation shows that the two vector spaces have the same dimension; therefore, it is enough to prove that the morphism above is injective, or equivalently, that
$$\im \phi \cap T_{[C,p_1,p_2]}\cH_{g,g_1^2} = \{0\}$$
inside $T^1\Def_{(C,p_1,p_2)}$.

The vector subspace $\im \phi \subset T^1\Def_{(C,p_1,p_2)}$ classifies deformations $(C_{\epsilon},p_1^{\epsilon},p_2^{\epsilon})$ such that $C_{\epsilon}$ is the trivial deformation of $C$; thus,
$$ \im \phi \cap T_{[C,p_1,p_2]}\cH_{g,g_1^2} = \phi\big( \Hom(\Omega_C,\cO_{p_1}\oplus \cO_{p_2})^{\rm inv}\big) $$
where $ \Hom(\Omega_C,\cO_{p_1}\oplus \cO_{p_2})^{\rm inv} \subset \Hom(\Omega_C,\cO_{p_1}\oplus \cO_{p_2})$ is the vector subspace of elements invariant under the hyperelliptic involution. Deformation theory helps us describe the kernel of  $\phi$ using the exact sequence
$$
\begin{tikzcd}
{\Hom(\Omega_C,\cO_C)} \arrow[r, "\varphi"] & {\Hom(\Omega_C,\cO_{p_1}\oplus \cO_{p_2})} \arrow[r, "\phi"] & {\Ext^1(\Omega_C,\cO_C(-p_1-p_2))};
\end{tikzcd}
$$
we will prove that $$\Hom(\Omega_C,\cO_{p_1}\oplus \cO_{p_2})^{\rm inv} \subset \im \varphi,$$ which implies the statement. Recall that $\Hom(\Omega_C,\cO_C)$ classifies the infinitesimal automorphisms of $C$, namely the tangent space of the automorphism group of $C$,  whereas $\varphi$ restricts the automorphisms to the points. A simple verification using the explicit description of $\varphi$ shows what we need, since any infinitesimal automorphism of $C$ commutes with the hyperelliptic involution (because of \cite[Section~3]{Per2}).
\end{proof}

\begin{remark}\label{rem:non-weierstrass-case}
    Notice that \Cref{prop:def-hyp-rat-sing} also holds for the case of a $1$-pointed odd atom $(C,p_1)$. Using \Cref{lem:odd-atom}, we deduce that there exists an involution $\sigma$ exchanging the two irreducible components, and $p_1$ is not a Weierstrass point. Therefore, we can reduce to the $2$-pointed odd atom case considering $(C,p_1,\sigma(p_1))$. Notice that this also implies that any deformation that deforms the singularity $q$ will give rise to a $1$-pointed hyperelliptic curve where the marking is not a Weierstrass point.
\end{remark}

The following is the key remark for understanding how the combinatorics of curves with $A$-singularities changes through isotrivial degenerations.

\begin{remark}\label{rem:decom-tang}
    Let $(C,p_1,p_2)$ be a $2$-pointed atom of genus $g$ and $q$ be the singularity. The kernel of the morphism
    $$ \alpha: T^1\Def_{(C,p_1,p_2)} \longrightarrow T_q^1$$
    considered in \Cref{prop:def-hyp-rat-sing} has a very explicit description. In fact, $\ker \alpha$ is the subspace of first-order deformations of $(C,p_1,p_2)$ in which the $A_{2g+1}$-singularity $q$ is preserved (possibly changing the crimping data). Therefore, it is the tangent space of the closed substack $\cD^{2g+1} \subset \cM_{g,2}$ at the point $[C,p_1,p_2]$, and to ease the notation, we will denote this by ${\rm Cr}_q$, and we will call it the \emph{crimping space} of the singularity $q$. \Cref{prop:def-hyp-rat-sing} implies that we have a decomposition
    $$ T_{[C,p_1,p_2]}\cM_{g,2}=T_{[C,p_1,p_2]}\cH_{g,g_1^2} \oplus {\rm Cr}_q$$
    of tangent spaces, namely that $\cD^{2g+1}$ and $\cH_{g,g_1^2}$ intersect transversely in $\cM_{g,2}$ at the closed point $(C,p_1,p_2)$. More generally, if $C$ is a curve and $q$ is a singularity, we have a decomposition
    $$ \Ext^1(\Omega_C,\cO_C)\simeq T_q^1\oplus {\rm Cr}_q\oplus R$$
    where
    \begin{itemize}
        \item  ${\rm Cr}_q\oplus R \simeq \oH^1(\cHom(\Omega_C,\cO_C))$;
        \item  $R$ can be identified with first-order deformations of the pointed partial normalization of $C$ at $q$.
    \end{itemize}
    A similar statement holds for the even atom.
\end{remark}

The sign decomposition of the $\bG_m$-action on the deformation spaces of the atoms coincides with the transversal decomposition given by the hyperelliptic locus and the crimping data.
\begin{corollary}\label{cor:gm-decom-tang}
    Let $(C,p_1,p_2)$ be a $2$-pointed odd atom of genus $g$ and fix an isomorphism $\bG_m \simeq \Aut(C,p_1,p_2)$. Then the decomposition introduced in \Cref{rem:decom-tang} coincides with the sign decomposition for the $\bG_m$-action, i.e., $$(T_{[C,p_1,p_2]}\cM_{g,2})^>=T_{[C,p_1,p_2]}\cH_{g,g_1^2}$$ and $$(T_{[C,p_1,p_2]}\cM_{g,2})^<=T_{[C,p_1,p_2]}\cD^{2g+1}$$ or vice versa. The same is true for the even case.
\end{corollary}

\begin{proof}
    After unraveling definitions, this is a direct consequence of Proposition 6.2 and Proposition 6.3 of \cite{AlpSmyVdW}.
\end{proof}

Now that we have dealt with the case of the atoms, we can focus on understanding the deformation theory of the curves with positive-dimensional automorphism group, specifically the subcurves that contribute to its dimension (see \Cref{prop:stab-no-out}). Let us start with an example to get the idea of how the $\bG_m$-action decomposes the deformation space.
\begin{example}\label{ex:deformation-alter}
    Let $C_0$ be an $n$-pointed rosary of hyperelliptic $A$-singularities of length $3$, and denote by $q_1,q_2$ the two singularities of type $A_{2h_1+1}$ and $A_{2h_2+1}$, respectively. Moreover, suppose there is an isotrivial degeneration $C_{\Theta}\rightarrow \Theta$ whose special fiber is isomorphic to $C_0$. Then the combinatorial structure of the generic fiber $C_1$ is controlled by the one of $C_0$.

    In fact, we have a decomposition
    $$\Def_C = T^1_{q_1}\oplus T^1_{q_2}\oplus {\rm Cr}_{q_1}\oplus {\rm Cr}_{q_2}$$
    of the deformation space of $C$ thanks to \Cref{rem:decom-tang}. Furthermore, the isotrivial degeneration induces a cocharacter $\bG_m \rightarrow \Aut(C_0)\simeq \bG_m$ which also induces a sign decomposition of $\Def_C$. Because of the explicit description of the action of $\Aut^{\circ}(C_0)$ on $C_0$, we have that the $\bG_m$-action on the tangent space of $q_1$ has the opposite sign compared to the one on the tangent space of $q_2$. Therefore, \Cref{cor:gm-decom-tang} implies that
    $$ (\Def_C)^> =T^1_{q_1}\oplus{\rm Cr}_{q_2} $$
    and
    $$ (\Def_C)^<=T^1_{q_2}\oplus {\rm Cr}_{q_1}$$
    or vice versa. Therefore, \Cref{lem:Theta-and-BGmR} ensures that if the isotrivial degeneration deforms one singularity, then the other one needs to be preserved (with a possibly different crimping datum) in the generic fiber $C_1$. Moreover, by \Cref{prop:def-hyp-rat-sing} and \Cref{rem:non-weierstrass-case}, the deformation of the odd atom of genus $h_2$ in the chain will be an honestly hyperelliptic curve of genus $h_2$ attached at a non-Weierstrass point.
    \begin{figure}[H]
        \caption{}
        \centering
        \includegraphics[width=0.8\textwidth]{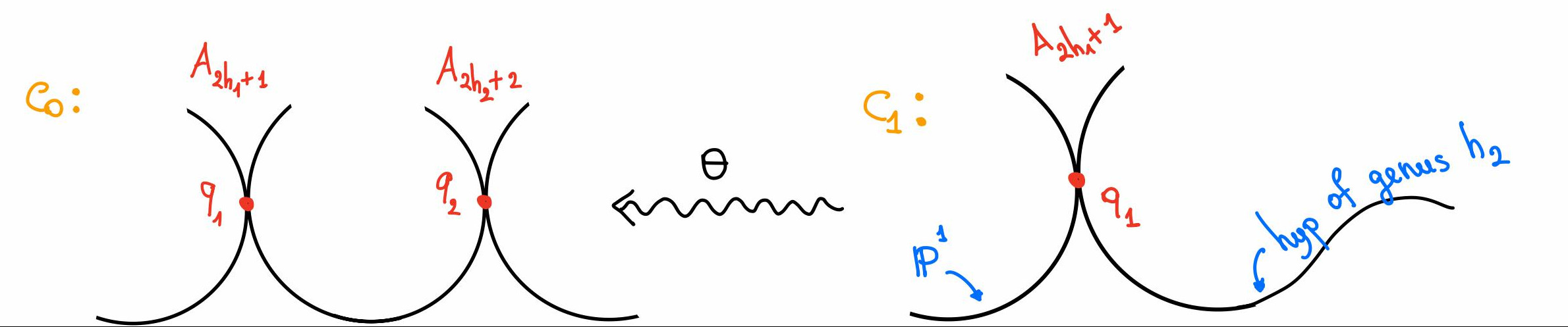}
        \label{fig:deg-hyp-chain}
    \end{figure}

    The same reasoning clearly works if the length of the chain is greater than $3$. Suppose that $C_0$ is a rosary of hyperelliptic $A$-singularities $q_1,\dots,q_n$ where $q_i$ and $q_{i+1}$ share one irreducible component for every $i=1,\dots, n-1$. Moreover, suppose that the isotrivial degeneration deforms $q_{i_0}$ where $i_0$ is odd. Then the singularities $q_2,q_4,\dots,q_{2k}$ for $k=\lfloor n/2 \rfloor$ will also appear in the generic fiber $C_1$ of the isotrivial degeneration (possibly with a different crimping datum), while the odd atom of genus $h_{2i+1}$ associated to the singularity $q_{2i+1}$ will deform to an honestly hyperelliptic curve of genus $h_{2i+1}$ where the two attaching points are exchanged by the involution.

    Similar examples can be constructed starting with a $1$-pointed odd atom of genus $h$ with singularity $q$ and pinching the marking $p$ into an $A_{2k}$-singularity using the unique $\Gm$-equivariant crimping datum.

    The automorphism group will be a one-dimensional torus; see \Cref{prop:stab-no-out}. We have a decomposition
    $$\Def_C = T^1_{q}\oplus T^1_{p}\oplus {\rm Cr_q} \oplus {\rm Cr}_p$$
    and again, an isotrivial degeneration can either deform the singularity $q$ while preserving $p$ (but not necessarily the crimping datum) or vice versa.
    \begin{figure}[H]
        \caption{}
        \centering
        \includegraphics[width=0.7\textwidth]{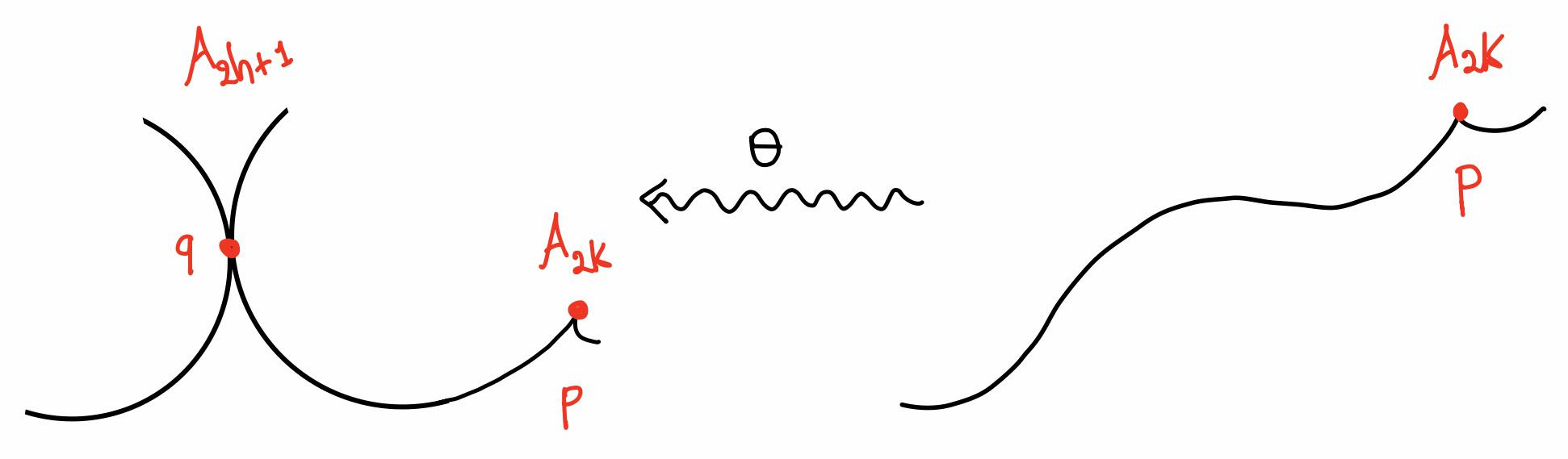}
        \label{fig:example-3.25}
    \end{figure}
\end{example}

Before going further, it is crucial to point out that the following remark is a key point in our study of isotrivial degenerations of $A$-stable curves. Indeed, it gives a significant constraint to the geometry of isotrivial degenerations, and it will play a fundamental role in all that follows.

\begin{remark}
\phantomsection
\label{rem:alternating-deformations}
    \Cref{ex:deformation-alter} determines a big constraint in the geometry of isotrivial degenerations: whenever we have a rosary of hyperelliptic $A$-singularities as above, no isotrivial degeneration can deform two adjacent $A$-singularities. If one is deformed, then the ones adjacent to it have to be preserved, and the curve between them is hyperelliptic. This is true also for nodes, with the extra hypothesis that the $\Gm$-action is trivial on one of the two components meeting at the node. Namely, consider a curve $C_0$ with a subcurve $\Gamma\subset C_0$ such that $(\Gamma,I_{\Gamma})$ is an $A_1$-attached $1$-pointed even atom (or, equivalently, an $A_1/A_{2h}$-attached rosary of length $1$). Suppose we have an isotrivial degeneration to $C_0$ such that the cocharacter $\Gm \rightarrow \Aut(C_0)$ acts trivially on the component $\Gamma'\subset C_0-\Gamma$ that intersects $\Gamma$ at the node. Then, if the singularity of the atom is deformed, the node has to be preserved, and thanks to \Cref{prop:def-hyp-rat-sing}, $\Gamma$ deforms to an honestly hyperelliptic tail attached to the rest of the curve via a node.
    \end{remark}

The previous example motivates the following definitions.

\begin{definition}[Tails and bridges]\label{def:tails and bridges}
Let $(C,p_1,\dots,p_n)$ be an $n$-pointed $A$-prestable curve over an algebraically closed field $k$ and let $\Gamma\subset C$ be a connected subcurve.
    \begin{itemize}
        \item we say that $\Gamma$ is an $A_{2k+1}$-attached hyperelliptic tail if $\Gamma$ is a hyperelliptic curve of positive genus, $P_{\Gamma}=\emptyset$, and $I_{\Gamma}$ consists of a single $A_{2k+1}$-singularity on $C$ which corresponds to a smooth Weierstrass point on $\Gamma$ (i.e., a point fixed by the hyperelliptic involution of \(\Gamma\)); we say that it is $A_0$-attached if $I_{\Gamma}=\emptyset$ (thus $\Gamma=C$) and $P_{\Gamma}$ consists of a single point (thus $n=1$);
        \item  we say that $\Gamma$ is an $A_{2h+1}/A_{2k+1}$-attached hyperelliptic bridge if $\Gamma$ is a hyperelliptic curve of positive genus, $P_{\Gamma}=\emptyset$, and $I_{\Gamma}$ consists of an $A_{2h+1}$ singularity and an $A_{2k+1}$ singularity on $C$ which correspond to two smooth points on $\Gamma$ exchanged by the involution.
        \item we say that $\Gamma$ is an $A_{2k+1}$-attached hyperelliptic dangling bridge if $\Gamma$ is a hyperelliptic curve of positive genus, $P_{\Gamma}=\emptyset$, and $I_{\Gamma}$ consists of a single $A_{2k+1}$-singularity on $C$ which corresponds to a smooth non-Weierstrass point on $\Gamma$; we say that it is $A_0$-attached if $I_{\Gamma}=\emptyset$ (thus $\Gamma=C$) and $P_{\Gamma}$ consists of a single point (thus $n=1$);
    \end{itemize}
If $\Gamma$ is honestly hyperelliptic, we say that it is an $A_{2k+1}$-attached honestly hyperelliptic tail. The same notation applies to bridges.
\end{definition}
\begin{definition}[(Attached) Hyperelliptic Chain]\label{def:attached-hyp-chain}\index{(Attached) Hyperelliptic Chain}
    Let $k$ be an algebraically closed field over $\kappa$. We say that $(\Gamma,r_1,r_2)$ over $k$ is a \emph{$2$-pointed hyperelliptic chain of genus bounded by $g$ of length $\ell$} if there exists a surjective gluing morphism
    \[
    \gamma \colon \coprod_{i=1}^{\ell} (\Gamma_i,q_{2i-1},q_{2i}) \longrightarrow (\Gamma,r_1,r_2)
    \]
    and a sequence of positive integers $k_1, \dots, k_{\ell-1}$ satisfying:
    \begin{enumerate}
      \item $\Gamma_i$ is a hyperelliptic curve of genus $g_i$ with $1 \leq g_i < g$, for $i=1,\dots,\ell$;
      \item the marked points $q_{2i-1},q_{2i}$ are exchanged by the hyperelliptic involution of $\Gamma_i$, for $i=1,\dots,\ell$;
      \item $\gamma(q_{2i})=\gamma(q_{2i+1})$ is an $A_{2k_i+1}$-singularity of $\Gamma$ for $i=1,\dots,\ell-1$;
      \item $\gamma(q_1)=r_1$ and $\gamma(q_{2\ell})=r_2$.
    \end{enumerate}
    If the curves $\Gamma_i$ are honestly hyperelliptic, we say that $(\Gamma,p_1,p_2)$ is a chain of honestly hyperelliptic curves. We define $1$-pointed and $0$-pointed chains of honestly hyperelliptic curves in the obvious way.
    We say that $(C,\{p_i\}_{i=1}^n)$ has an \emph{$A_{d_1}/A_{d_2}$-attached hyperelliptic chain of genus bounded by $g$ of length $\ell$} if there exists a gluing morphism
    \[
    \gamma \colon (\Gamma,r_1,r_2) \longrightarrow (C,\{p_i\}_{i=1}^n)
    \]
    such that
    \begin{enumerate}
      \item $(\Gamma,r_1,r_2)$ is a $2$-pointed hyperelliptic chain of genus bounded by $g$ of length $\ell$;
      \item for $j=1,2$, $\gamma(r_j)$ is an $A_{d_j}$-singularity of $C$, or if $d_j=0$ we allow $\gamma(r_j)$ to be a marked point of $(C,\{p_i\}_{i=1}^n)$.
    \end{enumerate}
    We define an $A_d$-attached hyperelliptic chain in the obvious way. We say that an attached hyperelliptic chain $(\Gamma,r_1,r_2)$ is separating if the curve $C - \gamma(\Gamma)$ is disconnected; otherwise, we say it is non-separating.
\end{definition}

We are ready to describe how isotrivial degenerations change the combinatorics of the components that contribute to the dimension of the automorphism group. We start with the case of $2$-pointed rosaries of hyperelliptic $A$-singularities, as it is the most frequent case we need to deal with in the upcoming papers. Recall that the atomic case was already treated earlier.

\begin{proposition}\label{prop:defor-rational-chain}
    Let $(C_0,p_1^0,p_2^0)$ be a $2$-pointed rosary of hyperelliptic (odd) $A$-singularities and $(C_1,p_1^1,p_2^1)$ be a $2$-pointed $A_r$-stable curve. If $(C_1,p_1^1,p_2^1)$ admits an isotrivial degeneration to the curve $(C_0,p_1^0,p_2^0)$, then it is one of the curves described below, depending on the parity of the length of the rational chain. Suppose that the length of the rosary is even; then $(C_1,p_1^1,p_2^1)$ is either
    \begin{itemize}
        \item[e1)] a $2$-pointed chain of honestly hyperelliptic curves;
        \item[e2)] an $A_{2h+1}/A_{2k+1}$-attached chain of honestly hyperelliptic curves, which is attached to two rational $1$-pointed smooth curves;
    \end{itemize}
    if instead the length is odd, we have that
    \begin{itemize}
        \item[o)] $(C_1,p_1^1,p_2^1)$ is a $2$-pointed chain of honestly hyperelliptic curves attached to a $2$-pointed rational smooth curve with an odd $A$-singularity.
    \end{itemize}
\end{proposition}

\begin{proof}
    It follows directly from \Cref{prop:def-hyp-rat-sing} and \Cref{rem:alternating-deformations}. Indeed, suppose the length $s+1$ of the rosary is odd, i.e., $s=2t$. Then we can decompose the deformation space of $(C_0,p_1^0,p_2^0)$ as follows:
    $$ \Def_{(C_0,p_1^0,p_2^0)}\simeq \bigoplus_{i=1}^s (T_{q_i}^1\oplus {\rm Cr}_{q_i})$$
    where $q_i$ is the $i$-th (odd) $A$-singularity. The same computations carried out in \Cref{ex:deformation-alter} give us that
    $$ (\Def_{(C_0,p_1^0,p_2^0)})^>=\bigoplus_{i=1}^{t} (T_{q_{2i-1}}^1\oplus {\rm Cr}_{q_{2i}})$$
    and
    $$ (\Def_{(C_0,p_1^0,p_2^0)})^<=\bigoplus_{i=1}^{t} (T_{q_{2i}}^1\oplus {\rm Cr}_{q_{2i-1}})$$
    or vice versa. \Cref{lem:Theta-and-BGmR} implies that any isotrivial degeneration to $(C_0,p_1^0,p_2^0)$ will be contained in one of the two spaces above. The statement follows from \Cref{prop:def-hyp-rat-sing}. The computations for the even case are identical.
\end{proof}

\begin{remark}
    Notice that a similar result holds for a $1$-pointed (or $0$-pointed) rosary of hyperelliptic (odd) $A$-singularities. We leave it to the reader to adapt the previous statement to these cases.
\end{remark}
Finally, we note that the same statement works for closed rosaries.

\begin{definition}
    Let $C$ be an $A$-prestable curve over an algebraically closed field $k$. We say that $C$ is a \emph{closed chain of hyperelliptic curves} of genus bounded by $g$ if there exists an $A_{2k+1}$-singularity with $k\geq 1$ such that the (pointed) normalization $(\widetilde{C},p_1,p_2)$ is a $2$-pointed chain of hyperelliptic curves of genus bounded by $g$.
\end{definition}
We leave the proof of the next proposition to the interested reader. It directly follows from \Cref{prop:defor-rational-chain}.

\begin{proposition}\label{prop:iso-deg-rosary}
    Let $(C_0,p_1^0,p_2^0)$ be a closed rosary of hyperelliptic $A$-singularities (of even length) and $(C_1,p_1^1,p_2^1)$ be an $A_r$-stable curve which admits an isotrivial degeneration to $C_0$. Then $(C_1,p_1^1,p_2^1)$ is a closed chain of hyperelliptic curves.
\end{proposition}

\subsection{Isotrivial degenerations}\label{sub:describing-attractors}

We end the section by globalizing some of the results appearing in \Cref{sub:def-theory}. To be more specific, in the previous section we described all the possible isotrivial degenerations to an atom and, more generally, to a rosary, giving necessary conditions for a curve to belong to a basin of attraction of \(\cM_{g,n}^r\). In this section, we are going to show that the necessary conditions are actually sufficient, completely classifying the possible isotrivial degenerations for $A_r$-stable curves.

We start with the atomic case. For the sake of simplicity, we deal with the case $n=0$, but the same proof holds also in the more general setting. Specifically, the next lemma shows us that every curve with an even $A$-singularity has an isotrivial degeneration to the one glued nodally to the respective even atom.

\begin{lemma}\label{lem:even-atom-attraction}
    Let $(C,q) \in \cA_{2h}$ be an $A_r$-stable curve of genus $g$ with an $A_{2h}$-singularity $q$ over a field and denote by $(\widetilde{C},\widetilde{q})$ the stabilization of the (pointed) partial normalization of $C$ at $q$. Moreover, let $C_0$ be the $A_r$-stable curve obtained by gluing nodally $(\widetilde{C},\widetilde{q})$ with the $1$-pointed even atom $(\Gamma,\widetilde{q})$ of genus $h$ at $\widetilde{q}$. Then there exists an isotrivial degeneration $C_{\Theta}\rightarrow \Theta$ such that
    \begin{itemize}
        \item the generic fiber of the degeneration is isomorphic to $C$;
        \item the special fiber of the degeneration is isomorphic to $C_0$.
    \end{itemize}
\end{lemma}

\begin{proof}
    Consider the stabilization morphism (see \Cref{def:stabilization})
    $$ \cM_{g-h,[h]}^r \simeq \cM_{g-h,1}^r \times [\bA^1/\bG_m] \longrightarrow \cM_{g-h,1}^r;$$
    the fiber over $[(\widetilde{C},\widetilde{q})] \in \cM_{g-h,1}^r$ is isomorphic to $[\bA^1/\bG_m]$. More specifically, the generic point of $[\bA^1/\bG_m]$ corresponds to the case when the (pointed) partial normalization of $C$ at $q$ is already stable, whereas the special point can be interpreted as the curve obtained by gluing a projective line $(\bP^1,0,\infty)$ with $(\widetilde{C},\widetilde{q})$ using the identification $0\equiv\widetilde{q}$. The group $\bG_m$ corresponds to the automorphisms of $(\bP^1,0,\infty)$.

    In the notation of \Cref{rem:crim-mor-even}, we have that the moduli stack $\cA_{2h}(\widetilde{C},\widetilde{q})$ parametrizing curves in $\cA_{2h}$ whose stabilization of the partial normalization at the even $A$-singularity is $(\widetilde{C},\widetilde{q})$ is isomorphic to $[V_{h,2}\times \bA^1/\bG_m]$. By definition, $(C,q)$ belongs to $\cA_{2h}(\widetilde{C},\widetilde{q})$. Unraveling the constructions, it is easy to see that $V_{h,2}\times \bA^1$ is a $\bG_m$-representation with strictly positive (or negative) weights; thus, every curve in $\cA_{2h}(\widetilde{C},\widetilde{q})$ admits an isotrivial degeneration to the $0$-section of $V_{h,2}\times \bA^1$, which is represented by the curve $C_0$.
\end{proof}

We proceed by doing the same for the other possible basins of attraction of the even atom, namely curves with hyperelliptic tails.

\begin{lemma}\label{lem:iso-deg-hyp-tails-atom-1}
    Suppose $r$ is an even positive integer. Let $C$ be an $A_r$-stable curve of genus $g$ over some field $k$ and $\Gamma\subset C$ be an $A_1$-attached honestly hyperelliptic tail of genus $h:=r/2$. We denote by $\widetilde{C}$ the subcurve $C-\Gamma$ and by $\widetilde{q}$ the intersection $\Gamma\cap (C-\Gamma)$. Moreover, let $C_0$ be the $A_{r}$-stable curve obtained by gluing nodally $(\widetilde{C},\widetilde{q})$ with the $1$-pointed even atom $(\Gamma,\widetilde{q})$ of genus $h$ at $\widetilde{q}$. Then there exists an isotrivial degeneration $C_{\Theta}\rightarrow \Theta$ such that
    \begin{itemize}
        \item the generic fiber of the degeneration is isomorphic to $C$;
        \item the special fiber of the degeneration is isomorphic to $C_0$.
    \end{itemize}
\end{lemma}

\begin{proof}
    Notice that there is a morphism
    $$\cH_{h,w}^r \times \spec k \rightarrow \cM_{g,n}^r \times \spec k$$
    that sends every curve $(H,\widetilde{q})\in \cH_{h,w}^r \times \spec k$ to the curve $H\cup_{\widetilde{q}}\widetilde{C}$. The result follows using \Cref{prop:hyp-global-descr}.
\end{proof}

The exact same proof gives us the following lemma, which takes care of the case of dangling bridges.
\begin{lemma}\label{lem:iso-deg-hyp-tails-atom-2}
    Suppose $r$ is an odd positive integer such that $r\geq 5$. Let $C$ be an $A_r$-stable curve of genus $g$ over some field $k$ and $\Gamma\subset C$ be an $A_1$-attached honestly hyperelliptic dangling bridge of genus $h:=(r-1)/2$. We denote by $\widetilde{C}$ the subcurve $C-\Gamma$ and by $\widetilde{q}$ the intersection $\Gamma\cap (C-\Gamma)$. Moreover, let $C_0$ be the $A_{r}$-stable curve obtained by gluing nodally $(\widetilde{C},\widetilde{q})$ with the $1$-pointed odd atom $(\Gamma,\widetilde{q})$ of genus $h$ at $\widetilde{q}$. Then there exists an isotrivial degeneration $C_{\Theta}\rightarrow \Theta$ such that
    \begin{itemize}
        \item the generic fiber of the degeneration is isomorphic to $C$;
        \item the special fiber of the degeneration is isomorphic to $C_0$.
    \end{itemize}
\end{lemma}

Again, the case of $2$-pointed bridges can be proved exactly as in \Cref{lem:even-atom-attraction}. Thus, we omit the proof of the next statement.

\begin{lemma}\label{lem:odd-atom-attraction}
    Let $(C,q) \in \cA_{2h+1}$ be an $A_r$-stable curve of genus $g$ with an $A_{2h+1}$-singularity $q$ over a field and denote by $(\widetilde{C},\widetilde{q}_1,\widetilde{q}_2)$ the stabilization of the pointed partial normalization of $C$ at $q$. Moreover, let $C_0$ be the $A_r$-stable curve obtained by gluing nodally $(\widetilde{C},\widetilde{q}_1,\widetilde{q}_2)$ with the $2$-pointed odd atom $(\Gamma,\widetilde{q}_1,\widetilde{q}_2)$ of genus $h$ at $\widetilde{q}_1,\widetilde{q}_2$. Then there exists an isotrivial degeneration $C_{\Theta}\rightarrow \Theta$ such that
    \begin{itemize}
        \item the generic fiber of the degeneration is isomorphic to $C$;
        \item the special fiber of the degeneration is isomorphic to $C_0$.
    \end{itemize}
\end{lemma}

We also describe one of the basins of attraction for an $A_1/A_{2h}$-attached rosary of hyperelliptic $A$-singularities (see the leftmost curve in \Cref{fig:example-3.25}). We do it for chains of length less than or equal to two, but the case of a longer rosary can be proven in the same way.

\begin{lemma}\label{lem:iso-deg-hyp-tails}
    Suppose $r$ is an even positive integer and $k\geq 1$. Let $C$ be an $A_r$-stable curve of genus $g$ over an algebraically closed field and $\Gamma\subset C$ be an $A_{2k+1}$-attached honestly hyperelliptic tail of genus $h:=r/2$ for $h\geq 1$. We denote by $\widetilde{C}$ the stabilization of the subcurve $C-\Gamma$ and by $\widetilde{q}$ the intersection $\Gamma\cap (C-\Gamma)$. Moreover, let $C_0$ be the $A_{r}$-stable curve obtained by gluing $(\widetilde{C},\widetilde{q})$ nodally with the $1$-pointed curve $(D,\widetilde{q})$ of genus $h+k$ obtained by pinching the $2$-pointed odd atom $(\widetilde{D}, \widetilde{q},q_2)$ of genus $k$ at $q_2$ such that $q_2$ is a hyperelliptic singularity (see \Cref{fig:example-3.25}).

    Then there exists an isotrivial degeneration $C_{\Theta}\rightarrow \Theta$ such that
    \begin{itemize}
        \item the generic fiber of the degeneration is isomorphic to $C$;
        \item the special fiber of the degeneration is isomorphic to $C_0$.
    \end{itemize}
\end{lemma}

\begin{proof}
    Consider the following diagram
    $$\begin{tikzcd}
        \cP \arrow[d] \arrow[r] &\cA_{2k+1}^h \arrow[d] \arrow[r]   &  \cM_{g,n}^r  \\
        \left(\cB\Aut(\widetilde{C},\widetilde{q}) \times [\bA^1/\bG_m]\right)\times \cH_{h,w}^{r,\circ} \arrow[r] & \left({\cM_{g-k-1,1}^r}\times [\bA^1/\bG_m]\right)\times\cM_{h,[k]}^r
    \end{tikzcd}$$
    where the left-hand square is cartesian and the morphism $\cB\Aut(\widetilde{C},\widetilde{q}) \rightarrow \cM_{g-k-1,1}^r$ is induced by $(\widetilde{C},\widetilde{q})$. Notice that $\cH_{h,w}^{r,\circ}$ is a locally closed substack of $\cM_{h,[k]}^r\simeq \cM_{h,1}^r\times [\bA^1/\bG_m]$.

    By construction, $\cP$ parametrizes curves over $k$ constructed by attaching an honestly hyperelliptic tail of genus $h$ to $\widetilde{C}$ with an $A_{2k+1}$-singularity.
    The factor $[\bA^1/\bG_m]$ accounts for the possibility of gluing the hyperelliptic tail directly at $\widetilde{q}$ or to an unstable projective line passing through $\widetilde{q}$, and it can be identified with the (stacky) deformation space of the node at $\widetilde{q}$ when the unstable $\bP^1$ appears. For a more detailed explanation, see \Cref{rem:sprout-proj-line}. Since $r=2h$, \Cref{prop:hyp-global-descr} gives that $\cH_{h,w}^{r,\circ}$ is the global quotient $[\bA^{2h-1}/\bG_m]$ and it can be identified with the deformation space of the (hyperelliptic) $A_{2h}$-singularity, by \Cref{prop:def-hyp-rat-sing}. Thus, \Cref{rem:descr-an-disp-i} implies that $\cP$ is the quotient stack $[\bA^1\times \bA^{2h-1} \times C_{k+1}/\bG_m]$, where $C_{k+1}:=(U_{k+1}/\Ga)$ is a $\bG_m$-representation of dimension $k$ which accounts for the crimping data of the $A_{2k+1}$-singularity (see also \Cref{rem:odd-atoms}). The zero of the representation $\bA^1\times \bA^{2h-1} \times C_{k+1}$ corresponds to the curve $(D,\widetilde{q})$ described in the statement of the lemma. Thus, the statement follows if we show that $\bG_m$ acts with strictly positive (or negative) weights on the whole representation. Notice that the three representations $\bA^1$, $\bA^{2h-1}$, and $C_{k+1}$ either have strictly positive or strictly negative weights.

    By \Cref{cor:gm-decom-tang}, it is clear that the weights of $\bA^{2h-1}$ and $C_{k+1}$ have the same signs: namely, if we set the convention that $\bG_m$ acts positively on the $A_{2h}$-singularity, then they both have strictly positive weights. Since $\bA^1$ can be identified with the deformation of the node created by gluing $(\widetilde{C},\widetilde{q})$ and $(D,\widetilde{q})$, the weight of the $\bG_m$-representation $\bA^1$ has the same sign as the weights of the other two $\bG_m$-representations (see, for instance, \Cref{ex:deformation-alter} or \Cref{rem:alternating-deformations}), finally proving the result.
\end{proof}

Finally, we describe the basins of attraction in the case of a $2$-pointed rosary of hyperelliptic $A$-singularities. This is the inverse of \Cref{prop:defor-rational-chain}. We avoid writing down the proof as it is not particularly enlightening, but it follows by applying the same strategy adopted for \Cref{lem:iso-deg-hyp-tails}.

\begin{proposition}\label{prop:iso-deg-chain}
    Let  $(C_1,p_1^1,p_2^1)$ be a $2$-pointed $A_r$-stable curve. Then $(C_1,p_1^1,p_2^1)$ admits an isotrivial degeneration to a $2$-pointed rosary of hyperelliptic $A$-singularities if it is one of the curves described below, depending on the parity of the length of the rosary. Suppose that the length of the rosary is even; then $(C_1,p_1^1,p_2^1)$ is either
    \begin{itemize}
        \item[e1)] a $2$-pointed chain of honestly hyperelliptic curves;
        \item[e2)] an $A_{2h+1}/A_{2k+1}$-attached chain of honestly hyperelliptic curves, which is attached to two rational $1$-pointed smooth curves (see the picture below);
    \end{itemize}
    if instead the length is odd, we have that
    \begin{itemize}
        \item[o)] $(C_1,p_1^1,p_2^1)$ is a $2$-pointed chain of honestly hyperelliptic curves attached by an odd $A$-singularity to a $2$-pointed rational smooth curve (see the picture below).
    \end{itemize}
    Moreover, if $(C_1,p_1^1,p_2^1)$ is a $2$-pointed chain of honestly hyperelliptic curves of length $1$, then the rosary has length $2$, thus it is a $2$-pointed odd atom.
\end{proposition}

\begin{figure}[H]
        \caption{}
        \centering
        \includegraphics[width=0.6\textwidth]{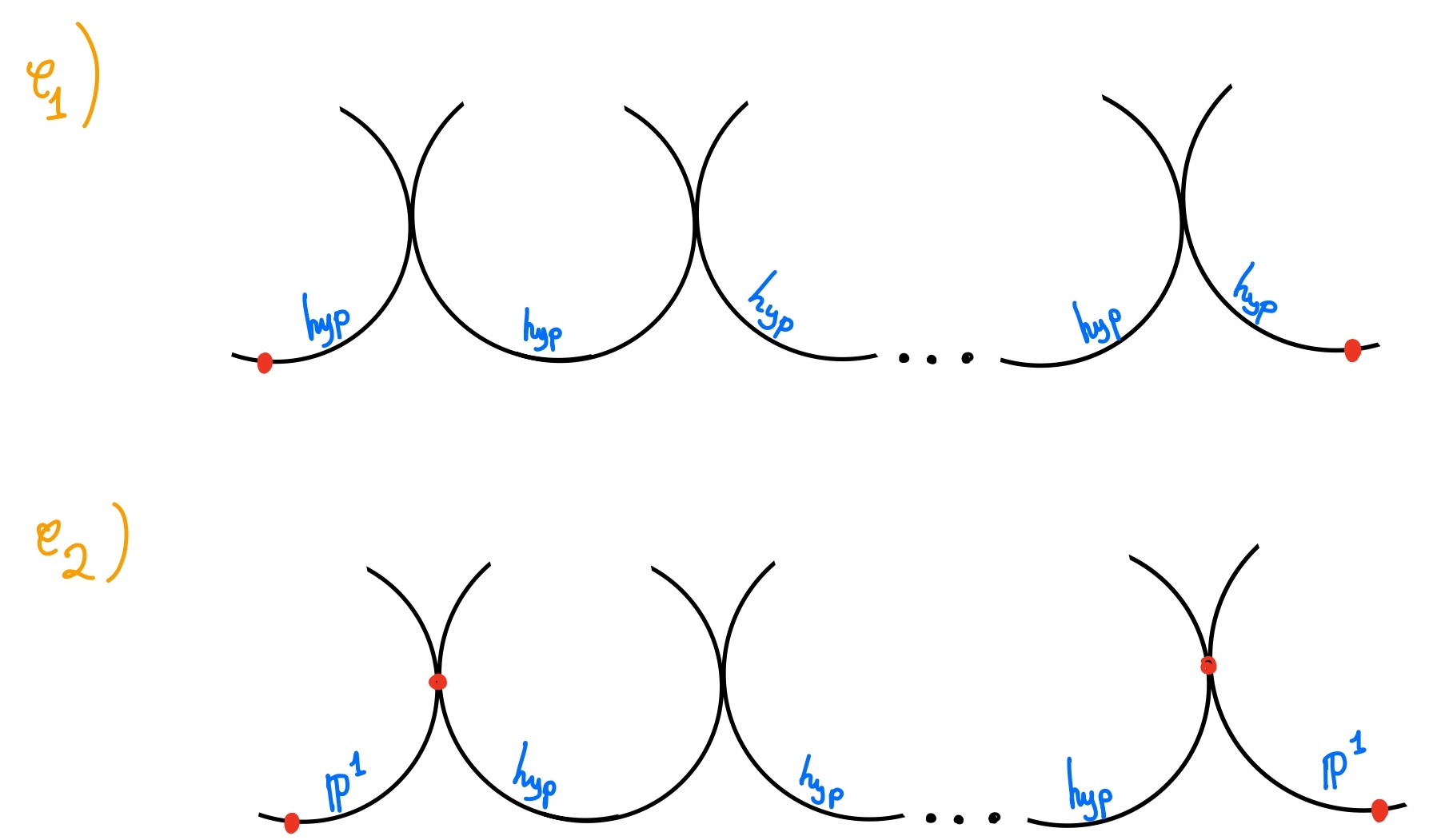}
    \end{figure}
\begin{figure}[H]
        \caption{}
        \centering
        \includegraphics[width=0.6\textwidth]{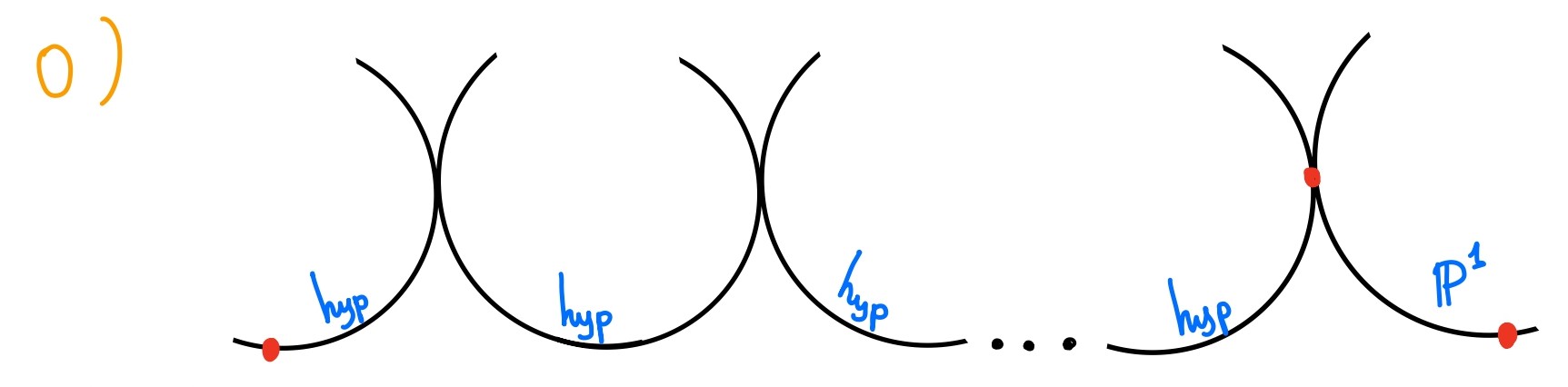}
    \end{figure}

\begin{remark}
    Notice that a rosary of hyperelliptic $A$-singularities of even length is a chain of hyperelliptic curves.
\end{remark}

We now introduce the notion of special curves, which we will prove are exactly the points of the stack $\cM_{g,n}^r$ without non-degenerate isotrivial degenerations: in other words, any family over \(\Theta_k\) with a special curve as the general fiber has a special fiber isomorphic to the same special curve. We will use this to characterize closed points of \(\cM_{g,n}^r\) below.

\begin{definition}\label{def:special-curve}
    We say that an $n$-pointed $A_r$-stable curve of genus $g$ over an algebraically closed field is \emph{special} if the following conditions hold:
    \begin{enumerate}
        \item every \(A_h\)-singularity, with \(h\geq 2\), is \emph{atomic}, that is, it lies in an \(A_i\)- or \(A_i/A_j\)-attached atom where \(i,j\in \{0,1\}\),
        \item every honestly hyperelliptic tail of genus \(\leq r/2\) and every honestly hyperelliptic bridge of genus \(\leq(r-1)/2\) is an atom (including dangling bridges).
        \item if \(n=0\), \(r\geq 2g\), and \(C\) is an honestly hyperelliptic curve, then $C$ is an atom, which is even if \(r=2g\) and odd otherwise.
    \end{enumerate}
\end{definition}
\begin{lemma}\label{lem:sing-preserved}
Let $k$ be an algebraically closed field, and let \(C \to \Theta_k\) be an isotrivial degeneration from a special curve \(C_1\) to a special curve \(C_0\). Let $q_0 \in C_0$ be an atomic $A$-singularity that lies on an atom $\Gamma \subset C_0$. If $q_0$ extends to an equisingular section of the family, then the same is true for any point in $\Gamma \cap {(C_0 - \Gamma)}$. Moreover, $\Gamma \cap {(C_0-\Gamma)}$ consists only of nodal singularities.
\end{lemma}
\begin{proof}
Let $q_1 \in C_1$ be the point corresponding to the equisingular section passing through $q_0$. By condition~(1) in \Cref{def:special-curve}, there exists an attached atom $(D, r_i) \to C_1$ (where the atom has either one or two markings) such that the points $r_i$ are either lonely nodes or marked points. By properness of the family, these extend to sections $\overline{r}_i$ of the total family $C$, which are either nodal sections or markings by \Cref{prop:alp-fund}. Restricting to $C_0$, they yield precisely the points in $\Gamma \cap (C_0 - \Gamma)$ (since $C_0$ is stable), which therefore are nodal singularities.
\end{proof}

\begin{lemma}\label{lem:iso-triv-special-degen-keeps-node}
Let $k$ be an algebraically closed field, and let \(C \to \Theta_k\) be an isotrivial degeneration from a special nodal stable curve \(C_1\) to a special $A_r$-stable curve \(C_0\). Then \(C_1 \simeq C_0\).
\end{lemma}

\begin{proof}
We first reduce to the case where $C$ does not admit any nodal equisingular sections. Indeed, if such sections exist, we may perform the partial normalization along them; it then suffices to prove the statement on each of the resulting connected components.

It is enough to show that $C_0$ does not have worse-than-nodal singularities, i.e., that it contains no atoms. Suppose, for contradiction, that $\Gamma \subset C_0$ is an atom, and let $p_0 \in \Gamma \cap (C_0 - \Gamma)$ be a node. Assume that $p_0$ is the intersection of two nodally attached atoms, which we denote by $\Gamma$ and $\Gamma'$ (by abuse of notation). Since $p_0$ deforms nontrivially, we have $\deg_{\mathbb{G}_m} T^1_{p_0} > 0$, which implies that at least one of the two atomic singularities must be preserved (again by \Cref{rem:alternating-deformations}). Applying \Cref{lem:sing-preserved}, we obtain a contradiction. Hence, distinct atoms cannot intersect.

It follows that $C_0$ is itself a (pointed) atom. Indeed, since the generic fiber is nodal, every worse-than-nodal singularity must deform. Therefore, \Cref{rem:alternating-deformations} implies that any nodes in $\Gamma \cap (C_0 - \Gamma)$ do not deform along the degeneration. However, by our reduction step, there are no equisingular nodal sections, hence $C_0 = \Gamma$.

Finally, we treat the case where $C_0$ is a pointed atom. This follows from \Cref{prop:def-atom} and \Cref{rem:non-weierstrass-case}, together with the assumption that $C_1$ is special.
\end{proof}

\begin{theorem}
\label{thm:closed_points_Mgnr}
Let \(k\) be an algebraically closed field and let \(C \in \cM_{g,n}^r(k)\). If \(C\) is a closed point of \(\cM_{g,n}^r \times_{\kappa} \Spec k\), then \(C\) is special. Moreover, if either \(n \geq 1\) or \(r \leq 2g\), the converse also holds.
\end{theorem}

\begin{proof}
We first show that the conditions in \Cref{def:special-curve} are necessary. In particular, we show that if they are not satisfied, then there exist nontrivial isotrivial degenerations from \(C\) to another curve defined over the same base, and therefore the point is not closed.
Suppose \(C\) has a non-atomic \(A_h\)-singularity \(q\), with \(h \geq 2\). Then there exists a non-degenerate isotrivial degeneration from \(C\) to a curve \(C_0\), where \(C_0\) is constructed as in \Cref{lem:even-atom-attraction} or \Cref{lem:odd-atom-attraction}, according to whether \(h\) is even or odd. Similarly, suppose \(\Gamma \subset C\) is an \(A_1\)- or \(A_0\)-attached honestly hyperelliptic tail of genus \(\leq r/2\) which is not an atom. Then, by \Cref{lem:iso-deg-hyp-tails}, there exists a non-degenerate isotrivial degeneration from \(C\) to a curve \(C_0\). If \(\Gamma \subset C\) is an \(A_i/A_j\)-attached honestly hyperelliptic bridge of genus \(\leq (r-1)/2\), where \(i,j \in \{0,1\}\), then the pointed partial normalization of \(C\) at the attaching nodes \(p_1, p_2\) has a connected component \((\Gamma, p_1, p_2)\). By \Cref{prop:iso-deg-chain}, this component admits a non-degenerate isotrivial degeneration to an atom. Gluing back nodally, we obtain a nontrivial isotrivial degeneration of \(C\). If \(n=0\) and \(C\) is honestly hyperelliptic with \(r \geq 2g+1\), then one can use the quotient stack description in terms of \(2\!:\!1\) covers of \(\mathbb{P}^1\) to construct an isotrivial degeneration from \(C\) to the odd atom (see, for instance, the proof of \Cref{theo:classify-opens-hyp}). Using the same description, when \(r=2g\) and \(n=0\), any honestly hyperelliptic curve can be isotrivially degenerated to the even atom.

For the converse, it suffices to show that there are no non-degenerate isotrivial degenerations from a special curve \(C\) to a closed point \(C_0\) in \(\cM_{g,n}^r \times_{\kappa} \Spec k\). Indeed, this implies that \(C\) corresponds to a closed point, since every degeneration of \(k\)-points to a closed point can be realized as an isotrivial degeneration by \cite[Lemma~3.24]{ExistenceOfModuli} in \(\cM_{g,n}^r\) (which is locally reductive if \(r \leq 2g\) or \(n \geq 1\)).

Suppose we are given an isotrivial degeneration of \(C\) with special fiber \(C_0\), which is closed. By the first part, \(C_0\) is special.

We first reduce to the case where \(C\) has no \(A_k\)-singularities for \(k \geq 2\). Since \(C\) is special, any \(A_{2k+1}\)-singularity on \(C\) is lonely and lies on an atom \(\Gamma \subset C\). By \Cref{prop:alp-fund}, it extends to an equisingular section of the family, and by \Cref{lem:sing-preserved}, its limit lies on an atom of \(C_0\), whose nodal attaching points extend to nodal sections of the family. After normalizing along these equisingular nodal sections, the claim reduces to the connected components that are not atomic. Indeed, if $n\geq 1$, atoms can only degenerate to themselves, while if $n=0$ and $r\geq 2g+1$, condition (3) of \Cref{def:special-curve} deals with the atomic case. Hence, we may assume that \(C\) has no \(A_{2k+1}\)-singularities. The even case is similar: any \(A_{2h}\)-singularity of \(C\) lies on a nodally attached tail. Applying \Cref{prop:alp-fund} to the attaching nodes, we reduce to the connected components of the normalization along the corresponding nodal equisingular sections. Finally, if \(C\) has no \(A_k\)-singularities for \(k > 1\), the claim follows from \Cref{lem:iso-triv-special-degen-keeps-node}.
\end{proof}

We do not know a priori if the converse in \Cref{thm:closed_points_Mgnr} holds also in the case $n=0$ and $r=2g+1$, due to the fact that for stacks with non-reductive stabilizers (at closed points) it is not always true that degenerations of $k$-points can be described using degenerations over $\Theta_k$: see for instance \([\bP^1/\bG_a]\) where the action is by translation.
\begin{remark}
     For \(r\geq 2g+1\), the stack $\cM^r_{g} \setminus \cM_g^{2g}$ is equal to $\cA_{2g+1}$, which is isomorphic to
     $$[\bA^{g-1}/(\Gm \ltimes \Ga)].$$ The ring of invariants is actually the base field, but there exists a unique closed point which has a non-reductive automorphism group.
\end{remark}

\begin{corollary}\label{cor:closed-smooth-curves}
    If \(C\in \cM_{g,n}^r(k)\) is a smooth curve, then the corresponding point in \(\cM_{g,n}^r\times_{\kappa}\Spec k\) is closed unless \(r\geq 2g\), \(n\leq 2\), and \(C\) is hyperelliptic. Moreover, if \(r=2g\) and \(C\) is not closed, then \(n\leq 1\).
\end{corollary}
\begin{proof}
    This follows from \Cref{thm:closed_points_Mgnr} unwinding the conditions in \Cref{def:special-curve} for a smooth curve.
\end{proof}
The following is a direct consequence of \Cref{cor:closed-smooth-curves}.
\begin{corollary}\label{cor:stable-gms}
    If either \(r<2g\) or $n>2$, the open embedding \(\cM_{g,n}\subset \cM_{g,n}^r\) is $\Theta$-surjective. In particular, if $\cM_{g,n} \subset \cU \subset \cM_{g,n}^r$ is a sequence of open embeddings such that \(\cU\) admits a good moduli space \(\operatorname{U}\), then \(\cM_{g,n}\subset \cU\) is saturated with respect to \(\pi\), and the induced map \(\operatorname{M}_{g,n}\to \operatorname{U}\) is an open embedding.
\end{corollary}

\bibliographystyle{alpha}
\bibliography{r_leq_5/References.bib}

@article{StructureOfInstability,
  title={On the structure of instability in moduli theory},
  author={Daniel Halpern-Leistner},
  journal={arXiv: Algebraic Geometry},
  year={2014},
  url={https://api.semanticscholar.org/CorpusID:117058823}
}

@article{LunaetaleSliceStacks,
author = {Jarod Alper and Jack Hall and David Rydh},
title = {{A Luna étale slice theorem for algebraic stacks}},
volume = {191},
journal = {Annals of Mathematics},
number = {3},
publisher = {Department of Mathematics of Princeton University},
pages = {675 -- 738},
year = {2020},
doi = {10.4007/annals.2020.191.3.1},
URL = {https://doi.org/10.4007/annals.2020.191.3.1}
}

@phdthesis{VanDerWyck,
    author={van der Wyck, Frederick D. W.},
    year={2010},
    title={Moduli of singular curves and crimping},
    journal={ProQuest Dissertations and Theses},
    pages={144},
    isbn={978-1-124-09171-6},
    school={Harvard University},
    url={https://www.proquest.com/dissertations-theses/moduli-singular-curves-crimping/docview/613377231/se-2},
}

@article{Viviani,
    title={ON THE FIRST STEPS OF THE MINIMAL MODEL PROGRAM FOR THE MODULI SPACE OF STABLE POINTED CURVES},
    DOI={10.1017/S1474748021000116},
    journal={Journal of the Institute of Mathematics of Jussieu},
    publisher={Cambridge University Press},
    author={Codogni, Giulio and Tasin, Luca and Viviani, Filippo},
    year={2023},
    pages={145–211}
}

@article{hassett2013log,
  title={Log minimal model program for the moduli space of stable curves: the first flip},
  author={Hassett, Brendan and Hyeon, Donghoon},
  journal={Annals of Mathematics},
  pages={911--968},
  year={2013},
  publisher={JSTOR}
}

@misc{stacks-project,
  author       = {The {Stacks project authors}},
  title        = {The Stacks project},
  howpublished = {\url{https://stacks.math.columbia.edu}},
  year         = {2023},
}

@article{AlpFedSmyWyck,

title={Second flip in the Hassett-Keel program: A local description }, 
volume={153},  
journal={Compositio Mathematica}, 
publisher={London Mathematical Society}, 
author={Alper, J. and Fedorchuck, M. and Smyth, D.I. and van der Wyck, F.}, 
year={2017}, 
pages={1547-1583}}

@article {ExistenceOfModuli,
    AUTHOR = {Alper, Jarod and Halpern-Leistner, Daniel and Heinloth,
              Jochen},
     TITLE = {Existence of moduli spaces for algebraic stacks},
   JOURNAL = {Invent. Math.},
  FJOURNAL = {Inventiones Mathematicae},
    VOLUME = {234},
      year = {2023},
    NUMBER = {3},
     PAGES = {949--1038},
      ISSN = {0020-9910,1432-1297},
   MRCLASS = {14D23 (14A20)},
  MRNUMBER = {4665776},
       DOI = {10.1007/s00222-023-01214-4},
       URL = {https://doi.org/10.1007/s00222-023-01214-4}}

@Article{Per1,
  author  = {Pernice, M.},
  title   = {{The moduli stack of $A_r$-stable curves}},
  journal = {Preprint on arXiv},
  year    = {2023},
  url     = {https://arxiv.org/abs/2302.10877},
}

@Article{AlpHalRyd,
      title={The \'etale local structure of algebraic stacks}, 
      author={Alper, J. and Hall, J. and Rydh, D.},
      year={2025},
      eprint={1912.06162},
      archivePrefix={arXiv},
      primaryClass={math.AG},
      url={https://arxiv.org/abs/1912.06162}, 
      journal={arXiv}
}

@Article{Per2,
  author  = {Pernice, M.},
  title   = {{Hyperelliptic $A_r$-stable curves (and their moduli stack)}},
  journal = {Transactions of the American Mathematical Society},
  year    = {2024},
  volume = {377},
  pages = {4133-4169}
}

@article{Cat,
author={Catanese, F.},
title={{Pluricanonical Gorenstein Curves}},
journal={Enumerative Geometry and Classical Algebraic Geometry},
year={1982},
publisher={Birkh{\"a}user Boston},
address={Boston, MA},
pages={51--95},
isbn={978-1-4684-6726-0},
doi={10.1007/978-1-4684-6726-0_4},
url={https://doi.org/10.1007/978-1-4684-6726-0_4}
}

@article{ArVis,
	Author={Arsie, A. and Vistoli, A.},
	Title={{Stacks of cyclic covers of projective spaces}},
	Year={2004},
	Journal={Compos. Math. 140, no.3, 647-666}	,
}

@misc{Vicnote,
    Author = {Victoria Hoskins},
    Title = {{Moduli problems and Geometric Invariant Theory}},
    Year = {2015},
    note = {available at: \url{https://userpage.fu-berlin.de/hoskins/M15_Lecture_notes.pdf}}
}

@misc{AlpSmyVdW,
 author = {Alper, Jarod and Smyth, David Ishii and van der Wyck, Frederick},
 title = {Weakly proper moduli stacks of curves},
 year = {2010},
 howpublished = {Preprint, {arXiv}:1012.0538 [math.{AG}] (2010)},
 url = {https://arxiv.org/abs/1012.0538},
 arXiv = {arXiv:1012.0538}
}

@article{CasmarLaza,
 ISSN = {00029947},
 URL = {http://www.jstor.org/stable/23513483},
 author = {Casalaina Martin, S. and Laza, R.},
 journal = {Transactions of the American Mathematical Society},
 number = {5},
 pages = {2271--2295},
 publisher = {American Mathematical Society},
 title =  {{Simultaneous semi-stable reduction for curves with ADE singularities}},
 volume = {365},
 year = {2013}
}

@article{Fedor,
  title = {{Moduli spaces of hyperelliptic curves with A and D singularities}},
  volume = {276},
  ISSN = {1432-1823},
  url = {http://dx.doi.org/10.1007/s00209-013-1201-6},
  DOI = {10.1007/s00209-013-1201-6},
  number = {1–2},
  journal = {Mathematische Zeitschrift},
  publisher = {Springer Science and Business Media LLC},
  author = {Fedorchuk,  M.},
  year = {2013},
  pages = {299–328}
}

@article{AlpHalHalLeiRyd, 
title={{Artin algebraization for pairs with applications to the local structure of stacks and Ferrand pushouts}},
volume={12}, 
DOI={10.1017/fms.2023.60}, 
journal={Forum of Mathematics, Sigma},
author={Alper, Jarod and Hall, Jack and Halpern-Leistner, Daniel and Rydh, David}, year={2024}
}

@article{DeligneMumford,
 author = {Deligne, Pierre and Mumford, D.},
 title = {The irreducibility of the space of curves of a given genus},
 fjournal = {Publications Math{\'e}matiques},
 journal = {Publ. Math., Inst. Hautes {\'E}tud. Sci.},
 issn = {0073-8301},
 volume = {36},
 pages = {75--109},
 year = {1969},
 language = {English},
 doi = {10.1007/BF02684599},
 url = {https://eudml.org/doc/103899},
 zbMATH = {3289080},
 Zbl = {0181.48803}
}

@article{PseudostableSchubert,
 author = {Schubert, David},
 title = {A new compactification of the moduli space of curves},
 fjournal = {Compositio Mathematica},
 journal = {Compos. Math.},
 issn = {0010-437X},
 volume = {78},
 number = {3},
 pages = {297--313},
 year = {1991},
 language = {English},
 url = {https://eudml.org/doc/90093},
 zbMATH = {9343},
 Zbl = {0735.14022}
}

@article{AlpFedSmyExistence,
 author = {Alper, Jarod and Fedorchuk, Maksym and Smyth, David Ishii},
 title = {Second flip in the {Hassett}-{Keel} program: existence of good moduli spaces},
 fjournal = {Compositio Mathematica},
 journal = {Compos. Math.},
 issn = {0010-437X},
 volume = {153},
 number = {8},
 pages = {1584--1609},
 year = {2017},
 language = {English},
 doi = {10.1112/S0010437X16008289},
 zbMATH = {6764719},
 Zbl = {1403.14038}
}

@article{AlpFedSmyProjectivity,
 author = {Alper, Jarod and Fedorchuk, Maksym and Smyth, David Ishii},
 title = {Second flip in the {Hassett}-{Keel} program: projectivity},
 fjournal = {IMRN. International Mathematics Research Notices},
 journal = {Int. Math. Res. Not.},
 issn = {1073-7928},
 volume = {2017},
 number = {24},
 pages = {7375--7419},
 year = {2017},
 language = {English},
 doi = {10.1093/imrn/rnw216},
 zbMATH = {7004507},
 Zbl = {1405.14063}
}

@article{Alp,
      title={Good moduli spaces for Artin stacks}, 
      author={Alper, J.},
      year={2013},
      journal = {Annales de l'Institut Fourier},
     volume = {63},
    number = {6},
    pages = {2349--2402},
}

@article{CoherentTannaka,
 author = {Hall, Jack and Rydh, David},
 title = {Coherent {Tannaka} duality and algebraicity of {Hom}-stacks},
 fjournal = {Algebra \& Number Theory},
 journal = {Algebra Number Theory},
 issn = {1937-0652},
 volume = {13},
 number = {7},
 pages = {1633--1675},
 year = {2019},
 language = {English},
 doi = {10.2140/ant.2019.13.1633},
 url = {hdl.handle.net/10150/634911},
 zbMATH = {7110518},
 Zbl = {1423.14010}
}

@article{hassett2009log,
  title={Log canonical models for the moduli space of curves: the first divisorial contraction},
  author={Hassett, Brendan and Hyeon, Donghoon},
  journal={Transactions of the American Mathematical Society},
  volume={361},
  number={8},
  pages={4471--4489},
  year={2009}
}

@article{codogni2021some,
  title={On some modular contractions of the moduli space of stable pointed curves},
  author={Codogni, Giulio and Tasin, Luca and Viviani, Filippo},
  journal={Algebra \& Number Theory},
  volume={15},
  number={5},
  pages={1245--1281},
  year={2021},
  publisher={Mathematical Sciences Publishers}
}

@phdthesis{AltCompClustAlg, 
author = {Gori, Davide},
title = {Alternative Compactifications of $\operatorname{M_{g,n}}$ via Cluster Algebras and their Birational Geometry},
year = {2026},
url = {https://iris.uniroma1.it/handle/11573/1759004},
note = {\url{https://iris.uniroma1.it/handle/11573/1759004}},
school = {Universit\`a La Sapienza, Roma},
language = {English}
}
\end{document}